\documentclass[reqno]{amsart}
\usepackage{amssymb}
\usepackage{mathrsfs}
\usepackage{hyperref}
\usepackage{tikz-cd}
\usepackage{amsthm}
\usepackage{amsmath}
\usepackage{esint}
\usepackage{derivative}
\usepackage{amsfonts}
\usepackage{bbold}

\usepackage{biblatex}
\newcommand{\R}{\check{R}}
\newcommand{\W}{\mathcal{W}}

\newcommand{\copr}{\Delta}

\newtheorem{theorem}{Theorem}[section]
\newtheorem{lemma}[theorem]{Lemma}
\newtheorem{cor}[theorem]{Corollary}
\newtheorem{prop}[theorem]{Proposition}
\theoremstyle{definition}
\newtheorem{definition}[theorem]{Definition}
\newtheorem{example}[theorem]{Example}

\theoremstyle{remark}
\newtheorem{remark}[theorem]{Remark}

\newcommand{\dontprint}[1]\relax

\begin{document}
\title{Restricted quantum groups as graded Hopf algebras}
\author{Jelena Ani\'c}
\address{Department of Mathematics, ETH Zurich, 8092 Zurich, Switzerland}
\email{jelena.anic@math.ethz.ch}
\author{Giovanni Felder}
\address{Department of Mathematics, ETH Zurich, 8092 Zurich, Switzerland}
\email{giovanni.felder@math.ethz.ch}
\begin{abstract}
  We introduce the notion of $\pi^2$-graded Hopf algebra, where the grading is by the
  double groupoid of commutative diagrams of a finite groupoid $\pi$. The finite dimensional
  representations of a $\pi^2$-graded Hopf algebra form a rigid monoidal category with
  a fibre functor to the category of $\pi$-graded vector spaces. The main example is
  given by the restricted quantum groups underlying the Andrews--Baxter--Forrester
  restricted solid-on-solid models of statistical mechanics and, more generally, the
  Jimbo--Miwa--Okado models associated to classical Lie algebras.
\end{abstract}
\maketitle
\tableofcontents
\section{Introduction}
Quantum groups emerged in the 1980s through the independent works of Drinfeld \cite{Drinfeld1987}, \cite{MR914215} and Jimbo \cite{Jimbo1985}, motivated by the study of integrable systems and solutions to the quantum Yang--Baxter equation. Formally, a quantum group is a Hopf algebra deformation of the universal enveloping algebra $U(\mathfrak{g})$ of a semisimple Lie algebra $\mathfrak{g}$, depending on a deformation parameter $q=e^{\hbar}\in\mathbb{C}^\times$. Beyond standard quantum groups, \textit{dynamical quantum groups}, introduced in \cite{FelderICMP1995}, describe additional levels of deformation. They depend on a \textit{dynamical parameter} $\lambda$ taking values in the dual of a Cartan subalgebra $\mathfrak{h}^\ast$. This dependence arises naturally in the study of dynamical Yang--Baxter equations 
\begin{align}\label{e-DYBE}
    R(z-w,\lambda +\hbar h^{(3)})^{(12)}R(z,\lambda)^{(13)}R(w,\lambda +\hbar h^{(1)})^{(23)}\\
    =R(w,\lambda )^{(23)}R(z,\lambda +\hbar h^{(2)})^{(13)}R(z-w,\lambda )^{(12)},\notag
\end{align}
where the $R$-matrix is a function $R(z,\lambda)$ of both a complex spectral and a dynamical variable and takes values in $\operatorname{End}_{\mathfrak{h}}(V\otimes V)$. The vector space $V=\oplus _{\mu\in\mathfrak{h}^\ast} V_{\mu}$ is a finite-dimensional semisimple $\mathfrak{h}$-module and the upper indices in the equation describe on which weight subspace the operator acts, for example, $R(w,\lambda-\hbar h^{(1)})^{(23)}$ acts on $V_{\mu_1}\otimes V_{\mu_2}\otimes V_{\mu_3}$ as $\operatorname{Id}\otimes R(w,\lambda-\hbar \mu_1)$.

Dynamical \textit{elliptic quantum groups} constitute the most general class of these deformations. They naturally appear in the study of elliptic integrable systems (e.g., the elliptic Calogero--Moser and Ruijsenaars--Schneider models), elliptic Knizhnik--Zamolodchikov--Bernard equations, \cite{FelderVarchenko1996}, \cite{FelderVarchenko1997}, \cite{FelderVarchenko1999}, \cite{FelderTarasovVarchenkoAMS1997}, \cite{FelderTarasovVarchenko1999}, \cite{FelderVarchenkoAoM2002}, \cite{FelderVarchenko2004},  \cite{EtingofSchiffmann1999}, \cite{Konno2020}. They are associated with elliptic solutions of the dynamical Yang--Baxter equation, where the dependence on the spectral parameter is elliptic rather than trigonometric or rational. They appear in several new developments in representation theory and mathematical physics, see, e.g., \cite{Fronsdal}, \cite{JimboKonnoOdakeShiraishi1999},\cite{KoelinkvanNordenRosengren2004}, \cite{Rosengren2011}, \cite{Hartwig2009},  \cite{GautamToledanoLaredo2017}, \cite{AganagicOkounkov2016}, \cite{AganagicOkounkov2017}, \cite{AganagicFrenkelOkounkov2018}, \cite{FelderZhangSM2017}, \cite{Zhang2018}, \cite{CostelloWittenYamazakiI2018}, \cite{RimanyiTarasovVarchenko2019}, \cite{RimanyiTarasovVarchenko2019}. The elliptic quantum groups can be defined in terms of $L$ (or $T$) operators, which is the elliptic and dynamical analogue of the FRST formulation (after Faddeev--Reshetikhin--Semenov--Tian--Shansky--Takhtadjan) of the quantum groups. The algebra is generated by the matrix entries of an $L$-operator $L(u,\lambda)=\left(L_{ij}(u,\lambda) \right)^n_{i,j=1}$ that satisfy the quadratic $RLL$ relations:
\begin{align*}
    R(z-w,\lambda +\hbar h^{(3)})^{(12)}L(z,\lambda)^{(1)}L(w,\lambda +\hbar h^{(1)})^{(2)}\\
    =L(w,\lambda)^{(2)}L(z,\lambda +\hbar h^{(2)})^{(1)}R(z-w,\lambda)^{(12)}.
\end{align*}
Because of the dynamical dependence of $L$-operators, one has to introduce shifts in the coalgebraic structure. For example, the coproduct acts on $L$-operators by
\begin{align*}
    \copr (L(z,\lambda))=L(z,\lambda+\hbar h^{(3)})^{(12)}L(z,\lambda )^{(13)}.
\end{align*}
Moreover, the representations are the vector spaces over the field of meromorphic functions of the dynamical variables on which generators act as difference operators. Thus, the suitable structure underlying elliptic quantum groups is that of $\mathfrak{h}$-\textit{Hopf algebroid}, introduced by Etingof and Varchenko \cite{EtingofVarchenko1998-2}.

As in the case of standard quantum groups, there are two more presentations of the elliptic quantum groups, the one in terms of the Chevalley type generators with a structure of a quasi-Hopf algebra, and the one in terms of the Drinfeld generators with a Hopf algebroid structure. Konno \cite{Konno2015} established an isomorphism between the central extension of the elliptic quantum group $E_{q,p}(\hat{\mathfrak{gl}_N})$ defined via the $L$-operators and the Drinfeld realization $U_{q,p}(\hat{\mathfrak{gl}_N})$.

In the context of solid-on-solid (SOS) lattice models, elliptic quantum groups provide the algebraic framework governing their integrability. These models describe the configurations of discrete (integer) height variables $l_i$, assigned to each lattice site $i\in M$ on a subset $M$ of a square lattice in the plane or 2-dimensional torus, subject to the constraints $\vert l_i -l_j\vert =1$ for adjacent sites $i,j$. Each unit square (face) with vertices in $M$ carries a Boltzmann weight, and the probability of a given configuration is proportional to the product of the Boltzman weights over all faces. The Boltzmann weights satisfy the star-triangle equation, which is a dynamical Yang--Baxter equation, where the dynamical parameter corresponds to the local height variable of the model. Restricted solid-on-solid (RSOS) models are special cases of SOS models where the allowed height variables take only a finite set of discrete values $l_i\in \{1,2,...,r-1\}$. These models were introduced by Andrews, Baxter and Forrester \cite{AndrewsBaxterForrester1984} and play a crucial role in conformal field theory and minimal models, since their continuum limits correspond to specific rational conformal field theories. 

Constructing a quantum group-like object for restricted models, however, involves specific challenges: one needs to specialize the $R$-matrices and representation matrices, which are meromorphic functions of the dynamical variables to a finite set. One then must avoid poles and ensure that the relations do not involve evaluation of $R$-matrices and $L$-operators at values of the dynamical variables outside the set. 

The aim of this paper is precisely to formulate the algebraic structure underlying these models. For this purpose
the second author and Ren \cite{FelderRen2021} introduced the groupoid-grading approach to describe categories of representations of restricted quantum groups. Unlike standard elliptic quantum groups whose representation spaces are vector spaces of meromorphic functions in continuous dynamical variables, \textit{restricted quantum groups} can have only finitely many allowed heights, meaning that their dynamical variables take values in a finite set. In \cite{FelderRen2021}, the authors replaced  the continuous dependence on the dynamical variable by a grading by a discrete groupoid $\pi$. They used the categorical viewpoint where they treated the category of representations of the restricted quantum group as the main object, which is a monoidal category with a forgetful functor to $\pi$-graded vector spaces. Within this $\pi$-graded setting they defined the $R$-matrix as isomorphisms $V_i\otimes V_j \cong V_j\otimes V_i$ in the graded category that satisfy a dynamical Yang-Baxter equation. Shifts are now implemented via the groupoid grading. This enables the machinery of the quantum inverse scattering method to be applied.  

 In this work, we develop the groupoid-graded framework further, providing an explicit Hopf-algebraic realization of the restricted quantum group rather than defining it through its representation category. The grading is extended to a \textit{double groupoid} $\pi^2$ equipped with its commutative squares (2-cells), which may be viewed as morphisms between two groupoids - the vertical and the horizontal ones - both having $\pi$ as their set of objects. The \textit{double groupoid graded Hopf algebra} is defined and its finite dimensional representations are shown to form a rigid monoidal category with a faithful monoidal functor to the category of $\pi$-graded vector spaces. 
 
 The main examples are restricted quantum groups for classical Lie algebras, for which the $R$-matrices are obtained from the elliptic solutions of the star-triangle relations of Jimbo, Miwa and Okado \cite{JimboMiwaOkado1988}, generalizing the Andrew--Baxter--Forrester solution. We show that they possess the structure of a $\pi^2$-graded Hopf algebra. Additionally to $RTT$ relations we impose crossing symmetry relation  to ensure the existence of inverse elements of the generators and to realize the antipode axioms. The finite groupoid associated with these representations is constructed so that its objects are in one-to-one
 correspondence with the integrable modules of  the affine Lie algebra of type $\mathfrak g^{{(1)}}=A_n^{(1)}$, $B_n^{(1)}$, $C_n^{(1)}$ or $D_n^{(1)}$ of fixed level $\ell\in\mathbb Z_{\geq0}$. By construction these Hopf algebras come with a representation, the vector
 representation. Its twist by the shift automorphism provides a continuous family of representations, the
 vector representation with spectral parameter. By taking subquotients of tensor products of such representations we get
 a category of representations, and corresponding solutions of the Yang--Baxter equation with elliptic coefficients.  

 As shown by Frenkel and Reshetikhin \cite{FrenkelReshetikhin1992}, the monodromy of the trigonometric quantum Knizhnik-Zamolodchikov equation provides elliptic solutions of the star-triangle relations. It should be expected that our approach extends to that setting, and give $\pi^2$-graded Hopf algebras associated with any semisimple Lie algebra. 

The dynamical Yang--Baxter equations has motivated several algebraic definitions of what a dynamical quantum group should be, starting from the notion of $\mathfrak h$-bialgebroids of Etingof and Varchenko \cite{EtingofVarchenko1998-2}. These definitions focus on the case of a dynamical variable taking continuous values, while in our main examples it takes values in a finite set. We review some of the literature on bialgebroids and compare it with our approach in \ref{sec-bialgebroids}.

Groupoid-graded vector spaces recently appeared in Baxterization \cite{Ren2023-1}, further models of statistical mechanics \cite{Ren2023-2} and in 4-dimensional super Yang-Mills theory \cite{BertlePomoniZhangZoubos2025}. It would be interesting to apply our construction to these cases. An open question is also to describe restricted elliptic quantum groups for general Lie algebras (or quivers) by Drinfeld currents as in \cite{JimboKonnoOdakeShiraishi1999}, \cite{GautamToledanoLaredo2017}, \cite{Konno2020}, \cite{YangZhao2017},\cite{YangZhao2019}. 

The text is organized as follows. In Section 2 we define the $\pi^2$-graded Hopf algebra and show that its finite dimensional representations have a natural structure of a rigid monoidal category. In Section 3 we briefly review the Yang--Baxter equation and its solution ($R$-matrix) in the context of groupoid-graded vector spaces and introduce the concept of $R$-matrix rotations. The main example of such $R$-matrices is given by Jimbo--Miwa--Okado models discussed in Section 4. In Section 5 we introduce the notion of {\em elliptic quantum algebra}, suitable for families of algebras defined by quadratic relations, such as $RTT$-relations, with elliptic coefficients. We construct the restricted elliptic quantum groups associated to classical Lie algebras $\mathfrak{sl}_n$, $\mathfrak{so}_n$ and $\mathfrak{sp}_n$ and show that they admit a $\pi^2$-graded Hopf algebra structure. The Jimbo--Miwa--Okado models provide representations of these algebras. 
\section{Double groupoid graded Hopf algebras}
The goal of this section is to introduce the notion of Hopf algebras
in the context of groupoid graded vector spaces and describe the
structure of the category of their representations.  We start by
reviewing the notion of groupoid graded vector spaces as introduced in
\cite{FelderRen2021} and proceed to define groupoid graded algebras
and coalgebras. In this context, bialgebras are naturally graded by
double groupoids. For the application we have in mind it will be
sufficient to consider gradings by the double groupoid of commutative
squares of a groupoid.

\subsection{Groupoid graded vector spaces}\label{ss-2.1}
We start with groupoid graded vector spaces, see \cite{FelderRen2021}
for more details. We fix a ground field $k$, which will be $\mathbb C$ in the examples.
\begin{definition}
  Let $\pi$ be a small groupoid with set of morphisms (arrows) $\pi_1$
  and set of objects $\pi_0$.  A $\pi$-graded vector space over $k$ is
  collection $(V_\alpha)_{\alpha\in \pi_1}$ of $k$-vector spaces
  labeled by arrows of $\pi$. Morphisms $f\colon V\to W$ are
  collections of linear maps $f_\alpha\colon V_\alpha\to W_\alpha$.
\end{definition}
We denote by $\pi(a,b)\subset\pi_1$ the set of arrows $\alpha$ from $a\in\pi_0$
to $b\in\pi_0$; $a=s(\alpha)$ is the source of $\alpha$ and $b=t(\alpha)$ is the target of $\alpha$. We denote by $\operatorname{id}_a\in\pi(a,a)$ the identity
object of $a\in\pi_0$ and by $\alpha^{-1}\in\pi(b,a)$ the inverse of
$\alpha\in\pi(a,b)$.

The $\pi$-graded vector spaces over $k$ form an abelian monoidal
category $\mathrm{Vect}_\pi$ with tensor product
$(V\otimes W)_\gamma=\oplus_{\beta\circ\alpha=\gamma} V_\alpha\otimes
V_\beta$ and unit $\mathbf 1$ such that $1_\alpha=k$ for identity
arrows $\alpha$ and $1_\alpha=0$ otherwise.  The category
$\operatorname{Vect}_\pi$ has internal homs with
$\underline{\operatorname{Hom}}(V,W)_\alpha
=\Pi_\beta\operatorname{Hom}_k(V_\beta,W_{\beta\circ\alpha})$.  The
dual $V^*=\underline{\operatorname{Hom}}(V,\mathbf 1)$ of a
$\pi$-graded vector space $V$ has components
$(V^*)_\alpha=\operatorname{Hom}_k(V_{\alpha^{-1}},k)$ and comes with
a nondegenerate pairing $V^*\otimes V\to \mathbf1$.
The category $\operatorname{Vect}^f_\pi$ of $\pi$-graded vector spaces
of finite type is the full subcategory of $\operatorname{Vect}_\pi$
whose objects have finite dimensional components $V_\alpha$ and are
such that for any object $a\in\pi_0$ the set of arrows
$\alpha\in\pi_1$ with source or target $a$ such that $V_\alpha\neq 0$
is finite.  We have a character map
$\operatorname{ch}\colon \operatorname{Vect}_\pi\to\mathbb
Z(\pi)=\mathbb Z^{\pi_1}$ sending $V$ to the collection of dimensions
$\operatorname{dim}(V_\alpha)_{\alpha\in\pi_1}$ inducing an
isomorphism from the Grothendieck ring of $\operatorname{Vect}^f_\pi$
to the convolution ring of the groupoid.

\subsection{Groupoid graded algebras and coalgebras}\label{ss-2.2}
\begin{definition}
  A $\pi$-graded unital algebra over $k$ is a $\pi$-graded vector
  space $(A_\alpha)_{\alpha\in\pi_1}$ over $k$ with morphisms of
  $\pi$-graded vector spaces $\nabla\colon A\otimes A\to A$ and
  $\eta\colon\mathbf 1\to A$.  obeying the associativity and unit
  axioms. In other words $A$ is a unital algebra object in the
  monoidal category $\operatorname{Vect}_\pi$.
\end{definition}
Representations of a $\pi$-graded algebras $A$ are naturally graded by
the set of objects of $\pi$:
\begin{example}\label{example1}
  Let $V=(V_a)_{a\in\pi_0}$ be a collection of vector spaces labeled
  by objects of $\pi$. Then
  $\operatorname{End}(V)=(\operatorname{End}(V)_\alpha)_{\alpha\in\pi_1}$
  with
  \[
    \operatorname{End}(V)_\alpha=\operatorname{Hom}_k(V_a,V_b), \quad
    \alpha\in\pi(b, a),
  \]
  is a groupoid graded algebra.
\end{example}
A representation of a $\pi$-graded algebra on $V=(V_a)_{a\in\pi_0}$ is
then a morphism of $\pi$-graded algebras
$\rho\colon A\to \operatorname{End}(V)$. Explicitly, for any arrow
$\alpha\in\pi(a,b)$ we have a linear map
$\rho_\alpha\colon A_\alpha\to\operatorname{Hom}_k(V_b,V_a)$ such that
$\rho_{\beta\circ \alpha}(\nabla(x,y))=\rho_\alpha(x)\rho_\beta(y)$ for
all composable arrows $\beta,\alpha$ and all $x\in A_\alpha$ and
$y\in A_\beta$, and $\rho_{\alpha}(\eta_\alpha(1))=\mathrm{Id}_{V_a}$
for identity arrows $\alpha\in\pi_1(a,a)$.

We can view representations as left modules $A\otimes V\to V$,
$a\otimes v\mapsto \rho(a)v$, where the tensor product is defined as
having components
$(A\otimes V)_a=\oplus_{\alpha\in\pi(a,b)}A_\alpha\otimes V_b$. 

\begin{definition}
  A $\pi$-graded counital coalgebra is a coalgebra in the category of
  $\pi$-graded vector spaces.
\end{definition}
Thus, such a coalgebra $C$ comes with structure maps
$\Delta\colon C\to C\otimes C$ and $\epsilon\colon C\to \mathbf 1$
obeying coassociativity and counit axioms.

\subsection{Generators and relations}  \label{sec-2.3}
A left ideal of a $\pi$-graded algebra $A$ is a $\pi$-graded subspace
$B\subset A$ such that $\nabla(x,y)\subset B_{\beta\circ\alpha}$ for
all composable arrows $\beta,\alpha$ and $x\in A_\alpha$ and
$y\in B_\beta$. Similarly, we have the notions of right and two-sided
ideals. The quotient of a $\pi$-graded algebra by a two-sided ideal is
then also a $\pi$-graded algebra.  We can thus define $\pi$-graded
algebras by generators and relations.  The free $\pi$-graded algebra
generated by a $\pi$-graded vector space $V$ is the tensor algebra
$T(V)=\mathbf 1\oplus V\oplus (V\otimes V)\oplus\cdots$. For a
$\pi$-graded subspace $R\subset T(V)$ the $\pi$-graded algebra
generated by the $\pi$-graded vector space $V$ subject to the
relations $R$ is the quotient of $T(V)$ by the smallest ideal
containing $R$.

\subsection{The double groupoid of commutative squares}\label{ss-2.3}
Let $\pi$ be a small groupoid. The commutative squares
\begin{equation}\label{e-alpha}
  \begin{tikzcd}
    a\arrow[d,"u"']\arrow[r,"f"]\arrow[rd,phantom,""] &
    a'\arrow[d,"v"]
    \\
    b\arrow[r,"g"]&b'
  \end{tikzcd}
\end{equation}
in $\pi$ are the 2-cells of a double groupoid $\Gamma=\pi^2$. This
means that they are morphisms (but we call them 2-cells to avoid
confusion) for two groupoids, called the horizontal groupoid
$\Gamma^\bullet$ and the vertical groupoid $\Gamma^\circ$. Both have
the morphisms of $\pi$ as their set of objects.  The objects
$\Gamma_0^\bullet$ of $\Gamma^\bullet$ are called vertical morphisms
and those of $\Gamma^\circ$ are called horizontal morphisms. The
commutative square $\alpha$ in the picture above is both a morphism of
the vertical groupoid from the horizontal morphism $f$ to the
horizontal morphism $g$ and also a morphism of the horizontal groupoid
from $u$ to $v$:
\[
  \alpha\in \Gamma^\circ(f,g),\quad \alpha\in \Gamma^\bullet(u,v).
\]
The horizontal composition $\bullet$ of 2-cells is defined
if the corresponding squares have matching vertical morphisms when placed
next to each other: given two matching 2-cells
\begin{equation}\label{e-hor-composition}
  \begin{tikzcd}
    a\arrow[d,"u"']\arrow[r,"f"]\arrow[rd,phantom,"\alpha"]
    &
    a'\arrow[d,"v"']\arrow[r,"f'"]\arrow[rd,phantom,"\beta"]
        &
    a''\arrow[d,"w"]
   \\
    b\arrow[r,"g"']&    b'\arrow[r,"g'"']&b''
  \end{tikzcd},
\end{equation}
the horizontal composition is given by the square
\[
  \beta\bullet\alpha=
    \begin{tikzcd}
    a\arrow[d,"u"']\arrow[r,"f'\bullet f"]
    &
    a'\arrow[d,"w"]
   \\
    b\arrow[r,"g'\bullet g"']&b'
  \end{tikzcd}.
\]
Here $f'\bullet f$ and $g'\bullet g$ is the composition of morphisms in $\pi$, denoted
$\bullet$ for consistency of notation.
The vertical composition  $\circ$ is defined if the horizontal morphisms match when we put the squares
on top of each other:
\[
  \beta\bullet\alpha=
    \begin{tikzcd}
    a\arrow[d,"u"']\arrow[r,"f'\bullet f"]
    &
    a'\arrow[d,"w"]
   \\
    b\arrow[r,"g'\bullet g"']&b'
  \end{tikzcd},\quad
  \begin{tikzcd}
    a\arrow[d,"u"']\arrow[r,"f"]\arrow[rd,phantom,"\alpha"]
    &
    a'\arrow[d,"v"]
   \\
    b\arrow[d,"u'"']\arrow[r,"g"]\arrow[rd,phantom,"\beta"]
    &
    b'\arrow[d,"v'"]
   \\
    c\arrow[r,"h"']&c'
  \end{tikzcd}
  \quad \longmapsto
  \quad \beta\circ\alpha=
    \begin{tikzcd}
    a\arrow[d,"u'\circ u"]\arrow[r,"f"]
    &
    a'\arrow[d,"v'\circ v"]
   \\
    c\arrow[r,"h"']&c'
  \end{tikzcd}
\]
We have an identity 2-cell $\mathrm{id}^\bullet_u$ of $\Gamma^\bullet$ for
each vertical morphism $u$, whose source and target are horizontal
identity morphisms. Similarly, there is an identity 2-cell
$\mathrm{id}^\circ_f$ for each horizontal morphism $f$:
\[
  \mathrm{id}^\bullet_u=
 \begin{tikzcd}
   a\arrow[d,"u"']\arrow[r,"\mathrm{id}_a"]
   & a\arrow[d,"u"]
   \\
   b\arrow[r,"\mathrm{id}_b"']&b
 \end{tikzcd},
 \qquad\mathrm{id}^\circ_f=
  \begin{tikzcd}
    a\arrow[d,"\mathrm{id}_a"']\arrow[r,"f"]
    &
    a'\arrow[d,"\mathrm{id}_{a'}"]
   \\
    a\arrow[r,"f"']&a'
  \end{tikzcd}
\]
The 2-cells are invertible in both $\Gamma^\bullet$ and $\Gamma^\circ$. For example, the inverse
of $\alpha\in \Gamma_1^\bullet$ above is the commutative square
\[
  \begin{tikzcd}
    a'\arrow[d,"v"']\arrow[r,"f^{-1}"]\arrow[rd,phantom," "]
    &
    a\arrow[d,"u"]
   \\
    b'\arrow[r,"g^{-1}"]&b
  \end{tikzcd}
\]  
The compatibility
of horizontal and vertical composition is expressed by the {\em interchange law}
\[
  (\delta\circ\beta)\bullet(\gamma \circ\alpha)
  =
  (\delta\bullet\gamma)\circ(\beta\bullet\alpha)
\]
which holds whenever both sides are defined, namely if the sources and targets match
as in the following picture:
\begin{equation}\label{e-exchange}
  \begin{tikzcd}
    a\arrow[d,"u"']\arrow[r,"f"]\arrow[rd,phantom,"\alpha"]
    &
    a'\arrow[d,"v"]\arrow[r,"f'"]\arrow[rd,phantom,"\beta"]
    &
    a''\arrow[d,"w"]
    \\
    b\arrow[d,"u'"']\arrow[r,"g"]\arrow[rd,phantom,"\gamma"]
    &
    b'\arrow[d,"v'"]\arrow[r,"g'"]\arrow[rd,phantom,"\delta"]
    &
    b''\arrow[d,"w'"]
    \\
    c\arrow[r,"h"']&c'\arrow[r,"h'"']& c''
  \end{tikzcd}
\end{equation}
In a more abstract language the above properties are the statement
that commutative squares form a double groupoid, namely a groupoid
in the category of small groupoids.
Its groupoid of objects is the groupoid whose objects are the objects
of $\pi$ and whose morphisms are the horizontal morphisms; its groupoid of
morphisms is the groupoid whose objects are vertical morphisms and whose
morphisms are the 2-cells.
In fact, our construction works for any double groupoid, but we limit
ourselves to the special case needed for our application.

\subsection{The doubly graded category of \texorpdfstring{$\Gamma$}{gamma}-graded vector spaces}\label{ss-2.4}
Let $\pi$ be a small groupoid and $\Gamma=\pi^2$ the double category
of its commutative squares. A $\Gamma$-graded
vector space over a field $k$ is a collection $(A_\alpha)_{\alpha\in \Gamma}$ of vector
spaces over $k$ labeled by the 2-cells of $\Gamma$. Since 2-cells are morphisms
of two groupoids $\Gamma^\bullet$ and $\Gamma^\circ$, we have two
tensor products $\otimes_h$, $\otimes_v$ on $\Gamma$-graded vector
spaces. They are defined by
\[
  (a\otimes_h b)_{\gamma}=\oplus_{\beta\bullet\alpha=\gamma}
  a_\alpha\otimes b_\beta,
\]
and
\[
  (a\otimes_v
  b)_{\gamma}=\oplus_{\beta\circ\alpha=\gamma}a_\alpha\otimes b_\beta,
\]
respectively.

Thus, we have two monoidal structures, the horizontal one and the vertical
one in the category of $\Gamma$-graded vector spaces. The unit $I_h$
for the horizontal structure has components $(I_h)_\alpha=k$ if
$\alpha=\mathrm{id}^\bullet_u$ is an identity 2-cell and
$(I_h)_\alpha=0$ otherwise. Similarly, we have a vertical unit $I_v$.
The two monoidal structures are compatible in the sense that there is
a natural transformation (of 4 arguments)
\begin{equation}\label{e-iso}
  (A\otimes_h C)\otimes_v(B\otimes_hD)
  \to (A\otimes_v B)\otimes_h(C\otimes_vD).
\end{equation}
sending $(a\otimes c)\otimes(b\otimes d)$ for
$a\in A_\alpha,\dots, d\in A_\delta$ to
$(a\otimes b)\otimes (c\otimes d)$ whenever $\alpha,\dots,\delta$ are
as in \eqref{e-exchange} and to zero otherwise. The natural
transformation \eqref{e-iso} expresses the fact that the horizontal
tensor product
$\otimes_h\colon \mathcal V\times \mathcal V\to\mathcal V$ on the
category $\mathcal V$ of $\Gamma$-graded vector spaces is a strong monoidal
functor for the vertical monoidal structure.

Variants of this construction exist: one that will be relevant for us to define the monoidal category of modules over doubly graded Hopf algebras is the vertical tensor product of a $\Gamma$-graded vector space $A$ with a $\pi$-graded vector space $V$. It is the $\pi$-graded vector space $A\otimes V$ such that $(A\otimes_v V)_f=\oplus_{s^\circ(\alpha)=f}A_\alpha\otimes V_{t^\circ(\alpha)}$. The sum is over the commutative squares $\alpha$ whose $\Gamma^\circ$-source (top arrow) is $f$.  This vertical tensor product is a functor $\mathcal V\times \operatorname{Vect}_{\pi}\to\operatorname{Vect}_{\pi}$ and we have a natural transformation corresponding to the picture \eqref{e-hor-composition}
\[
(A\otimes_v V)\otimes(B\otimes_v W)\to (A\otimes_h B)\otimes_v (V\otimes W),
\]
for $\Gamma$-graded $A,B$ and $\pi$-graded $V,W$. Here $\otimes$ is the tensor product in $\operatorname{Vect}_\pi$ which in this context can be thought of as horizontal.

\begin{example}\label{ex-2.5} Let $V,W$ be $\pi$-graded vector spaces. Then the vector space of morphisms $\operatorname{Hom}(V\otimes W,W\otimes V)$ is naturally $\pi^2$-graded. Homogeneous elements are morphisms which send $V_\alpha\otimes W_\beta$ to $W_\gamma\otimes V_\delta$ and are zero on other components. The morphism property requires that $\beta\circ\alpha=\delta\circ\gamma$, i.e., the four arrows form a commutative square. Here we took $V,W,W,V$ as the involved vector spaces since we have $R$-matrices in mind, see Section \ref{sec-3} but the same holds for any four $\pi$-graded vector spaces.
\end{example}
\subsection{\texorpdfstring{$\Gamma$}{gamma}-graded bialgebras for finite groupoids}\label{ss-2.5}
Here we assume that $\pi$ is a finite groupoid (i.e., with finite sets of objects and morphisms), which is
the case for our application to restricted quantum groups. This avoids questions
of completion. Let again $\Gamma=\pi^2$ be the double groupoid of commutative squares in $\pi$.

A $\Gamma$-graded bialgebra is a $\Gamma$-graded vector space with a product
$\nabla\colon A\otimes_vA\to A$ with unit $\eta\colon I_v\to A$ and a
coproduct $\Delta\colon A\to A\otimes_hA$ with counit
$\epsilon\colon A\to I_h$, obeying the axioms (H1)--(H5) below.
\begin{enumerate}
\item[(H1)] $(A,\nabla,\eta)$ is a unital algebra for the vertical
  monoidal structure.
\item[(H2)] $(A,\Delta,\epsilon)$ is a counital coalgebra for the
  horizontal monoidal structure.
\item[(H3)] The following diagrams (involving the morphism
  \eqref{e-iso}) commute:
      \[
        \begin{tikzcd}
          A\otimes_vA\arrow[d,"\Delta\otimes_v\Delta"] \arrow[r,"\nabla"]
          & A\arrow[r,"\Delta"] & A\otimes_h A
          \\
          (A\otimes_h A) \otimes_v(A\otimes_h A) \arrow[rr] & &
          (A\otimes_vA) \otimes_h (A\otimes_vA)
          \arrow[u,"\nabla\otimes_h\nabla"]
        \end{tikzcd}
      \]
    \end{enumerate}
A $\Gamma$-graded vector space obeying (H1-H3) can be called a $\Gamma$-graded
bialgebra. The notion also makes sense for an infinite groupoid. For the remaining
axioms it is easier to assume that $\pi$ is finite since in this case no completion of
tensor products is needed.

Before formulating the axioms for the compatibility with unit and counit, we notice that
the multiplication of scalars defines an associative product on $I_h$
for the {\em vertical} tensor product, namely the map
\[
\nabla\colon  I_h\otimes_v I_h\to I_h
\]
sending $x\otimes y$ with $x\in (I_h)_{\mathrm{id}^\bullet_u}$ and
$y\in (I_h)_{\mathrm{id}^\bullet_{u'}}$ for composable $u,u'$ to
\[
  xy\in (I_h)_{\mathrm{id}^\bullet_{u'}\circ \mathrm{id}^\bullet_{u}
    =\mathrm{id}^\bullet_{u'\circ u}}
  =(I_h)_{\mathrm{id}^\bullet_{u'\circ u}}.
\]
Similarly, we have a coassociative coproduct
$\Delta\colon I_v\to I_v\otimes_h I_v$.

We will also need the map $\sigma\colon I_v\to I_h$ whose components
are all zero except for $\mathrm{id}^\circ_{f}$ with $f=\mathrm{id}_a$
with $a$ an object of $\pi$, depicted below, where the map is the
identity on $k$.
\[
 \begin{tikzcd}
   a\arrow[d,"\mathrm{id}_a"']\arrow[r,"\mathrm{id}_a"]\arrow[rd,phantom,""]
   & a\arrow[d,"\mathrm{id}_a"]
   \\
   a\arrow[r,"\mathrm{id}_a"']&a
  \end{tikzcd},
\]
Similarly, we have a map $\sigma'\colon I_h\to I_v$ which is the
identity on the component labeled by the above commutative square and
zero on the other components.
\begin{enumerate}
\item[(H4)] The following diagrams commute:
  \[
    \begin{tikzcd}
      I_v\arrow[r,"\eta"] \arrow[d,"\Delta"] &A\arrow[d,"\Delta"]
      \\
      I_v\otimes_h I_v\arrow[r,"\eta\otimes_h\eta"]&A\otimes_h A
    \end{tikzcd}
    \begin{tikzcd}
      A\otimes_vA\arrow[r,"\epsilon\otimes\epsilon"]
      \arrow[d,"\nabla"]& I_h\otimes_v I_h\arrow[d,"\nabla"]
      \\
      A\arrow[r,"\epsilon"] & I_h
    \end{tikzcd}
  \]
\item[(H5)] $\epsilon\circ\eta=\sigma\colon I_v\to I_h.$
\end{enumerate}
A $\Gamma$-graded bialgebra $A$ is in particular a
$\Gamma^\circ$-graded algebra and thus comes with a category of
modules graded by the objects of $\Gamma^\circ$ which are the
(horizontal) morphisms of $\pi$. Thus a module over a
$\Gamma^\circ$-graded algebra is a $\pi$-graded vector space.  The
coproduct defines an action of $A$ on tensor products of $\pi$-graded
modules and we obtain a monoidal category of modules. To define dual
modules, we need to add the existence of an antipode to our list of
axioms.  The antipode does not preserve the grading; rather it sends a
component labeled by $\alpha$ to the component labeled by
$\alpha^{-1}$. Let $A$ be a $\Gamma$-graded bialgebra. The opposite
bialgebra $A^{op}$ has components $A^{op}_{\alpha}=A_{\iota(\alpha)}$
and opposite product and coproduct. Here, the double inverse $\iota$ is
the composition of the inversion for both the vertical and the
horizontal groupoid structure:
\begin{equation}\label{e-double-inverse}
  \begin{tikzcd}
    a\arrow[d,"u"']\arrow[r,"f"]\arrow[rd,phantom,""]
    &
    a'\arrow[d,"v"]
   \\
    b\arrow[r,"g"]&b'
  \end{tikzcd}
  \stackrel{\iota}\longmapsto
  \begin{tikzcd}
    b'\arrow[d,"v^{-1}"']\arrow[r,"g^{-1}"]\arrow[rd,phantom,""]
    &
    b\arrow[d,"u^{-1}"]
   \\
    a'\arrow[r,"f^{-1}"]&a
  \end{tikzcd}
\end{equation}

\begin{definition} Let $\Gamma=\pi^2$ be the category of commutative
  squares of a finite groupoid $\pi$. A {\em $\Gamma$-graded Hopf
    algebra} is a $\Gamma$-graded bialgebra $A$ with an antipode
  $S\colon A\to A^{\mathrm{op}}$ such that
  \begin{enumerate}
    \item[(H6)] $S$ is an algebra isomorphism.
    \item[(H7)] The following diagrams commute:
  \[
    \begin{tikzcd}
      I_h\arrow[r,"\sigma' "] &I_v\arrow[d,"\eta"]
      \\
      A\arrow[u,"\epsilon"] \arrow[d,"\Delta"] &A
      \\
      A\otimes_h A\arrow[r,"S\otimes\mathrm{id}"]
      &
      A\otimes_v A\arrow[u,"\nabla"]
    \end{tikzcd}
    \begin{tikzcd}
      I_h\arrow[r,"\sigma' "] &I_v\arrow[d,"\eta"]
      \\
      A\arrow[u,"\epsilon"] \arrow[d,"\Delta"] &A
      \\
      A\otimes_h A\arrow[r,"\mathrm{id}\otimes S"]
      &
      A\otimes_v A\arrow[u,"\nabla"]
    \end{tikzcd}
  \]
  \end{enumerate} 
\end{definition}
The maps $S\otimes \mathrm{id}$ and $\mathrm{id}\otimes S$ require an
explanation since $S$ does not preserve the grading. More precisely, we
use the maps
\[
  A\otimes_h A\stackrel{i}\hookrightarrow A\otimes
  A\stackrel{p}\twoheadrightarrow A\otimes_vA,
\]
the first map being the inclusion map and the second the projection
onto the components with matching horizontal morphisms. Then
$S\otimes\mathrm{id}$ and $\mathrm{id}\otimes S$ are interpreted as
$p\circ(S\otimes\mathrm{id})\circ i$,
$p\circ(\mathrm{id}\otimes S)\circ i$, respectively.

\subsection{The monoidal category of representations}\label{ss-2.6}
In this section, we show that finite dimensional representations of $\pi^2$-graded Hopf algebras
have a natural structure of a rigid monoidal category.

Let $\pi$ be a finite groupoid, $\Gamma=\pi^2$ the double groupoid of its
commutative squares.  Let $A$ be a $\Gamma$-graded Hopf algebra. Then it
is in particular a $\Gamma^\circ$-graded algebra and by the construction
of Section \ref{ss-2.2} applied to the groupoid $\Gamma^\circ$,
representations of $A$ are defined on $\pi$-graded vector spaces,
$\pi$ being here the set of objects of $\Gamma^\circ$, namely the set of horizontal
morphisms.

Morphisms of representations are morphisms of $\pi$-graded vector
spaces commuting with the action of $A$. We thus get an abelian
category of representations of $A$ with a faithful functor to the
category of $\pi$-graded vector spaces.

The coproduct then defines a representation of $A$ on the tensor
product $V\otimes W$ of representations on $\pi$-graded vector spaces
$V$ and $W$. Namely, we first have an action of $A\otimes_h A$:
$a\otimes b\in A\otimes_h A$ acts on $V\otimes W$ by
\[
  ((a\otimes b)\cdot(v\otimes w))_l=\sum a_\alpha v_g\otimes b_\beta
  w_{g'}, \quad l\in\pi_1,
\]
where the summation is over pairs of matching 2-cells $\alpha,\beta$
as in \eqref{e-hor-composition} with bottom horizontal morphisms
$g,g'$, and such that $f'\bullet f=l$. Using the coproduct
$\Delta\colon A\to A\otimes A$ we obtain an action of $A$, defining a
representation of $A$ on $V\otimes W$.  The first part of Axiom (H4)
ensures that it is a representation of the unital algebra $A$.

The unit for the tensor product of representations is the {\em trivial
  representation}: its representation space is the $\pi$-graded vector
space $\mathbf 1$ (the monoidal unit of $\operatorname{Vect}_\pi$)
with components $(\mathbf 1)_f=k$ for $f=\mathrm{id}_a$, $a\in\pi_0$
and zero otherwise.  The action of $A$ is via the counit
$\epsilon\colon A\to I_h$. The second part of (H4) and (H5) state that
this is a morphism of unital $\Gamma^\circ$-graded algebras. The
composition
\[
  A\stackrel\epsilon\longrightarrow I_h\to \operatorname{End}(\mathbf 1)
\]
with the algebra morphism $I_h\to \operatorname{End}(\mathbf 1)$
sending $1\in (I_{h})_{\mathrm{id}^{\bullet}_u}$ for $u\in\pi(a,b)$ to the
identity map
$\mathbf 1_{\mathrm{id}_b}\to\mathbf 1_{\mathrm{id}_a}$ (both are
$k$) is the trivial representation.

The representation $V^*$ dual to a finite dimensional representation
$V$ is the $\pi$-graded dual of $V$ with the action given by the
antipode.  In more detail, a representation on $V$ is given by a
collection of linear maps
$\rho_\alpha\colon A_\alpha\to \operatorname{Hom}_k(V_g,V_f)$ for each
$\alpha\in \Gamma^\circ(f,g)$. Then the map
\begin{align*}
  \rho^*_\alpha
  &\colon A_\alpha\to
    \operatorname{Hom}_k((V^*)_g,(V^*)_f)
  \\
  &=\operatorname{Hom}_k((V_{g^{-1}})^*,(V_{f^{-1}})^*)
  \\
  &\cong\operatorname{Hom}_k((V_{f^{-1}},V_{g^{-1}})
\end{align*}
is  $\rho_{\iota(\alpha)}\circ S_{\alpha}$, which is well-defined
since $\iota(\alpha)\in \Gamma^\circ(g^{-1},f^{-1})$. Explicitly, if we denote
by $\langle\;,\;\rangle$ the pairing $V^*\otimes V\to \mathbf 1$,
we have that
\[
  \langle\rho^*_\alpha(a)\phi,v\rangle=\langle\phi, \rho_{\iota(\alpha)}(S_{\alpha}(a))v\rangle,\qquad a\in
  A_{\alpha}, \quad \phi\in (V_{g^{-1}})^*,\quad v\in V_{f^{-1}}.
\]
By (H6) this is indeed a representation.  It then follows from the
second diagram of Axiom (H7) that the pairing
$V^*\otimes V\to \mathbf 1$ is a morphism to the trivial
representation.  Similarly, with our finiteness assumption, we have a
dual map $\mathbf 1\to V\otimes V^*$, which by the first diagram in
(H7) is also a morphism of representations.

As in the case of the category of finite dimensional representations
of an ordinary Hopf algebra, see, e.g., \cite[Chapter 5]{ChariPressleyBook1994},
one checks that these objects statisfy the axioms of a rigid
monoidal category:

\begin{theorem}
  The finite dimensional representations of a $\Gamma$-graded Hopf
  algebra form a rigid monoidal category with a faithful monoidal
  functor to the category of $\pi$-graded vector spaces.
\end{theorem}

\subsection{Bialgebroids}\label{sec-bialgebroids}
Quantum groups based on solutions of the dynamical Yang--Baxter equation and their categories of representations are not immediately captured by the language of Hopf algebras and braided categories as their non-dynamical counterpart. This fact has led to the development of new algebraic notions and constructions at various levels of generalization and abstraction. The main focus of this research was the
case where the dynamical variable takes values in a vector space, while for our application to RSOS models it takes values in a finite set.
Here is a short review of the literature and a comparison with our approach.

The prototype notion is that of an $\mathfrak h$-bialgebroid of Etingof and Varchenko \cite{EtingofVarchenko1998-2} in the setting of Lu's bialgebroids \cite{Lu1996}. In \cite{EtingofVarchenko1998-2} the dynamical variable is taken to live in the dual $\mathfrak h^*$ of a finite dimensional abelian Lie algebra $\mathfrak h$, and one considers the field $L$ of meromorphic functions on $\mathfrak h^*$, with its action of $\mathfrak h^*$ by translations $T_\alpha f(\lambda)=f(\lambda+\alpha)$, ($\alpha\in\mathfrak h^*,f\in L$).  An $\mathfrak h$-algebra is a $\mathfrak h^*$-bigraded algebra $A=\oplus_{\alpha,\beta\in\mathfrak h^*} A_{\alpha\beta}$ over $\mathbb C$ with two algebra embeddings $\mu_l,\mu_r\colon\colon M\to A_{00}$, called moment maps, such that for $a\in A_{\alpha\beta}$,
\begin{equation}\label{e-moment}
  \mu_l(f)a=a\mu_l(T_\alpha f),\quad
  \mu_r(f)a=a\mu_r(T_\beta f).
\end{equation}
An $\mathfrak h$-bialgebroid is a
$\mathfrak h$-algebra with a coproduct, which by definition is a linear map
$A\to A\otimes A$ sending $A_{\alpha\beta}$ to
$\oplus_\gamma A_{\alpha\gamma}\otimes_{L}A_{\gamma\beta}$, obeying
bialgebra-type axioms. One then obtains a monoidal category of
modules, which are $\mathfrak h^*$-graded vector space over $L$ and
where the algebra acts by difference operators in the dynamical
variables, see also \cite{Xu2001}.  This structure was further
generalized by Donin and Mudrov \cite{DoninMudrov2005},
\cite{DoninMudrov2006}. The behaviour of $\mu_r$, $\mu_l$ is
formalized with the notion of a base algebra \cite[Section
4.1]{DoninMudrov2006} in a monoidal category: a base algebra in a
monoidal category $\mathcal C$: recall that the centre $Z(\mathcal C)$
of $\mathcal C$ is a braided monoidal category consisting of pairs
$(L,\tau)$ where $L$ is an object of $C$ and $\tau$ is a collection
$\tau_X\in \operatorname{Hom}_{\mathcal C}(L\otimes X,X\otimes L)$ of
natural isomorphisms. A base algebra in $\mathcal C$ is a commutative
algebra in $Z(\mathcal C)$. The main example is when $\mathcal C$ is
the category of modules over a Hopf-algebra.
 Kalmykov and Safronov
\cite{KalmykovSafronov2022} considered a further generalization. They define the category HC of
Harish-Chandra bimodules over a base algebra $(L,\tau)$ as the
category of left $L$-modules in $\mathcal C$. In this context
bialgebroids are defined as bialgebra objects in a suitable tensor product
$\mathrm{HC}\otimes\mathrm{HC}$.

Such an object arises naturally from $\pi$-graded vector spaces and
$\pi^2$-graded bialgebras: the category $\operatorname{Vect}_\pi$ of
$\pi$-graded vector spaces comes, like any monoidal category, with a
canonical base algebra, namely the unit object $\mathbf 1$, which can
be identified with the algebra of functions on the set of objects of
$\pi$. A $\Gamma$-graded bialgebra $A$, with $\Gamma=\pi^2$ is in
particular a $\Gamma^\circ$-graded algebra for the vertical groupoid
$\Gamma^\circ$. Objects of $\Gamma^\circ$ are arrows in $\pi$ and
taking the direct sum of components graded by $\Gamma^\circ$ with
fixed source and target, we obtain a decomposition with components
$A_{\alpha\beta}$ labeled by pairs of arrows $\alpha,\beta$ of
$\pi$. We have two embeddings $\mu_l,\mu_r$ of the base algebra $L$ of
functions on objects of $\pi$: for $f\in L$ and
$a\in A_{\alpha\beta}$, we have
$\mu_l(f)a=f(s(\alpha))a, a\mu_l(f)=f(t(\alpha))a$ and
$\mu_r(f)a=f(s(\beta))a, a\mu_r(f)=f(t(\beta))a$.

In fact, in our main example of Jimbo--Miwa--Okado restricted models, see Section \ref{sec-4}, the groupoid
$\pi$ is a finite full subgroupoid of the transformation groupoid of the weight lattice $P$
of a split algebraic torus $H\cong (\mathbb C^\times)^r$. In this case we can identify our $\pi^2$-graded bialgebras as special cases of
Harish-Chandra bialgebroids, see Definition 3.29 and Example 3.30 in
\cite{KalmykovSafronov2022}. These are essentially $\mathfrak h^*$-bialgebroids
for the Lie algebra $\mathfrak h$ of $H$ but with some modifications in the definition:
(i) the base algebra is taken to be the algebra of regular functions $\mathcal O(\mathfrak h^*)$,
(ii) the algebra morphisms $\mu_l,\mu_r$ are not required to be embeddings, (iii) the
bigrading is by the weight lattice $P\subset\mathfrak h^*$. The additional condition
characterizing $\pi^2$-graded bialgebras is a support condition for $A$ for the four
module structures (left and right action via $\mu_l,\mu_r$)
over the base algebra $\mathcal O(\mathfrak h^*)$: namely, $A$
is supported on the subset of $(\mathfrak h^*)^4$ of weights forming the objects
of a commutative square in $\pi^2$.
It would be
interesting to adapt the Tannakian reconstruction theorem of
\cite{KalmykovSafronov2022} to this class of bialgebroids.


\section{Yang--Baxter equation}\label{sec-3} Let $k=\mathbb C$ and let $V\in\operatorname{Vect}^f_{\mathbb{C}}(\pi)$ be a finite dimensional $\pi$-graded vector space for a finite groupoid $\pi$.

%
%
%
An {\em $R$-matrix} is a meromorphic function $u\mapsto\check R(u)\in\operatorname{End}(V\otimes V)$ of the {\em spectral parameter} $u\in\mathbb C$ with values in the endomorphisms of $V\otimes V$, obeying the Yang--Baxter equation
\begin{equation}\label{e-YB}
  \check R^{(23)}(u-v)\check R^{(12)}(u)
  \check R^{(23)}(v)
  =  \check R^{(12)}(v)
\check R^{(23)}(u)
\check R^{(12)}(u-v)
\end{equation}
in $\operatorname{End}(V\otimes V\otimes V)$, together with the inversion relation
\begin{equation}
\label{inv}
    \R(u)\R(-u)=\operatorname{Id}_{V\otimes V},
\end{equation}
for all generic values of the spectral parameters $u,v$.

By restricting the operators to the graded component corresponding to a fixed source $a\in\pi_0$, one obtains the standard dynamical formulation, characterized by the explicit dependence of the operator on the dynamical variable,
\[
    \check R(u,a)\in\oplus_{\mu\in\pi,s(\mu)=a}\operatorname{End}_{\mathbb{C}}((V\otimes V)_{\mu}).
\]
For a transformation groupoid $\pi=G\rtimes M$ for the action of a group $G$ on a set $M$, where $\pi(a,b)=\{(a,g)\in M\times G\,|\,g\cdot a=b$ for $a,b\in \pi_0=M\}$, the Yang--Baxter equation
\eqref{e-YB} becomes the dynamical Yang--Baxter equation for $\check R(u,a)$:
\begin{align*}
  \check R^{(23)}(u-v,g^{(1)}a)\check R^{(12)}(u,a)
  &\check R^{(23)}(v,g^{(1)} a)
  \\
  &=  \check R^{(12)}(v,a)
\check R^{(23)}(u,g^{(1)} a)
\check R^{(12)}(u-v,a).
\end{align*}
Here $\check R^{(23)}(v,g^{(1)}a)$ acts as $R^{(23)}(v,g\cdot a) $ on $V_{\alpha}\otimes V_\beta\otimes V_\gamma$
where $\alpha=(a,g)$. For the weight lattice acting on the dual of a Cartan subalgebra by translation one obtains the
standard form \eqref{e-DYBE} of the dynamical Yang--Baxter equation for $R=P\check R$ and $P u\otimes v=v\otimes u$.

 We can define the components of the $R$-matrix by restricting it to $V_i\otimes V_j$ for composable arrows $i,j\in\pi_1$:
\[
  \check R(u)|_{V_i\otimes V_j}=\oplus_{k,l\in\pi_1}
  \check R_{ij}^{kl}(u).
\]
The sum is over $k,l$ such that $j\circ i=l\circ k$, and the component $\R^{kl}_{ij}(u)$ lives in
\[
  \R_{ij}^{kl}(u)\in\operatorname{Hom}_{\mathbb{C}}
  (V_i\otimes V_j , V_k\otimes V_l ).
\]
This is a homogeneous component for the $\pi^2$-grading of $\operatorname{End}(V\otimes V)$, see Example \ref{ex-2.5}.

Throughout this paper, we will often represent the $\pi^2$-graded component of an operator graphically via the commutative square corresponding to it in the double groupoid. For example, the component of the $R$-matrix will be depicted as
\tikzset{>=latex}
\[
\begin{tikzpicture}[scale=1.5]
  \draw[->,dashed](-.6,.5) node[left]{  $\R_{ij}^{kl}(u)=$}
  -- (1.7,.5);
  \draw[->,dashed](.5,-.6) 
  -- (.5,1.6);
  \draw[->](0,1) node[left]{$a$} -- (.5,1) node[above,fill=white]{$k$} -- (1,1)
  node[right]{$d$};
  \draw[->](1,1) -- (1,.5) node[right , fill=white]{$l$} -- (1,0);
  \draw[->](0,1) -- (0,.5) node[left, fill=white]{$i$} -- (0,0);
  \draw[->]node[left]{$b$}(0,0) -- (.5,0) node[below,fill=white]{$j$}-- (1,0)
  node[right]{$c$};
 \end{tikzpicture}
\]
where the dashed arrows indicate the convention for the direction in which the operator acts---sending $V_i\otimes V_j$ to $V_k\otimes V_l$---and will be omitted in subsequent diagrams.

\subsection{Crossing symmetry and
  rotationally invariant \texorpdfstring{$R$}{R}-matrix}\label{sec-3.1}
In this subsection, we introduce a special family of $R$-matrices that possess rotational symmetry.
This symmetry follows from a property called crossing symmetry, which is inspired by the theory of scattering
matrices in 2-dimensional integrable quantum field theory. See \cite{ReshetikhinSemenovTianShansky1990} for
versions of this property in representation theory. The crossing symmetry is formulated in terms of
a non degenerate bilinear form. Two cases will be distinguished. In the
orthogonal case (including the symplectic case) the bilinear form is defined on the underlying vector space $V$
of the $R$-matrix. In the general linear case we extend the $R$-matrix to $V\oplus V^*$ and define the bilinear
form there.

We begin by defining the notion of an inner product on $\pi$-graded vector spaces.
\begin{definition}
    An {\em inner product on a $\pi$-graded vector space $V$} is a morphism of $\pi$-graded vector spaces $\omega:V\otimes V\to \mathbf 1$ that is bilinear and nondegenerate. 
\end{definition}

Let $\omega_{i,j}:=\omega\vert_{V_i\otimes V_j}$ be the inner-product coefficients. From the definition one can see that the inner product is non-zero only for composable arrows $i$ and $j$. Moreover, since $\mathbf{1}$ has nonzero components only when graded by identity arrows, $i$ and $j$ must be mutual inverses, forcing $\omega_{i,j}$ to be of the form $\omega_{i,j}=\delta_{i,j^{-1}}\omega_{i,i^{-1}}$. We will use the notation $\omega_{i,-i}:=\omega_{i,i^{-1}}$.

Let $\sigma :\mathbf{1}\to V\otimes V$ be dual to the  inner product, defined by relations
\begin{equation}
\label{os}
    (\mathrm{id}_V\otimes\omega)(\sigma\otimes\mathrm{id}_V)=\mathrm{id}_V,\quad (\omega\otimes\mathrm{id}_V)(\mathrm{id}_V\otimes\sigma)=\mathrm{id}_V.
\end{equation}
We define the coefficient $\sigma^{i,-i}$ as the coefficient of the image of $1$ in the graded component $V_i\otimes V_{i^{-1}}\in (V\otimes V)_{\operatorname{id}^\bullet}$, i.e $\sigma (1)=\oplus_{i\in\pi_1} \sigma^{i,-i}$. We will graphically represent coefficients $\omega_{i,j}$ and $\sigma^{i,j}$ as
\begin{center}
$\omega_{i,j}=$\begin{tikzcd}
 a\arrow[d,"i"']\arrow[dd,bend left=40]\\
 b\arrow[d,"j"']\\
 c 
\end{tikzcd}$\qquad$ and $\qquad$
$\sigma^{i,j}=$\begin{tikzcd}
 a\arrow[d,"i"]\arrow[dd,bend right=40]\\
 b\arrow[d,"j"]\\
 c 
\end{tikzcd}.
\end{center}
The relations (\ref{os}) read as 
\begin{align*}
    \omega_{i,-i}\sigma^{-i,i}=\operatorname{id}_{V_{i^{-1}}} \quad\text{and}\quad \omega_{-i,i}\sigma^{i,-i}=\operatorname{id}_{V_{i^{-1}}},
\end{align*}
and are depicted as
\begin{center}
 \begin{tikzcd}
a\arrow[d,"i^{-1}"]\arrow[dd,bend right=40]  \\
 b\arrow[dd,bend left=40] \arrow[d,"i"]\\
 a\arrow[d,"j^{-1}"']\\
 c  
\end{tikzcd}
 $=\delta_{i,j}\qquad$ and $\qquad$ 
 \begin{tikzcd}
a\arrow[d,"j^{-1}"']\arrow[dd,bend left=40]  \\
 b\arrow[dd,bend right=40] \arrow[d,"i"]\\
 c\arrow[d,"i^{-1}"]\\
 b  
\end{tikzcd}$=\delta_{i,j}$.
\end{center}
We equip $V\in\operatorname{Vect}_{\mathbb{C}} (\pi)$ with an inner product $\omega:V\otimes V\to\mathbf{1}$.
\begin{definition}
  Let us call an $R$-matrix $\R(u)$ {\em crossing symmetric} with crossing parameter $\lambda\in \mathbb C$ with respect
  to the inner product $\omega$ if
   \begin{equation}
     \label{oRR}
     \begin{split}
         \omega^{(23)}\R^{(12)}(u+\lambda)\R^{(23)}(u)&=\alpha(u)\omega^{(12)},\\
        \R^{(12)}(u)\R^{(23)}(u+\lambda)\sigma^{(12)}&=\alpha(u)\sigma^{(23)},
         \end{split}
     \end{equation}
     for some meromorphic function $\alpha$ of $u\in\mathbb{C}$.
     \end{definition}
Let us define the operator $\R^{*_1}(u)\in\operatorname{End}(V\otimes V)$ as 
\begin{equation*}
    \R^{*_1}(u):=(\mathrm{id}_V\otimes\mathrm{id}_V\otimes\omega)(\mathrm{id}_V\otimes\R(\lambda-u)\otimes \mathrm{id}_V) (\sigma\otimes\mathrm{id}_V\otimes \mathrm{id}_V)
\end{equation*}    
and call it {\em $R$-matrix rotation by} $90^\circ $. We will also use the abbreviated notation
\begin{align*}
     \R^{*_1}(u)=\omega^{(23)}\R^{(12)}(\lambda -u)\sigma^{(01)}.
\end{align*}
\begin{definition}
\label{Rrs}
 A {\em rotationally invariant $R$-matrix} is an $R$-matrix $\R (u)$ such that $\R^{*_1}(u)=\R(u)f(u)$ for some function $f$ of $u\in\mathbb{C}$.
\end{definition}
\begin{prop}
    The $R$-matrix $\R(u)$ that satisfies (\ref{oRR}) is a rotationally invariant $R$-matrix.
\end{prop}
\begin{proof}
    The second relation in (\ref{oRR}) is equivalent to 
    \begin{align*}
        \R^{(12)}(-u)\sigma^{(23)}\alpha(u)=\R^{(23)}(u+\lambda)\sigma^{(12)},
    \end{align*}
   due to the inverse relation (\ref{inv}). This, together with (\ref{os}) gives
    \begin{align*}
        \R^{*_1}(u)=\omega^{(23)}\R^{(12)}(\lambda -u)\sigma^{(01)}=\omega^{(23)}\R^{(01)}(u)\sigma^{(12)}\alpha(-u)\\
        =\omega^{(23)}\sigma^{(12)}\R^{(01)}(u)\alpha(-u)=\R(u)\alpha(-u).
    \end{align*}
\end{proof}


%

When treating the general linear case, we adopt a modified definition of the rotations of $R$-matrix: instead
of $V$ we take the direct sum of $V\oplus V^*$ and consider an inner product of the form
$\genfrac(){0pt}{}{0\;\omega}{\omega^*\;0}$
We equip $V\oplus V^*$ with a nondegenerate bilinear map $\omega:V\otimes V^*\to \mathbf 1$
and its dual map $\sigma:\mathbf 1\to V^*\otimes V $ that satisfies 
\begin{equation}\label{sigmaii}
    (\mathrm{id}_{V^*}\otimes\omega)(\sigma\otimes\mathrm{id}_{V^*})=\mathrm{id}_{V^*},\qquad (\omega\otimes\mathrm{id}_{V})(\mathrm{id}_{V}\otimes\sigma)=\mathrm{id}_{V}
\end{equation}
We also introduce a map $\omega^*:V^*\otimes V\to\mathbf 1$ whose coefficients $\omega^*_{-i,i}:=\omega^*|_{V^*_{-i}\otimes V_i}$ are of the form
\begin{equation}\label{omegastar}
\omega^*_{-i,i}=\frac{G(s(i))}{G(t(i))}\omega_{i,-i}\circ P_{-i,i},
\end{equation}
where $G:\pi_0\to\mathbb{C}$ is a function on the objects of groupoid $\pi$, and $P_{-i,i}:V^*_{i^{-1}}\otimes V_i\to V_i\otimes V^*_{i^{-1}}$ denotes the flip $v\otimes w\mapsto w\otimes v$, $v\in V^*_{i^{-1}}$, $w\in V_i$. Let $\sigma^*:\mathbf{1}\to V\otimes V^*$ be its dual map satisfying
\begin{equation}\label{sigmaiistar}
    (\mathrm{id}_V\otimes\omega^*)(\sigma^*\otimes\mathrm{id}_V)=\mathrm{id}_V,\quad (\omega^{*}\otimes\mathrm{id}_{V^*})(\mathrm{id}_{V^*}\otimes\sigma^*)=\mathrm{id}_{V^*}.
\end{equation}
From this and (\ref{omegastar}) it follows that the coefficient $\sigma^{*i,-i}$ is of the form
\begin{equation}\label{sigmastar}
    \sigma^{*i,-i}=\frac{G(t(i))}{G(s(i))}P_{i,-i}\circ\sigma^{-i,i}.
\end{equation}
\begin{definition}
    \label{Rr2}
    Let $\R_{VV}\in\operatorname{End}(V\otimes V)$ be an $R$-matrix. The operators $\R_{V^*V}(u)\in\operatorname{Hom}(V^*\otimes V, V\otimes V^*)$, $\R_{V^*V^*}(u)\in\operatorname{Hom}(V^*\otimes V^*,V^*\otimes V^*)$ and $\R_{VV^*}(u)\in\operatorname{Hom}(V\otimes V^*, V^*\otimes V)$, given by 
\begin{equation*}
    \begin{split}
         &\R_{V^*V}(u):=(\omega^*\otimes\mathrm{id}_V\otimes\mathrm{id}_{V^*})\circ(\mathrm{id}_{V^*}\otimes\R_{VV}(\lambda-u)\otimes \mathrm{id}_{V^*})\circ (\mathrm{id}_{V^*}\otimes \mathrm{id}_V\otimes\sigma^*),\\
          &\R_{VV^*}(u):=(\mathrm{id}_{V^*}\otimes\mathrm{id}_{V}\otimes\omega)(\mathrm{id}_{V}\otimes\R_{VV}(\lambda-u)\otimes \mathrm{id}_{V}) (\sigma\otimes\mathrm{id}_V\otimes \mathrm{id}_{V^*}), \\
          &\R_{V^*V^*}(u):=(\mathrm{id}_{V^*}\otimes\mathrm{id}_{V^*}\otimes\omega)(\mathrm{id}_{V^*}\otimes\R_{VV^*}(\lambda-u)\otimes \mathrm{id}_V) (\sigma\otimes\mathrm{id}_{V^*}\otimes \mathrm{id}_{V^*}),
    \end{split}
\end{equation*}
are called {\em rotations by $-90^\circ$, $90^\circ$ and $180^\circ$ of matrix $\R_{VV}(u)$}, respectively.
\end{definition}
We will denote the components of these matrices by $(\R_{V^*V})^{k,-l}_{-i,j}(u)\in\operatorname{Hom}(V^*_{i^{-1}}\otimes V_j, V_k\otimes V^*_{l^{-1}})$, $(\R_{VV^*})^{-k,l}_{i,-j}(u)\in\operatorname{Hom}(V_i\otimes V^*_{j^{-1}}, V^*_{k^{-1}}\otimes V_{l})$ and $(\R_{V^*V^*})^{-k,-l}_{-i,-j}(u)\in\operatorname{Hom}(V^*_{i^{-1}}\otimes V^*_{j^{-1}}, V^*_{k^{-1}}\otimes V^*_{l^{-1}})$. The rotation relations now can be written as
\begin{equation}\label{Rrel}
    \begin{split}
     (\R_{V^*V})^{k,-l}_{-i,j}(u)&=\omega^*_{-i,i}(\R_{VV})_{j,l}^{i,k}(\lambda -u)\sigma^{*l,-l},\\
    (\R_{VV^*})^{-k,l}_{i,-j}(u)&=\sigma^{-k,k}(\R_{VV})_{k,i}^{l,j}(\lambda -u)\omega_{j,-j},\\
    (\R_{V^*V^*})^{-k,-l}_{-i,-j}(u)&=\sigma^{-k,k}(\R_{VV^*})_{k,-i}^{-l,j}(\lambda -u)\omega_{j,-j}.
    \end{split}
\end{equation}
Equivalently, one can define the rotation of $\R_{VV}(u)$ by $180^\circ$ as 
\begin{align*}
\R_{V^*V^*}'(u):=(\omega^*\otimes\mathrm{id}_{V^*}\otimes\mathrm{id}_{V^*})(\mathrm{id}_{V^*}\otimes\R_{V^*V}(\lambda-u)\otimes \mathrm{id}_V) (\mathrm{id}_{V^*}\otimes \mathrm{id}_{V^*}\otimes\sigma^*),
\end{align*}
which in the components takes the following form
\begin{equation}\label{Rprime}
     (\R_{V^*V^*}')^{-k,-l}_{-i,-j}(u)=\omega^*_{-i,i}(\R_{V^*V})_{-j,l}^{i,-k}(\lambda -u)\sigma^{*l,-l}.
\end{equation}
\begin{lemma}\label{rotpi}
 The rotation of $\R_{VV}(u)$ by $-180^\circ$ is the same as the rotation by $180^\circ$, i.e., $\R_{V^*V^*}'(u)=\R_{V^*V^*}(u)$.
\end{lemma}
\begin{proof}
   From (\ref{Rrel}) and (\ref{Rprime}) we have
    \begin{align*}
        (\R_{V^*V^*}')^{-k,-l}_{-i,-j}(u)=\omega^*_{-i,i}\omega^*_{-j,j}(\R_{VV})_{l,k}^{j,i}(u)\sigma^{*k,-k}\sigma^{*l,-l}
    \end{align*}
    and
\begin{align*}
     (\R_{V^*V^*})^{-k,-l}_{-i,-j}(u)&=\sigma^{-k,k}\sigma^{-l,l}(\R_{VV})_{l,k}^{j,i}(u)\omega_{i,-i}\omega_{j,-j}.
\end{align*}
Using (\ref{omegastar}) and (\ref{sigmastar}) the first equation becomes
\begin{align*}
   (\R_{V^*V^*}')^{-k,-l}_{-i,-j}(u)=& \tfrac{G(t(k))G(t(l))}{G(s(k))G(s(l))}\sigma^{-k,k}\sigma^{-l,l}(\R_{VV})_{l,k}^{j,i}(u)\omega_{i,-i}\omega_{j,-j}\tfrac{G(s(i))G(s(j))}{G(t(i))G(t(j))}\\
   =&\tfrac{G(t(k))G(t(l))}{G(s(k))G(s(l))}(\R_{V^*V^*})^{-k,-l}_{-i,-j}(u)\tfrac{G(s(i))G(s(j))}{G(t(i))G(t(j))}.
\end{align*}
But $(\R_{VV})^{l,k}_{j,i}(u)$ is graded by the component \begin{tikzcd}
a \arrow[r, "l"] \arrow[d, "j"']
& b \arrow[d, "k"] \\
c \arrow[r, "i"']
& d 
\end{tikzcd} with $i\circ j=k\circ l$, which gives
\begin{align*}
    t(k)=t(i),\quad s(l)=s(j),\quad s(k)=t(l)\quad s(i)=t(j).
\end{align*}
This proves the statement.
\end{proof}
\begin{cor}\label{rotcor}
    The rotation of the operator $\R_{VV^*}(u)$ by $-180^\circ$ and by $180^\circ$ is the same, i.e.,
    \begin{align*}
        \omega^*_{-i,i}\omega_{j,-j}(\R_{VV^*})_{l,-k}^{-j,i}(u)\sigma^{-k,k}\sigma^{*l,-l}=\sigma^{*k,-k}\sigma^{-l,l}(\R_{VV^*})_{l,-k}^{-j,i}(u)\omega_{i,-i}\omega^*_{-j,j},
    \end{align*}
    and it is equal to the operator $ (\R_{V^*V})^{k,-l}_{-i,j}(u)$.
\end{cor}
\begin{proof}
Using the defining relation for $\sigma$ (\ref{sigmaii}), the second relation in (\ref{Rrel}) becomes
\begin{align*}
    (\R_{VV})_{k,i}^{l,j}(\lambda -u)=\omega_{k,-k}(\R_{VV^*})^{-k,l}_{i,-j}(u)\sigma^{-j,j},
\end{align*}
which together with the first relation in (\ref{Rrel}) gives
\begin{align*}
    (\R_{V^*V})^{k,-l}_{-i,j}(u)=\omega^*_{-i,i}\omega_{j,-j}(\R_{VV^*})^{-j,i}_{l,-k}(u)\sigma^{-k,k}\sigma^{*l,-l}.
\end{align*}
On the other hand, the relation (\ref{Rprime}) can be written as 
\begin{align*}
    (\R_{V^*V})^{k,-l}_{-i,j}(u)=\sigma^{*k,-k}(\R_{V^*V^*}')_{-k,-i}^{-l,-j}(\lambda -u)\omega^*_{-j,j},
\end{align*}
which, due to Lemma \ref{rotpi} and the third relation in (\ref{Rrel}) becomes
\begin{align*}
    (\R_{V^*V})^{k,-l}_{-i,j}(u)=\sigma^{*k,-k}\sigma^{-l,l}(\R_{VV^*})_{l,-k}^{-j,i}(u)\omega_{i,-i}\omega^*_{-j,j}.
\end{align*}
\end{proof}
\begin{lemma}\label{yb*}
Rotations satisfy the following Yang--Baxter equations
    \begin{align*}
      \R^{(23)}_{V^*V}(u_1-u_2)\R^{(12)}_{V^*V}(u_1)\R^{(23)}_{VV}(u_2)&= \R^{(12)}_{VV}(u_2)\R^{(23)}_{V^*V}(u_1)\R^{(12)}_{V^*V}(u_1-u_2),\\
 \R^{(23)}_{V^*V^*}(u_1-u_2)\R^{(12)}_{V^*V}(u_1)\R^{(23)}_{V^*V}(u_2)&= \R^{(12)}_{V^*V}(u_2)\R^{(23)}_{V^*V}(u_1)\R^{(12)}_{V^*V^*}(u_1-u_2).
    \end{align*}
\end{lemma}
\begin{proof}
 Using the definition \ref{Rr2} and writing the equations in the component form,
 \begin{align*}
      &\omega^*_{-i,i}(\R_{VV})_{l,n}^{i,p}(\lambda-u_1)(\R_{VV})^{n,o}_{m,r}(\lambda-u_1+u_2)(\R_{VV})_{j,k}^{l,m}(u_2)\sigma^{*r,-r}\\
      &= \omega^*_{-i,i}   (\R_{VV})_{l,n}^{p,o}(u_2)(\R_{VV})_{j,m}^{i,l}(\lambda-u_1+u_2)(\R_{VV})_{k,r}^{m,n}(\lambda-u_1)\sigma^{*r,-r},\\
 &\omega^*_{-i,i} \omega^*_{-j,j} (\R_{VV})_{r,o}^{m,p}(u_1-u_2)(\R_{VV})_{k,m}^{l,j}(\lambda-u_2)(\R_{VV})_{l,p}^{i,n}(\lambda -u_1)\sigma^{*o,-o}\sigma^{*r,-r}\\
 &= \omega^*_{-i,i} \omega^*_{-j,j}(\R_{VV})_{k,r}^{m,p}(\lambda-u_1)(\R_{VV})_{p,o}^{l,n}(\lambda-u_2)(\R_{VV})_{m,l}^{i,j}(u_1-u_2)\sigma^{*o,-o}\sigma^{*r,-r},
 \end{align*}
 it is obvious that the relations are satisfied due to the Yang--Baxter equation for $\R_{VV}(u)$.
\end{proof}
\begin{definition}\label{def-crossing-gl}
  Let us call an $R$-matrix $\R_{VV}(u)$ {\em crossing-symmetric} with crossing parameter $\lambda\in\mathbb C$ if
\begin{equation}\label{qu1}
\omega^{(23)}\R_{VV}^{(12)}(u+\lambda)\R^{(23)}_{V^*V}(u)=\alpha(u)\omega^{(12)},
\end{equation}
\begin{equation}\label{qu2}
\R^{(12)}_{V^*V}(u)\R^{(23)}_{VV}(u+\lambda)\sigma^{(12)}=\alpha(u)\sigma^{(23)},
 \end{equation}
 for some meromorphic function $\alpha$ of $u\in\mathbb C$.
\end{definition}
\begin{remark} This definition is a $\pi$-graded analogue of the definition of \cite{ReshetikhinSemenovTianShansky1990}.
\end{remark}
\begin{lemma}
\label{inversion}
    If $\R_{VV}(u)$ is crossing symmetric, then we have inversion identities
    \begin{align*}
        \R_{VV^*}(-u)\R_{V^*V}(u)=\alpha(u),\qquad \R_{V^*V}(u)\R_{VV^*}(-u)=\alpha(u).
    \end{align*}
\end{lemma}
\begin{proof}
   By acting with $\sigma^{(01)}$ on the right side of the identity (\ref{qu1}) and using the definition of $\R_{V^*V}(u)$ and (\ref{sigmaii}), we obtain the first inversion relation. Similarly, acting with $\omega^{(34)}$ on the left of the identity (\ref{qu2}) and using the definition of $\R_{VV^*}(u)$, we get the other inversion relation.
\end{proof}

\section{Jimbo--Miwa--Okado RSOS models}\label{sec-4}
In this section, we consider a face model on a two dimensional square lattice built upon the affine Lie algebra $A_n^{(1)}$, $B_n^{(1)}$, $C_n^{(1)}$ or $D_n^{(1)}$ introduced by Jimbo, Miwa and Okado \cite{JimboMiwaOkado1988}. They found solutions of the star-triangle relation

\begin{equation}
\label{eqn:str}
\begin{split}
\sum_g \W\Big(\begin{matrix}
  f & g\\
  e & d
\end{matrix}\Big\vert u_2\Big)
\W\Big(\begin{matrix}
  b & c\\
  g & d
\end{matrix}\Big\vert u_1-u_2\Big)
\W\Big(\begin{matrix}
  a & b\\
  f & g
\end{matrix}\Big\vert u_1\Big)\\
=\sum_g \W\Big(\begin{matrix}
  a & b\\
  g & c
\end{matrix}\Big\vert u_2\Big)
\W\Big(\begin{matrix}
  a & g\\
  f & e
\end{matrix}\Big\vert u_1-u_2\Big)
\W\Big(\begin{matrix}
  g & c\\
  e & d
\end{matrix}\Big\vert u_1\Big).
\end{split}
\end{equation}
where the fluctuation variables $a, b,\dots$ are dominant affine weights. As observed in \cite{FelderICMP1995}, the star-triangle relation can be translated into a dynamical Yang--Baxter equation. We start with a review of the corresponding $R$-matrices.

The non-restricted case is revisited in Appendix A, where we establish several useful results, and we adapt them to the restricted case in the rest of this section. 

In the course of this, we also define rotations of the $R$-matrix appropriate for the $\mathfrak{gl}_n$ case, starting from the $A_{n-1}^{(1)}$-type $R$-matrix. The $R$-matrices constructed for the $B^{(1)}_n$, $C_n^{(1)}$, and $D_n^{(1}$ cases provide examples of rotationally invariant $R$-matrices as defined in the previous section.

\subsection{Unrestricted \texorpdfstring{$R$}{R}-matrix}\label{4.1} The coefficients of the Jimbo--Miwa--Okado Boltzmann weights depend on two parameters $\tau, L\in\mathbb{C}$ such that $\Im(\tau)>0$. Let $$[u]=\theta(\tfrac{1}{L} u,\tau)/(\tfrac{1}{L}\theta'(0,\tau)),$$
where
\[
  \theta(u,\tau)=-\sum_{n\in\mathbb Z}
  e^{i\pi (n+\frac12)^2\tau+2\pi i (n+\frac12) (u+\frac12)}
\]
is the odd Jacobi theta function. The function $f(u)=[u]$ is an odd entire function with first order zeros on the lattice $\Lambda=\textstyle \mathbb ZL+\mathbb Z L\tau$ and is normalized to have derivative $1$ at $u=0$. By the Jacobi triple product identity,
\begin{equation}\label{e-triple}
[u]=\frac L\pi\sin\frac{\pi u}L\prod_{m=1}^\infty
\frac{\left(1-e^{2\pi im\tau+2\pi i\frac uL}\right)\left(1-e^{2\pi im\tau-2\pi i\frac uL}\right)}
{\left(1-e^{2\pi im\tau}\right)^2}.
\end{equation}

Let $\mathfrak{g}$ denote one of the classical Lie algebras $A_{n-1}=\mathfrak{sl}_n$, $B_n=\mathfrak{so}_{2n+1}$, $C_n=\mathfrak{sp}_{2n}$ or $D_n=\mathfrak{so}_{2n}$ and $(\rho,\bar V=\mathbb{C}^m)$ its vector representation, where $m=n$, $2n+1$, $2n$ and $2n$ for $A_{n-1}$, $B_n$, $C_n$ and $D_n$, respectively. We realize the dual $\mathfrak h^*$ of the Cartan subalgebra as $\mathbb R^n$ for $B_n,C_n,D_n$ and as the subspace of $\mathbb R^n$ of vectors with zero coordinate sum for $A_{n-1}$. Let $\bar\varepsilon_{i}$ be the standard basis of $\mathbb R^n$. Let $\bar\varepsilon :=\sum_{i=1}^{n}\bar\varepsilon_i$ for $A_{n-1}$ and $\bar\varepsilon :=0$ in the other cases. We set $\varepsilon_i:=\bar\varepsilon_i-\frac{1}{n}\bar\varepsilon$, $\varepsilon_{-i}:=-\varepsilon_{i}$ and $\varepsilon_0:=0$. The defining representation $\bar V$ of $\mathfrak{g}$ has the following weight decomposition
\begin{align*}
\bar V=
\begin{cases}
\oplus_{i=1}^{n}\bar V_{\varepsilon_i}\textrm{ for }A_{n-1},\\
\oplus_{i=-n}^{n}\bar V_{\varepsilon_i}\textrm{ for }B_n,\\
\oplus_{\substack{i=-n\\ i\neq 0}}^{n}\bar V_{\varepsilon_i}\textrm{ for }C_n,D_n,
\end{cases}
\end{align*}
where the weight spaces are 1-dimensional $\bar V_{\varepsilon_i}=\mathbb{C} e_i$. The weights of $\bar V$ span $\mathfrak h^*$. The inner product for which the $\bar\varepsilon_i$ are orthonormal is the Weyl group invariant inner product normalized so that the long roots have norm $\sqrt2$ for $A_{n-1},B_{n},D_{n}$ and $2$ for $C_n$, as in Bourbaki.

For $a\in \mathfrak h^*_\mathbb C$ we write $a=\sum_{i=1}^n a_i\bar\varepsilon_i$, where $\sum_{i=1}^n a_i=0$ in the case of $A_{n-1}$. We also set 
\[
\text{$a_{-i}=-a_{i}\quad (i=1,\dots, n)$, for $B_n,C_n,D_n$ and  $a_0=-\tfrac12$ for $B_n$.}
\]




Let $E_{ij}$ be a matrix such that $E_{ij}e_k=\delta_{jk}e_i$. Define the (unnormalized) elliptic dynamical $R$-matrix with spectral parameter $u\in\mathbb C$ as 
\begin{equation}
\label{eqn:RM}
\R(u,a):=\sum_{\substack{i,j,k,l \\ \varepsilon_i+\varepsilon_j=\varepsilon_k+\varepsilon_l}}\R^{k,l}_{i,j}(u,a)E_{ki}\otimes E_{lj},
\end{equation}
where we slightly abuse notation by using letters $i,j,k,l$ to denote indices of the set $\{\varepsilon_i\}_i$ rather than morphisms of the groupoid, which in this case is the transformation groupoid $\mathfrak h^*\rtimes P$ of the weight lattice $P$ acting on $\mathfrak h^*$, consisting of pairs $(a,\mu)\in \mathfrak h^*\times P$ with composition $(a_1,\mu_1)\circ(a_2,\mu_2)=(a_2,\mu_1+\mu_2) $ defined if $a_1=a_2+\mu_2$. Moreover, we let $\R(u,a)|_{\bar V_{\varepsilon_i}\otimes \bar V_{\varepsilon_j}}=\oplus_{k,l}\R^{\varepsilon_k,\varepsilon_l}_{\varepsilon_i,\varepsilon_j}(u,a):=\sum_{k,l}\R^{k,l}_{i,j}(u,a)E_{ki}\otimes E_{lj}$. The $R$-matrix component is graphically represented by a diagram with all four arrows fixed 
\begin{center}
$\R^{k,l}_{i,j}(u,a)=$\begin{tikzcd}
a \arrow[r, "\varepsilon_k"] \arrow[d, "\varepsilon_i"']\arrow[rd,phantom,"\R(u)"]
& b \arrow[d, "\varepsilon_l"] \\
c \arrow[r, "\varepsilon_j"']
& d 
\end{tikzcd}.
\end{center}
We distinguish two cases:
\begin{enumerate}
  \renewcommand{\labelenumi}{(\roman{enumi})}
    \item the $R$-matrix of type $A^{(1)}_{n-1}$;
    \item the $R$-matrices of types $B^{(1)}_n$, $C^{(1)}_n$ and $D^{(1)}_n$.
\end{enumerate}

In the $A_{n-1}^{(1)}$ case, the matrix element 
\begin{equation}
\label{RAn}
(\R_{VV})^{k,l}_{i,j}(u,a)=\mathcal{W}
\Big(\begin{matrix}
  a & a+\varepsilon_k\\
  a+\varepsilon_i & a+\varepsilon_i+\varepsilon_j
\end{matrix}\Big\vert u\Big)
\end{equation}
is such that the indices are positive, $i,j,k,l\in\{1,...,n\}$, and is given by the Jimbo--Miwa--Okado Boltzmann weights $\mathcal{W}$ described by the diagrams (1)-(3) below. 

The $R$-matrix element $\R^{k,l}_{i,j}(u,a)$ in the $B_n^{(1)}$, $C_n^{(1)}$ and $D_n^{(1)}$ cases is 
\begin{center}
$\R^{k,l}_{i,j}(u,a)=\mathcal{W}
\Big(\begin{matrix}
  a & a+\varepsilon_k\\
  a+\varepsilon_i & a+\varepsilon_i+\varepsilon_j
\end{matrix}\Big\vert u\Big),$
\end{center}
with Jimbo--Miwa--Okado Boltzmann weights $\mathcal{W}$ given by diagrams (1)-(6). The indices $i,j$ range over $-n,...,n$ for $B_n^{(1)}$ and over $-n,...,n$ excluding $0$ for $C_n^{(1)}$ and $D_n^{(1)}$. 

The non-vanishing $R$-matrix elements (Jimbo--Miwa--Okado Boltzmann weights) are of the following
form\footnote{We use a normalized version of the Boltzmann weights of \cite{JimboMiwaOkado1988} so that the
  coefficients (1) are equal to 1}
\\
$(1) $\begin{tikzcd}
a \arrow[r, "\varepsilon_i"] \arrow[d, "\varepsilon_i"']\arrow[rd,phantom,"\R(u)"]
& b \arrow[d, "\varepsilon_i"] \\
b \arrow[r, "\varepsilon_i"']
& d 
\end{tikzcd}$=1\textrm{  }(i\neq 0),$\\
$(2) $\begin{tikzcd}
a \arrow[r, "\varepsilon_i"] \arrow[d, "\varepsilon_i"']\arrow[rd,phantom,"\R(u)"]
& b \arrow[d, "\varepsilon_j"] \\
b \arrow[r, "\varepsilon_j"']
& d 
\end{tikzcd}$=\displaystyle\frac{[1][a_{i}-a_{j}-u]}{[1+u][a_{i}-a_{j}]}\textrm{  } (i\neq \pm j),$\\
$(3) $\begin{tikzcd}
a \arrow[r, "\varepsilon_j"] \arrow[d, "\varepsilon_i"']\arrow[rd,phantom,"\R(u)"]
& b \arrow[d, "\varepsilon_i"] \\
c \arrow[r, "\varepsilon_j"']
& d 
\end{tikzcd}$=\displaystyle\frac{[u]}{[1+u]}\sqrt{\frac{[a_i-a_{j}+1][a_i-a_j-1]}{[a_i-a_j]^2}}\textrm{  } (i\neq \pm j),$\\
$(4) $\begin{tikzcd}
a \arrow[r, "\varepsilon_j"] \arrow[d, "\varepsilon_i"']\arrow[rd,phantom,"\R(u)"]
& b \arrow[d, "-\varepsilon_j"] \\
c \arrow[r, "-\varepsilon_i"']
& a 
\end{tikzcd}$=\displaystyle\frac{[u][1][a_{i}+a_{j}+1+\lambda-u]}{[\lambda -u][1+u][a_{i}+a_{j}+1]}\sqrt{G_{a,i}G_{a,j}}\textrm{  } (i\neq j),$\\
$(5) $\begin{tikzcd}
a \arrow[r, "\varepsilon_i"] \arrow[d, "\varepsilon_i"']\arrow[rd,phantom,"\R(u)"]
& b \arrow[d, "-\varepsilon_i"] \\
b \arrow[r, "-\varepsilon_i"']
& a
\end{tikzcd}$=\displaystyle\frac{[1][2a_i+1-u]}{[1+u][2a_{i}+1]}
\!+\!\displaystyle\frac{[1][u][2a_i+1+\lambda-u]}{[\lambda -u][1+u][2a_{i}+1]}G_{a,i}\quad (i\neq0),$\\
$(6) $\begin{tikzcd}
a \arrow[r, "\varepsilon_0"] \arrow[d, "\varepsilon_0"']\arrow[rd,phantom,"\R(u)"]
& a \arrow[d, "\varepsilon_0"] \\
a \arrow[r, "\varepsilon_0"']
& a
\end{tikzcd}$=\displaystyle\frac{[\lambda+u][1][2\lambda-u]}{[\lambda-u][1+u][2\lambda]}
-\displaystyle\frac{[u][1]}{[1+u][2\lambda]}\sum_{j\neq0}\frac{[a_j+\frac12+2\lambda]}{[a_j+\frac12]}.$
\\
The factors in the formulas above are given by
\begin{align}\label{e-Ga}
G_a=&
\begin{cases}
\epsilon(a)\prod_{1\leq i <j\leq n}[a_{i}-a_j]\qquad\textrm{for }A_{n-1}^{(1)},\\
\epsilon(a)\prod_{i=1}^n h(a_{i})\prod_{1\leq i<j\leq n}[a_{i}-a_j][a_{i}+a_{j}] \qquad\textrm{for }B_n^{(1)},C_n^{(1)},D_n^{(1)},
\end{cases}\\
G_{a,i}=&\frac{G_{a+\varepsilon_i}}{G_a}=
\begin{cases}
    &1\qquad (i=0),\\
    &\sigma\frac{h(a_{i}+1)}{h(a_{i})}\prod_{j\neq\pm i,0}\frac{[a_{i}-a_{j}+1]}{[a_{i}-a_{j}]} \qquad(i\neq 0) 
\end{cases}\notag
\end{align}
where $\lambda$ is {\em the crossing parameter} given in the Table \ref{Tab:t1} together with the function $h(a)$ and the sign $\sigma$. The sign factor $\epsilon (a)$ is such that $\epsilon(a+\varepsilon_i)/\epsilon(a)=\sigma$. It is given by $\epsilon(a)=1$ in $A_{n-1}^{(1)}$, $B_n^{(1)}$ and $D_n^{(1)}$ cases, and in $C_n^{(1)}$ case it is given by 
\[
\epsilon(a)=
\begin{cases}
    1\qquad\text{$a$ even,}\\
    -1\qquad\text{$a$ odd,}
\end{cases}
\]
where $a$ is called even if $\sum_{i=1}^na_i$ is even, otherwise is odd.
\begin{table}[ht]
\caption{Factors and functions}
\label{tab}
\centering
\begin{tabular}{ |c|c|c|c|c| } 
 \hline
 Type & $A_{n-1}^{(1)}$ & $B_n^{(1)}$ & $C_n^{(1)}$ & $D_n^{(1)}$ \\ 
 \hline 
 $\lambda$ & $-\frac{n}{2}$ & $-n+\frac{1}{2}$ & $-n-1$ & $-n+1$ \\ 
  \hline
 $h(a)$ & $1$ & $[a]$ & $[2a]$ & $1$ \\
  \hline
 $\sigma$ & $1$ & $1$ & $-1$ & $1$ \\
 \hline
\end{tabular}
\label{Tab:t1}
\end{table} 

\begin{remark}\label{r-1} The function $x\mapsto [x]$ takes real values if $\tau$
  is purely imaginary and $L>0$, and its restriction to the interval
  $0<x<L$ is positive. There is then an open subset of
  $\mathfrak h^*$, which includes the affine alcove of the next
  section, on which the arguments of the square roots are positive,
  and where the square roots are defined to be positive. On this
  domain the star-triangle relation is verified.  The Boltzmann
  weights have then an analytic continuation to 2-valued functions on
  the complexification $\mathfrak h^*$. The star-triangle relation
  continues to hold by analytic continuation.
\end{remark}

\subsection{Restricted model: the setting}\label{4.2}
Let $\mathfrak{g}$ be one of the classical Lie algebras $A_{n-1}$
($n\geq 2$), $B_n$ ($n\geq 2$), $C_n$ ($n\geq 1$) or $D_n$
($n\geq 3$) and let $g=n,2n-1,n+1,2n$, respectively, be their dual
Coxeter numbers. Let $\alpha_i\in\mathfrak h^*$ be a system of roots and $\alpha_i^\vee\in\mathfrak h$ the
corresponding coroots. Denote by $\theta$ the highest root and $\rho$ the
half-sum of positive roots. The crossing parameter is
\begin{equation}\label{e-crossing}
  \lambda=-g(\theta,\theta)/4.
\end{equation}
In the restricted model the dynamical
variables are in the subset $P_{++}^{\ell+g}$ of the weight lattice $P$ consisting of regular dominant affine weights of level
$\ell+g$, namely
\[
  P_{++}^{\ell+g}=\{a\in P\,|\, \langle a,\alpha_i^\vee\rangle>0,\quad \langle a,\theta^\vee\rangle<\ell+g\},
\]
where $\theta^\vee$ is the coroot corresponding to $\theta$.
They are shifts by $\rho$ of the classical part of the dominant weights of level $\ell$
of the corresponding untwisted affine Kac--Moody algebra
$\mathfrak g^{(1)}$, which are highest weights of integrable representations.
The parameter $L$ is
\begin{equation}\label{e-L}
  L=\frac{(\theta,\theta)}2(\ell+g)
\end{equation}
It is an integer (even in the $C_n$ case) such that $L+2\lambda\geq 0$.
In terms of the invariant inner product $(\ ,\ )$ on $\mathfrak h^*$, the inequalities defining $P_{++}^{\ell+g}$ are
\[
  (a,\alpha_i)>0,\quad (a,\theta)<L.
\]

%
%
%
Explicitly, 
\begin{align*}
P_{++}^{\ell+g}=
\begin{cases}
\{a\in P_{++}
  \,|\, a_1-a_{n}< L\}&\textrm{ for }A_{n-1}^{(1)},\\
\{a\in P_{++}
  \,|\,a_1+a_2< L\}&\textrm{ for }B_n^{(1)},\\
\{a\in P_{++}
  \,|\, a_1< L/2\}&\textrm{ for }C_n^{(1)},\\
  \{a\in P_{++}
  \,|\,a_1+a_2< L\}&\textrm{ for }D_n^{(1)},
\end{cases}
\end{align*}
where $P_{++}$ is the set of regular dominant weights,
\begin{align*}
P_{++}=
\begin{cases}
\{a\in P
  \,|\, a_1>a_2> \cdots > a_{n}\}&A_{n-1},\\
\{a\in P
  \,|\,a_1 >a_2>\cdots > a_n, a_n> 0\}&B_n,\\
\{a\in P
  \,|\, a_1> a_2>\cdots > a_n,a_n> 0\}&C_n,\\
  \{a\in P
  \,|\,a_1> a_2>\cdots > a_n, a_{n-1}+a_n>0\}&D_n,\\
\end{cases}
\end{align*}
and $P$ is the weight lattice,
\begin{align*}
P=
\begin{cases}
\{a\in \left(\tfrac{1}{n}\mathbb Z\right)^{n}\,|\,\forall i,j: a_i-a_j\in\mathbb{Z},\sum_{i=1}^na_i=0\}&A_{n-1},\\
\{a\in \left(\tfrac12\mathbb Z\right)^n\,|\,\forall i,j: a_i-a_j\in\mathbb{Z}\}&B_n,\\
\{a\in \mathbb Z^{n}\}&C_n,\\
  \{a\in \left(\tfrac12\mathbb Z\right)^n\,|\,\forall i,j: a_i-a_j\in\mathbb{Z}\}&D_n.\\
\end{cases}
\end{align*}

Let $\pi$ be the full subgroupoid of the transformation groupoid $P\rtimes P$ on $P^{\ell+g}_{++}$. It has $P_{++}^{\ell+g}$ as its set of objects, whereas its set of morphisms consists of pairs $(a,\mu)$ of weights, with the source $a$, the target $a+\mu$, such that $a, a+\mu\in P^{\ell+g}_{++}$, and the composition law $(a_1,\mu_1)\circ(a_2,\mu_2)=(a_2,\mu_1+\mu_2)$ if $a_1=a_2+\mu_2$. 

We consider the $\pi$-graded space $V^{\textrm{RSOS}}=\bigoplus_{(a,\mu)\in\pi}V^{\textrm{RSOS}}_{(a,\mu)}$ with
\begin{align*}
V^{\textrm{RSOS}}_{(a,\mu)}=
\begin{cases}
\bar V_{\varepsilon_i}=\mathbb C e_i &\text{if $\mu=\varepsilon_i$, $i\neq0$, and $a,a+\varepsilon_i\in P^{\ell+g}_{++}$,}\\
\bar V_0=\mathbb C e_0 &\text{in the $B_n^{(1)}$ case if $\mu=\varepsilon_0$, $a\in P^{\ell+g}_{++}$ and $a_n\neq\frac12$,}\\
0 &\text{otherwise.}
\end{cases}
\end{align*}
In the $A_{n-1}^{(1)}$ case we also act on the dual space, whose restriction $(V^\ast)^{\textrm{RSOS}}=\bigoplus_{(a,\mu)\in\pi}(V^\ast)^{\textrm{RSOS}}_{(a,\mu)}$ is
\begin{align*}
    (V^*)^{\textrm{RSOS}}_{(a,\mu)}=
    \begin{cases}
        \bar V_{\varepsilon_{-i}}=\mathbb C e_{-i} &\text{if $\mu=\varepsilon_{-i}$ and $a,a+\varepsilon_{-i}\in P^l_{++}$,}\\
        0 &\text{otherwise,}
    \end{cases}
\end{align*}
where $(e_{-1},...,e_{-n})$ denotes the basis of $\bar V^*$ dual to the basis $(e_1,\dots, e_n)$ of $\bar V$.

\subsection{Restricted \texorpdfstring{$R$}{R}-matrix and rotations}\label{4.3}
We define the restricted $R$-matrix to be the one given in Section \ref{4.1}, such that the matrix elements $\R_{i,j}^{k,l}(u)$ now become homomorphisms of $\pi$-graded vector spaces through the action of $E_{ki}\otimes E_{lj}$ that sends $V^{\textrm{RSOS}}_{(a,\varepsilon_i)}\otimes V^{\textrm{RSOS}}_{(a+\varepsilon_i,\varepsilon_j)}$ to $V^{\textrm{RSOS}}_{(a,\varepsilon_k)}\otimes V^{\textrm{RSOS}}_{(a+\varepsilon_k,\varepsilon_l)}$. 

Let us give the inner product $\omega: V^{\textrm{RSOS}}_{(a,\varepsilon_i)}\otimes V^{\textrm{RSOS}}_{(a+\varepsilon_i,-\varepsilon_i)}\to \mathbb{C}$ and its dual map $\sigma: \mathbb{C}\to V^{\textrm{RSOS}}_{(a,\varepsilon_i)}\otimes V^{\textrm{RSOS}}_{(a+\varepsilon_i,-\varepsilon_i)}$ in $B_n^{(1)}$, $C_n^{(1)}$ and $D_n^{(1)}$ cases as


\begin{align}\label{e-omega}
  \omega(e_i\otimes_v e_{-j})&:=\delta_{\varepsilon_i,\varepsilon_j}\displaystyle\sqrt{\sigma\frac{G_{t((a,\varepsilon_i))}}{G_{s((a,\varepsilon_i))}}},\quad\sigma(1):=\sum_k e_{-k}\otimes_v e_{k}\displaystyle\sqrt{\sigma\frac{G_{t((a,\varepsilon_{-k}))}}{G_{s((a,\varepsilon_{-k}))}}},
\end{align}
where $\sigma$ in the argument of the square root is $1$ for $B^{(1)}_n,D^{(1)}_n$ and $-1$ for $C^{(1)}_n$, see Section \ref{4.1}.
In the case of algebra $A_{n-1}^{(1)}$, we set the maps $\omega$, $\sigma$, $\omega^*$ and $\sigma^*$ to be
\begin{align}\label{e-omega*}
 \omega(e_i\otimes_v e_{-j})&=\delta_{\varepsilon_i,\varepsilon_j}\displaystyle\sqrt{\frac{G_{t((a,\varepsilon_i))}}{G_{s((a,\varepsilon_i))}}},\qquad\sigma(1)=\sum_k e_{-k}\otimes_v e_{k} \displaystyle\sqrt{\frac{G_{t((a,-\varepsilon_k))}}{G_{s((a,-\varepsilon_k))}}}\\
  \omega^*(e_{-i}\otimes_v e_j)&=\delta_{\varepsilon_i,\varepsilon_j}\displaystyle\sqrt{\frac{G_{t((a,-\varepsilon_i))}}{G_{s((a,-\varepsilon_i))}}},\qquad\sigma^*(1)=\sum_k e_k\otimes_v e_{-k} \displaystyle\sqrt{\frac{G_{t((a,\varepsilon_k))}}{G_{s((a,\varepsilon_k))}}}.\notag
\end{align}
Here, $s$ and $t$ are the source and target maps, respectively, with $s((a,\varepsilon_i))=a$ and $t((a,\varepsilon_i))=a+\varepsilon_i$.

 The square roots are defined by analytic continuation from the region where their arguments are positive, as in Remark \ref{r-1}.

We now want to obtain the rotations of the $A_{n-1}^{(1)}$-type $R$-matrix $\R_{VV}(u)$ using Definition \ref{Rr2}. The rotation by $-90^\circ$ is
\begin{equation}
\label{rstar}
\R_{V^*V}(u,a):=\sum_{-\varepsilon_i+\varepsilon_j=\varepsilon_k-\varepsilon_l}(\R_{V^*V})^{k,-l}_{-i,j}(u,a)E_{k,-i}\otimes E_{-l,j},
\end{equation}
with
\[
(\R_{V^*V})^{k,-l}_{-i,j}(u,a):=\sqrt{\frac{G_{a+\varepsilon_k}G_{a-\varepsilon_i}}{G_aG_{a-\varepsilon_i+\varepsilon_j}}}(\R_{VV})_{j,l}^{i,k}(\lambda -u,a-\varepsilon_i),
\]
graphically presented as 
\begin{center}
\begin{tikzcd}
a \arrow[r, "\varepsilon_k"] \arrow[d, "-\varepsilon_i"']\arrow[rd,phantom,"\R_{V^*V}"]
& b \arrow[d, "-\varepsilon_l"] \\
d \arrow[r, "\varepsilon_j"']
& c 
\end{tikzcd}$:=\displaystyle\sqrt{\frac{G_bG_d}{G_aG_c}}$
\begin{tikzcd}
d \arrow[r, "\varepsilon_i"] \arrow[d, "\varepsilon_j"']\arrow[rd,phantom,"\R_{VV}"]
& a \arrow[d, "\varepsilon_k"] \\
c \arrow[r, "\varepsilon_l"']
& b
\end{tikzcd}.
\end{center}
The rotation by $90^\circ$ is
\begin{equation}
\R_{VV^*}(u,a):=\sum_{\varepsilon_i-\varepsilon_j=-\varepsilon_k+\varepsilon_l}(\R_{VV^*})^{-k,l}_{i,-j}(u,a)E_{-k,i}\otimes E_{l,-j},
\end{equation}
with
\begin{align*}
    (\R_{VV^*})^{-k,l}_{i,-j}(u,a):=\sqrt{\frac{G_{a-\varepsilon_k}G_{a+\varepsilon_i}}{G_aG_{a+\varepsilon_i-\varepsilon_j}}}(\R_{VV})_{j,l}^{i,k}(\lambda -u,a-\varepsilon_k),
\end{align*}
depicted as 
\begin{center}
\begin{tikzcd}
a \arrow[r, "-\varepsilon_k"] \arrow[d, "\varepsilon_i"']\arrow[rd,phantom,"\R_{VV^*}"]
& b \arrow[d, "\varepsilon_l"] \\
d \arrow[r, "-\varepsilon_j"']
& c 
\end{tikzcd}$:=\displaystyle\sqrt{\frac{G_bG_d}{G_aG_c}}$
\begin{tikzcd}
b\arrow[r, "\varepsilon_l"] \arrow[d, "\varepsilon_k"']\arrow[rd,phantom,"\R_{VV}"]
& c \arrow[d, "\varepsilon_j"] \\
a \arrow[r, "\varepsilon_i"']
& d
\end{tikzcd}.
\end{center}
Finally, the rotation by $180^\circ$ is
\begin{equation}
\R_{V^*V^*}(u,a):=\sum_{\varepsilon_i+\varepsilon_j=\varepsilon_k+\varepsilon_l}(\R_{V^*V^*})^{-k,-l}_{-i,-j}(u,a)E_{-k,-i}\otimes E_{-l,-j},
\end{equation}
with
\[
    (\R_{V^*V^*})^{-k,-l}_{-i,-j}(u,a):=(\R_{VV})^{j,i}_{l,k}(u,a-\varepsilon_i-\varepsilon_j),
\]
which corresponds graphically to 
\begin{center}
\begin{tikzcd}
a \arrow[r, "-\varepsilon_k"] \arrow[d, "-\varepsilon_i"']\arrow[rd,phantom,"\R_{V^*V^*}"]
& b \arrow[d, "-\varepsilon_l"] \\
d \arrow[r, "-\varepsilon_j"']
& c 
\end{tikzcd}$:=$
\begin{tikzcd}
c \arrow[r, "\varepsilon_j"] \arrow[d, "\varepsilon_l"']\arrow[rd,phantom,"\R_{VV}"]
& d \arrow[d, "\varepsilon_i"] \\
b \arrow[r, "\varepsilon_k"']
& a
\end{tikzcd}.
\end{center}

\subsection{Restricted model: results}
The vector space $V^{\textrm{RSOS}}$ is constructed in such a way that allows only the \textit{admissible pairs of weights}, as defined in \cite{JimboMiwaOkado1988}. Therefore, the condition that the $R$-matrices $\R_{VV}(u)$ and $\R(u)$ are well-defined on the $V^{\textrm{RSOS}}\otimes V^{\textrm{RSOS}}$ is equivalent to the \textit{finiteness property} of the restricted Jimbo--Miwa--Okado Boltzmann weights. It was shown that the restricted weights are finite and satisfy the star-triangle relation (\ref{eqn:str}) among themselves, which translates into the following result

\begin{lemma}
\label{lYB}
The $\R_{VV}(u)$ in the $A_{n-1}^{(1)}$ case and $\check R(u)$ in $B_n^{(1)}$, $C_n^{(1)}$, $D_n^{(1)}$ cases are well-defined on $V^{\mathrm{RSOS}}\otimes V^{\mathrm{RSOS}}$ and their restrictions to
  $V^{\mathrm{RSOS}}\otimes V^{\mathrm{RSOS}}$ satisfy the Yang--Baxter
  equation. Furthermore, they satisfy the inversion relations $\R_{VV}(u)\R_{VV}(-u)=\mathbb 1$ and $\R(u)\R(-u)=\mathbb 1$ on $V^{\mathrm{RSOS}}\otimes V^{\mathrm{RSOS}}$, where $\mathbb 1$ is an identity matrix depending on the algebra type.
\end{lemma}
\begin{proof}
The Yang--Baxter equations follow directly from the definition of $\R (u)$, $\R_{VV}(u)$ respectively, and the star-triangle relation (\ref{eqn:str}). The inversion relations in the non-restricted case follow from relation (\ref{eqn:invrel}),
 \begin{align*}
    \R(u)\R(-u)&=  \sum_{\substack{\varepsilon_i+\varepsilon_j=\varepsilon_k+\varepsilon_l,\\ \varepsilon_m+\varepsilon_o=\varepsilon_i+\varepsilon_j}}  \R^{k,l}_{i,j}(u)\R^{i,j}_{m,o}(-u)E_{km}\otimes E_{lo}\\
    &= \sum_{\substack{\varepsilon_i+\varepsilon_j=\varepsilon_k+\varepsilon_l,\\ \varepsilon_m+\varepsilon_o=\varepsilon_i+\varepsilon_j}}\mathcal{W}
\Big(\begin{matrix}
  b & a\\
  g & c
\end{matrix}\Big\vert u\Big) \mathcal{W}
\Big(\begin{matrix}
  b & g\\
  d & c
\end{matrix}\Big\vert -u\Big) E_{km}\otimes E_{lo}\\
&= \sum_{\substack{\varepsilon_i+\varepsilon_j=\varepsilon_k+\varepsilon_l,\\ \varepsilon_m+\varepsilon_o=\varepsilon_i+\varepsilon_j}}\delta_{k,m}\delta_{l,o}E_{km}\otimes E_{lo}= \sum_{\varepsilon_i-\varepsilon_j=\varepsilon_l-\varepsilon_k}E_{kk}\otimes E_{ll}= \mathbb{1},
\end{align*}
 where the dynamical parameters $a,b,c,d,g$ are as in the picture:
 \begin{center}
$\sum_{k,l,m,o} \sum_{g}$
\begin{tikzcd}
 b \arrow[r,"\varepsilon_k"] \arrow[d, "\varepsilon_i"'] \arrow[rd,phantom,"\R(u)"]
& a \arrow[d,"\varepsilon_l"]\\
g\arrow[r,"\varepsilon_j"'] 
& c 
\end{tikzcd}
\begin{tikzcd}
 b\arrow[r,"\varepsilon_i"]  \arrow[d, "\varepsilon_m"']\arrow[rd,phantom,"\R(-u)"] & g\arrow[d,"\varepsilon_j"]\\
d\arrow[r,"\varepsilon_o"'] 
& c 
\end{tikzcd}$=\mathbb 1$.
\end{center}
To see that this relation holds on the restricted space, it is enough to use the argument proved in \cite{JimboMiwaOkado1988}, that the weight 
\begin{tikzcd}
a \arrow[r,"\varepsilon_r"] \arrow[d, "\varepsilon_t"']\arrow[rd,phantom,"\R(u)"]
& b \arrow[d, "\varepsilon_u"] \\
d \arrow[r, "\varepsilon_s"']
& c
\end{tikzcd} with $(a,\varepsilon_t),(d,\varepsilon_s)\in\pi$ and $(a,\varepsilon_r)\notin\pi$ vanishes. The action of the inversion relation on $V^{\mathrm{RSOS}}\otimes V^{\mathrm{RSOS}}$ means that $(b,\varepsilon_m),(d,\varepsilon_o)\in\pi$ and therefore $\R_{m,o}^{i,j}(-u)$ vanishes unless $(b,\varepsilon_i)$ and $(g,\varepsilon_j)$ are inside $\pi$. Since $(b,\varepsilon_i),(g,\varepsilon_j)$ have to be in $\pi$, the same reasoning holds for the component $\R_{i,j}^{k,l}(u)$, showing that the terms in the inversion relation with arrows outside the groupoid $\pi$ vanish. Thus the relation holds on the restricted space.
\end{proof}

\begin{lemma}
\label{l513}
The operators $\R_{V^*V}(u)$, $\R_{VV^*}(u)$ and $\R_{V^*V^*}(u)$ in the case $A_{n-1}^{(1)}$ are well-defined on restricted spaces and satisfy the relations 
\begin{equation}
    \label{RA1}
        \begin{split}
      &\R^{(23)}_{V^*V}(u_1-u_2)\R^{(12)}_{V^*V}(u_1)\R^{(23)}_{VV}(u_2)= \R^{(12)}_{VV}(u_2)\R^{(23)}_{V^*V}(u_1)\R^{(12)}_{V^*V}(u_1-u_2),
    \end{split}
    \end{equation}
    \begin{equation}
        \label{RA2}
    \begin{split}
 \R^{(23)}_{V^*V^*}(u_1-u_2)\R^{(12)}_{V^*V}(u_1)\R^{(23)}_{V^*V}(u_2)= \R^{(12)}_{V^*V}(u_2)\R^{(23)}_{V^*V}(u_1)\R^{(12)}_{V^*V^*}(u_1-u_2).
    \end{split}
    \end{equation}
    on restricted spaces.
\end{lemma}
\begin{proof}
The operators $\R_{V^*V}(u)$, $\R_{VV^*}(u)$ and $\R_{V^*V^*}(u)$ are well-defined on restricted spaces since they are given by well-defined maps $\omega$, $\omega^*$, $\sigma$, $\sigma^*$ and $\R_{VV}(u)$. The relations (\ref{RA1}) and (\ref{RA2}) follow from the Lemma \ref{yb*}.
\end{proof}

\begin{lemma}
\label{lWWR}
The restriction of $\check R(u)$ to $V^{\mathrm{RSOS}}\otimes V^{\mathrm{RSOS}}$ is crossing symmetric, i.e. it obeys the relations 
\begin{align*}
\omega^{(23)}\R^{(12)}(u+\lambda)\R^{(23)}(u)=\rho(u)\omega\otimes \mathbb{1},\\
\R^{(12)}(u)\R^{(23)}(u+\lambda)\sigma^{(12)}=\rho(u)\mathbb{1}\otimes\sigma,
\end{align*}
for $B_n^{(1)}$, $C_n^{(1)}$, $D_n^{(1)}$, where $\mathbb 1$ is the $(2n+1)\times (2n+1)$ identity matrix in $B_n^{(1)}$ case and $2n\times 2n$ identity matrix in $C_n^{(1)}$ and $D_n^{(1)}$ cases. The function $\rho$ is
\begin{equation}\label{e-rho}
  \rho(u)=\frac{[u-1][\lambda +u]}{[u][1+\lambda +u]}.
\end{equation}
\end{lemma}
\begin{proof}
For the $R$-matrix (\ref{eqn:RM}) we have
\begin{align*}
&\left(\omega^{(23)}\R^{(12)}(u+\lambda)\R^{(23)}(u)-\rho(u)\omega^{(12)}\otimes \mathbb{1}\right)e_i\otimes e_j\otimes e_k\\
&=\Bigg(\sum_g \sqrt{\frac{G_{g}}{G_{a+\varepsilon_i+\varepsilon_j+\varepsilon_k}}}
W\Big(\begin{matrix}
  a & a+\varepsilon_i+\varepsilon_j+\varepsilon_k\\
  a+\varepsilon_i & g
\end{matrix}\Big\vert u+\lambda\Big)\\
&\qquad W\Big(\begin{matrix}
  a+\varepsilon_i & g\\
  a+\varepsilon_i+\varepsilon_j & a+\varepsilon_i+\varepsilon_j+\varepsilon_k
\end{matrix}\Big\vert u\Big)-\sqrt{\frac{G_{a+\varepsilon_i}}{G_a}}\delta_{0,\varepsilon_i+\varepsilon_j}\rho(u)\Bigg)e_{i+j+k}=0
\end{align*}
which follows from the equation (\ref{eqn:rel1}). Here $e_{i+j+k}$ is the vector in the weight space $\bar V_{\varepsilon_i+\varepsilon_j+\varepsilon_k}$. The second relation can be verified in a similar way using the identity (\ref{eqn:rel2}).

Now we have to check that the contributions of the arrows that are not elements of the subgroupoid $\pi$ and the contributions where $V_{(a,\varepsilon_0)}=\mathbb C e_0$ for $a_n=\frac12$ appears, vanish. Suppose that $(a,\varepsilon_i)$, $(a,\varepsilon_k)$, $(b,\varepsilon_l)$, $(d,\varepsilon_j)$ are composable arrows in the subgroupoid $\pi$ such that none of them satisfies the condition that the $n$-th component of the local state is $\frac12$ while the arrow is $\varepsilon_0$. Now assume that $(b,\varepsilon_p)\not\in\pi$ or that $(b,\varepsilon_p)=(b,\varepsilon_0)$ with $b\in P^{\ell+g}_{++}$.
\begin{center}
$\sum_g\displaystyle\sqrt{\frac{G_{g}}{G_{c}}}$\begin{tikzcd}
a \arrow[r, "\varepsilon_i"] \arrow[d, "\varepsilon_k"']\arrow[rd,phantom,"\R"]
& c\arrow[d, "\varepsilon_m"] \\
 b \arrow[r, "\varepsilon_p"] \arrow[d, "\varepsilon_l"']\arrow[rd,phantom,"\R"] & g \arrow [d, "-\varepsilon_m"] \\
d \arrow[r, "\varepsilon_j"'] & c
\end{tikzcd}$=\delta_{a,d}\displaystyle\sqrt{\frac{G_{b}}{G_a}}.$
\end{center}
As shown in Theorem 7 in \cite{JimboMiwaOkado1988}, in that case $\R^{p,-m}_{l,j}(b,u)=0$, unless $\mathfrak{g}^{(1)}=B_n^{(1)}$, $b_n=\frac12$ and $p=0$ or $p=-n$. The only configurations compatible with the first case $p =0$ are 
\begin{center}
\begin{tikzcd}
a \arrow[r, "\varepsilon_k"] \arrow[d, "\varepsilon_k"']\arrow[rd,phantom,"\R"]
& b\arrow[d, "\varepsilon_0"]\\
 b \arrow[r, "\varepsilon_0"] \arrow[d, "\varepsilon_l"']\arrow[rd,phantom,"\R"] & b \arrow [d, "\varepsilon_0"] \\
d \arrow[r, "-\varepsilon_l"'] & b
\end{tikzcd}$=\R^{k,0}_{k,0}(u+\lambda,a)\R^{0,0}_{l,-l}(u,b).$
\end{center}
Here $k\neq 0,n$ and $l\neq 0,-n$ because otherwise the arrows $(a,\varepsilon_k)$ and $(b,\varepsilon_l)$ are not in the subgroupoid $\pi$ or the arrows connected to $b$ should be $0$, which contradicts the assumption.

Similarly, in the second case $p=-n$ the only non-vanishing configurations are 
\begin{center}
\begin{tikzcd}
a \arrow[r, "\varepsilon_k"] \arrow[d, "\varepsilon_k"']\arrow[rd,phantom,"\R"]
& b\arrow[d, "-\varepsilon_n"] \\
 b \arrow[r, "-\varepsilon_n"] \arrow[d, "\varepsilon_l"']\arrow[rd,phantom,"\R"] & g \arrow [d, "\varepsilon_n"] \\
d \arrow[r, "-\varepsilon_l"'] & b
\end{tikzcd}$=\R^{k,-n}_{k,-n}(u+\lambda,a)\R^{-n,n}_{l,-l}(u,b)$
\end{center}
with $k\neq 0,n$ and $l\neq 0,-n$ for the same reason as above. By direct computation one can show that
\begin{align*}
\R^{k,0}_{k,0}(u+\lambda,a)\R^{0,0}_{l,-l}(u,b)\sqrt{G_b}
+\R^{k,-n}_{k,-n}(u+\lambda,a)\R^{-n,n}_{l,-l}(u,b)\sqrt{G_{b-\varepsilon_n}}=0.
\end{align*}
The only thing left to prove is that the weight 
\begin{tikzcd}
a \arrow[r, "\varepsilon_i"] \arrow[d, "\varepsilon_p"']\arrow[rd,phantom,"\R"]
& b\arrow[d, "\varepsilon_0"] \\
c \arrow[r, "\varepsilon_r"'] & d
\end{tikzcd}$=\R^{i,0}_{p,r}(u,a)$
with $b_n=\frac12$ and composable arrows $(a,\varepsilon_i),(a,\varepsilon_p),(c,\varepsilon_r)\in\pi$, such that none of them satisfies the condition that the $n$-th component of the local state is $\frac12$ while the arrow is $\varepsilon_0$, vanishes. The only configuration compatible with that case is
\begin{tikzcd}
a \arrow[r, "\varepsilon_i"] \arrow[d, "\varepsilon_0"']\arrow[rd,phantom,"\R"]
& b\arrow[d, "\varepsilon_0"] \\
a \arrow[r, "\varepsilon_i"'] & b
\end{tikzcd}$=\R^{\alpha,0}_{0,i}(u,a),$
 but since $(a,\varepsilon_i)$ and $(a,\varepsilon_0)$ are inside the groupoid $\pi$ and $a_n\neq\frac12$, it follows that $i=-n$ and $a_n=\frac32$, which gives $[a_{0}-a_{-n}-1]=0$.
 
A similar argument shows that the second identity holds in the restricted space.
\end{proof}
\begin{lemma}
\label{wttan}
In the $A_{n-1}^{(1)}$ case, the following relations are satisfied on the restricted space:
      \begin{align*}
\omega^{(23)}\R_{VV}^{(12)}(u+\lambda)\R^{(23)}_{V^*V}(u)&=\omega\otimes\rho_2(u),\\
\omega^{*(23)}\R^{(12)}_{V^*V}(u+\lambda)\R^{(23)}_{VV}(u)&=\omega^*\otimes\mathbb{1},\\
\R^{(12)}_{V^*V}(u)\R^{(23)}_{VV}(u+\lambda)\sigma^{(12)}&=\rho_2(u)\otimes\sigma,\\
\R^{(12)}_{VV}(u)\R^{(23)}_{V^*V}(u+\lambda)\sigma^{*(12)}&=\mathbb{1}\otimes\sigma^{*(23)},
      \end{align*}
      where
      \begin{equation}\label{e-rho2}
        \rho_2 (u)=\frac{[\lambda +u][\lambda -u]}{[1+\lambda +u][1+\lambda -u]},
      \end{equation}
      and $\mathbb 1$ is the $n\times n$ identity matrix.
\end{lemma}
\begin{proof}
The first relation follows from the inversion relation (\ref{eqn:invAn}),
\begin{align*}
    &\sum_{\substack{k,o,p\\ \varepsilon_p-\varepsilon_k=-\varepsilon_j+\varepsilon_m\\ \varepsilon_o+\varepsilon_k=\varepsilon_p+\varepsilon_i}}\omega_{k,-l}(\R_{VV})^{o,k}_{i,p}(u+\lambda)(\R_{V^*V})^{p,-k}_{-j,m}(u)e_{o}-\omega_{i,-j}e_{m}\rho_2(u)\\
    &=\sum_{g}\sqrt{\frac{G_g}{G_b}}\W\Big(\begin{matrix}
  a & b\\
  d & g
\end{matrix}\Big\vert \lambda +u\Big)\sqrt{\frac{G_cG_g}{G_bG_d}}\W\Big(\begin{matrix}
  c & d\\
  b & g
\end{matrix}\Big\vert \lambda -u\Big)e_o -\delta_{m,o}\sqrt{\frac{G_d}{G_a}}e_m\rho_2(u)=0,
\end{align*}
where the indices are clear from the picture,
\begin{center}
$\sum_g\displaystyle\sqrt{\frac{G_{g}}{G_{b}}}$\begin{tikzcd}
a \arrow[r, "\varepsilon_o"] \arrow[d, "\varepsilon_i"']\arrow[rd,phantom,"\R_{VV}"]
& b\arrow[d, "\varepsilon_k"] \\
 d \arrow[r, "\varepsilon_p"] \arrow[d, "-\varepsilon_j"']\arrow[rd,phantom,"\R_{V^*V}"] & g \arrow [d, "-\varepsilon_k"] \\
c \arrow[r, "\varepsilon_m"'] & b
\end{tikzcd}$=\delta_{i,j}\displaystyle\sqrt{\frac{G_{d}}{G_a}} \rho_2(u).$
\end{center}
The other equations follow from inverse relations (\ref{eqn:invrel}) and (\ref{eqn:invAn2}) in a similar way. 

To prove that relations hold on restricted space, we first recall from the proof of the Lemma \ref{lYB} that the diagram
\begin{tikzcd}
a \arrow[r, "\varepsilon_k"] \arrow[d, "\varepsilon_i"']\arrow[rd,phantom,"\R_{VV}"]
& b\arrow[d, "\varepsilon_l"] \\
d \arrow[r, "\varepsilon_j"'] & c
\end{tikzcd}
with $(a,\varepsilon_i),(d,\varepsilon_j)\in \pi$ is non-vanishing only when $(a,\varepsilon_k),(b,\varepsilon_l)\in \pi$, and is finite in that case. Similarly, the diagram
    \begin{center}
\begin{tikzcd}
a \arrow[r, "\varepsilon_k"] \arrow[d, "-\varepsilon_i"']\arrow[rd,phantom,"\R_{V^*V}"]
& b\arrow[d, "-\varepsilon_l"] \\
d \arrow[r, "\varepsilon_j"'] & c
\end{tikzcd}$=\displaystyle\sqrt{\frac{G_bG_d}{G_a G_c}}$
\begin{tikzcd}
d \arrow[r, "\varepsilon_i"] \arrow[d, "\varepsilon_j"']\arrow[rd,phantom,"\R_{VV}"]
& a\arrow[d, "\varepsilon_k"] \\
c\arrow[r, "\varepsilon_l"'] & b
\end{tikzcd}
\end{center}
with $(d,\varepsilon_i),(d,\varepsilon_j)\in \pi$ can only have $(a,\varepsilon_k),(c,\varepsilon_l)\in \pi$ for the following reason: firstly, it is finite since $a,c,d\in P^{\ell+g}_{++}$ means $G_a\neq 0$, $G_c\neq 0$ and $[d_j-d_l]\neq 0$. Furthermore, if we assume that $b\notin P^{\ell+g}_{++}$, it holds that $b_{k-1}=b_k$, since $a_{k-1}>a_k$, $b_{k-1}=a_{k-1}$ and $b_k=a_k+1$, therefore $G_b=0$.

 From this, one can easily see with restricted parameters $a,b,c,d,g$ and weights $\varepsilon_i,\varepsilon_j,\varepsilon_m,\varepsilon_o$, in the first equation
  \begin{center}
$\sum_g\displaystyle\sqrt{\frac{G_{g}}{G_{b}}}$\begin{tikzcd}
a \arrow[r, "\varepsilon_o"] \arrow[d, "\varepsilon_i"']\arrow[rd,phantom,"\R_{VV}"]
& b\arrow[d, "\varepsilon_k"] \\
 d \arrow[r, "\varepsilon_p"] \arrow[d, "-\varepsilon_j"']\arrow[rd,phantom,"\R_{V^*V}"] & g \arrow [d, "-\varepsilon_k"] \\
c \arrow[r, "\varepsilon_m"'] & b
\end{tikzcd}$=\delta_{i,j}\displaystyle\sqrt{\frac{G_{d}}{G_a}} \rho_2(u).$
\end{center}
 all the contributions that leave the restricted space vanish. The same holds for other relations.
\end{proof}

\begin{cor}
\label{Rvvstar}
    In the $A_{n-1}^{(1)}$ case we have the inversion identities
    \begin{align*}
        \R_{VV^*}(-u)\R_{V^*V}(u)=\rho_2(u),\qquad \R_{V^*V}(u)\R_{VV^*}(-u)=\rho_2(u)
    \end{align*}
    that hold on $(V^*)^{\mathrm{RSOS}}\otimes V^{\mathrm{RSOS}}$ and $V^{\mathrm{RSOS}}\otimes (V^*)^{\mathrm{RSOS}}$, respectively.
\end{cor}
\begin{proof}
    This follows from Lemma \ref{inversion} and Lemma \ref{wttan}.
\end{proof}

\section{Restricted quantum groups}\label{sec-5}

In this section, we introduce a class of $\pi^2$-graded Hopf algebras based on
$R$-matrices on $V\otimes V$ for a $\pi$-graded vector space $V$. They are defined by generators and
relations, where the generators have a definite $\pi^2$-degree and the relations are homogeneous, so that
the quotient is a $\pi^2$-graded algebra.
The main examples are obtained
from the RSOS models of Jimbo, Miwa, and Okado associated with classical Lie
algebras of the previous section, but we
first consider a more general situation, which presumably also applies to general simple
Lie algebras. Out of such an $R$-matrix one obtains a $\pi^2$-graded bialgebra
defined by a graded version of RTT relations. To construct the antipode and
obtain a Hopf algebra, one needs an additional property: the crossing symmetry,
which relies on an inner product. Two cases are to be considered:
either the inner product is defined on the space $V$ and we say that  the corresponding quantum  group is of orthogonal type (including the symplectic case); or
the inner product is defined on $V\oplus V^*$ and we extend the $R$ matrix to this space by ``rotation'', and we say that the quantum group is of general linear type.

\subsection{Elliptic quantum algebras}\label{ss-elliptic algebras}
In the general context of the quantum inverse scattering method \cite{ReshetikhinTakhtadzhyanFaddeev1989}, \cite{ReshetikhinSemenovTianShansky1990}
quantum algebras are defined by RTT relations which are quadratic relations, each of the form
\begin{equation}\label{e-rel1}
  \sum_{\alpha,\beta=1}^n c_{\alpha,\beta}(u_1-u_2)t_\alpha(u_1)t_\beta(u_2)=0
\end{equation}
for generating functions $t_\alpha(u)$ of generators and generating functions $c_{\alpha,\beta}(u)$ of
structure constants (matrix entries of $R$-matrices). The meaning of ``generating function'' depends on the
type of algebra and of the analytic structure of the coefficients $c_{\alpha,\beta}(u)$. For example in
the case of quantum affine algebras \cite{ReshetikhinSemenovTianShansky1990} the coefficients are meromorphic periodic
functions of $u$ admitting a power series expansion in the variable $z=\exp(2\pi i u/L)$ for some period $L$.
Accordingly, the generators $t_{\alpha}[n]$ of the algebra $A$
are collected in a formal Fourier series $t_\alpha(u)=\sum_{n=0}^\infty
t_\alpha[n]z^n\in A[[z]]$. Expanding \eqref{e-rel1} in powers of the two exponential variables $z_1,z_2$ and comparing coefficients
yields an infinite set of quadratic relations for the generator, each involving finitely many generators. This
can be best seen by writing \eqref{e-rel1} in multiplicative notation $c_\alpha(u)=\tilde c_\alpha(z)$, $t_\alpha(u)=\tilde t_\alpha(z)$:
\begin{equation}\label{e-rel2}
  \sum_{\alpha,\beta=1}^n \tilde c_{\alpha,\beta}(z_1)\tilde t_\alpha(z_1z_2)\tilde t_\beta(z_2)=0,\quad\text{in $A[[z_1,z_2]]$.} 
\end{equation}
For our application to elliptic restricted quantum groups we wish to generalize this to the elliptic case. It is convenient to consider {\em families} of algebras with quadratic relations parametrized by a modular parameter $\tau$ in the upper half-plane and include the trigonometric limit $\tau\to i\infty$. The relations take the form
\begin{equation}\label{e-rel3}
  \sum_{\alpha,\beta=1}^n c_{\alpha,\beta}(u_1-u_2,\tau)t_\alpha(u_1,\tau)t_\beta(u_2,\tau)=0,
\end{equation}
where the coefficients $c_{\alpha,\beta}(u,\tau)$ are elliptic functions admitting a convergent Fourier expansion
\[
  c_{\alpha,\beta}(u,\tau)=\sum_{n,m=0}^\infty c_{\alpha,\beta}[n,m] e^{2\pi i nu/L}e^{2\pi im(\tau-u/L)}.
\]
In the restricted SOS models of Section \ref{sec-4}, this follows from the form of the $R$-matrix coefficients listed in Section \ref{4.1} and the Jacobi
product formula \eqref{e-triple} for the theta function $[x]$ of $x$ and $\tau$.
Accordingly, the relations \eqref{e-rel3} are
relations for the coefficients $t_\alpha[n,m]$ of generating functions
\[
  t_\alpha(u)=\sum_{n,m=0}^\infty t_\alpha[n,m]z^n(p/z)^m,\quad z=e^{2\pi i u/L},\quad p=e^{2\pi i\tau}.
\]
Expanding \eqref{e-rel3} and comparing coefficients yields infinitely many quadratic relations for the
generators $t_\alpha[n,m]$, each involving a finite subset of generators. The $t_\alpha[n,0]$ obey
relations \eqref{e-rel1} for the trigonometric limit $c_{\alpha,\beta}(u,\tau=i\infty)$ of the coefficients.

\begin{definition}
    Let $\mathcal O=\mathbb C[[z,p/z]]$. An \emph{elliptic quantum algebra} is a unital algebra over $\mathbb C[p]$ generated by $t_\alpha[n,m]$, for $\alpha$ in a finite index set and $n,m\in\mathbb Z_{\geq0}$ subject to relations of the form \eqref{e-rel2}, with $c_{\alpha,\beta}\in\mathcal O$.
\end{definition}
We can write the relations in a form generalizing \eqref{e-rel2} by introducing the variable $w=p/z$.
With $\tilde c_{\alpha,\beta}(z,w)=\sum_{n,m\geq0} c_{\alpha,\beta}[n,m]z^nw^m$ and $\tilde t_\alpha(z,w)=
\sum_{n,m\geq0}t_\alpha[n,m]z^nw^m$, the relations \eqref{e-rel3} are
\[
  \sum_{\alpha,\beta=1}^n \tilde c_{\alpha,\beta}(z_1,z_2w)\tilde t_\alpha(z_1z_2,w)\tilde t_\beta(z_2,z_1w)
  =0,\quad\text{in $A[[z_1,z_2,w]]$.} 
\]
The relations are bilinear over $\mathbb C[p]$ where $p=z_1z_2w$. They reduce to the relations \eqref{e-rel2} in the trigonometric
specialization $w=0$. It is useful to think of the power series in $z$ and $p/z$ as functions on the ``formal annulus'': convergent series with convergence radius 1 in both variables represent holomorphic functions on the family of annuli $|p|<|z|<1$.

\subsection{Restricted quantum groups of general linear type}\label{sec-5.1}
Let $\pi$ be a finite groupoid with at most one morphism between any two objects. Take $V\in\operatorname{Vect}^f_{\mathbb{C}}(\pi)$ to be a finite-dimensional $\pi$-graded vector space over $\mathbb{C}$, and denote its dual space by $V^*$. Consider all the arrows $k\in\pi_1$ such that $V_k\neq \{0\}$ and let $\tilde\pi\subseteq\pi$ be the set containing all such arrows. Clearly, all nonempty components of  $V^*$ are graded by the set $-\tilde\pi:=\{k^{-1}\mid k\in\tilde\pi\}$.

 Let $\check R_{VV}(u)\in \operatorname{End}(V \otimes V )$ be a crossing-symmetric $R$-matrix as described in Section \ref{sec-3.1} and let $\check R_{V^*V}(u)$, $\check R_{V^*V^*}(u)$ and $\check R_{VV^*}(u)$ be its rotations given by Definition \ref{Rr2}. These $R$-matrices will be used as coefficients for quadratic relations defining elliptic quantum algebras as in the previous section and we therefore assume that the matrix elements of $\check R_{VV}(u)$ are meromorphic functions of $u\in\mathbb C$ and $\tau$ in the upper half-plane admitting a power series expansion in the variables $e^{2\pi i u/L},e^{2\pi i(\tau-u/L)}$
 
Consider the double groupoid $\Gamma=\pi^2$ and let $(A_\alpha )_{\alpha\in\pi^2}$ be a $\pi ^2$-graded unital algebra described in Section \ref{sec-2.3} with $\pi$ being replaced by $\Gamma^\circ$. It is generated by $T_i^j[m,n]$ and $(T^*)_{-i}^{-j}[m,n]$, $i,j\in\tilde\pi$, compactly represented by a formal Fourier series of the form
\begin{equation}
\label{fps}
\begin{split}
T_i^j(u)=&\sum_{m,n\geq 0}T_i^j[m,n]e^{2\pi inu/L}(p\,e^{-2\pi iu/L})^m,\\ (T^*)_{-i}^{-j}(u)=&\sum_{m,n\geq 0}(T^*)_{-i}^{-j}[m,n]e^{2\pi inu/L}(p\,e^{-2\pi iu/L})^m,
\end{split}
\end{equation}
as in Section \ref{ss-elliptic algebras}.
The generators $T_i^j(u)$ and $(T^*)_{-i}^{-j}(u)$ live in the algebra component graded by $\alpha\in\Gamma^{\bullet}(i,j)$, $\beta\in\Gamma^\bullet (i^{-1},j^{-1})$,
\begin{align*}
    T^j_i(u)\in \operatorname{Hom}_{\mathbb{C}}(V_{i},V_{j})\otimes A_{\alpha},\qquad
    (T^*)^{-j}_{-i}(u)\in \operatorname{Hom}_{\mathbb{C}}((V^*)_{i^{-1}},(V^*)_{j^{-1}})\otimes A_{\beta}.
\end{align*}
Here we really mean that the Fourier coefficients of $T_i^j(u)$ and $(T^*)^{-j}_{-i}(u)$ live in $\operatorname{Hom}_{\mathbb{C}}(V_{i},V_{j})\otimes A_{\alpha}$ and $\operatorname{Hom}_{\mathbb{C}}((V^*)_{i^{-1}},(V^*)_{j^{-1}})\otimes A_{\beta}$, respectively.
These generators can be graphically represented by their $\Gamma$-graded component, i.e. the commutative square with the fixed vertical morphisms $i$ and $j$,
\begin{center}
$T_i^j(u)= $\begin{tikzcd}
a \arrow[r,"l"] \arrow[d, "i"']\arrow[rd,phantom,"T(u)"]
& b \arrow[d, "j"] \\
c \arrow[r,"k"']
& d 
\end{tikzcd}, $\qquad (T^*)_{-i}^{-j}(u)=$
\begin{tikzcd}
c \arrow[r,"k"] \arrow[d, "i^{-1}"']\arrow[rd,phantom,"T^*(u)"]
& d \arrow[d, "j^{-1}"] \\
a \arrow[r,"l"']
& b 
\end{tikzcd}.
\end{center}

We obtain the standard dynamical formulation $T^j_i(u,a)\in A_\alpha$ by setting the dynamical parameter $a$ to be the source of the left vertical arrow $i$, $a=s(i)$. The dynamical formulation of  $(T^*)^{-j}_{-i}(u)$ is $(T^*)^{-j}_{-i}(u,c)$ for $c=s(i^{-1})$.

The unit $\eta\colon I_v\to A$ has a non-vanishing component $\eta (1)\in\oplus_{f\in\pi_1} A_{\mathrm{id}_{f}^\circ}$ and is presented as
\begin{tikzcd}
a \arrow[r,"f"] \arrow[d, "\mathrm{id}_a"']\arrow[rd,phantom,"\eta(1)"]
& b \arrow[d, "\mathrm{id}_b"] \\
a \arrow[r,"f"']
& b 
\end{tikzcd}. We will also consider central elements $\xi^+[m,n]$ and $\xi^-[m,n]$, whose non-vanishing components have the same grading as the unit, and can be represented by a formal power series 
\[
\xi^\pm(u)=\sum_{m,n\geq 0}\xi^\pm[m,n]e^{2\pi inu/L}(q^2e^{-2\pi iu/L})^m.
\]

Let us further introduce matrix $T(u)\in\operatorname{Hom}_{\operatorname{Vect}^f_{\mathbb{C}}(\pi)}(V,A\otimes_h V)$, where $A\otimes_h V$ is $\pi$-graded, $(A\otimes_hV)_i=\oplus_{\alpha\in\Gamma^\bullet (i,j)}A_{\alpha}\otimes V_j$, and its matrix elements are $T^{j}_{i}(u)$. Similarly, we have $T^*(u)\in\operatorname{Hom}_{\operatorname{Vect}^f_{\mathbb{C}}(\pi)}(V^*,A\otimes_h V^*)$ with $\pi$-graded vector space $(A\otimes_hV^*)_{i^{-1}}=\oplus_{\alpha\in\Gamma^\bullet (i^{-1},j^{-1})}A_\alpha\otimes V^*_{j^{-1}}$ and matrix elements $(T^*)^{-j}_{-i}(u)$.

Before introducing the definition of restricted quantum group, we first need to understand the $RTT$-relations in the context of double groupoid graded algebras. We have the following commutative diagram
\[
        \begin{tikzcd}
          V\otimes_vV\arrow[r," \operatorname{id}_V\otimes_v T(u_2)"] \arrow[d,"\R_{VV}(u_1-u_2)"]
          & V\otimes_v(A\otimes_hV)\arrow[r,"T^{(1)}(u_1)"] & (A\otimes_hV)\otimes_v(A\otimes_hV)\arrow[d,"(*)"']
          \\
          V\otimes_v V \arrow[d,"\operatorname{id}_V\otimes_v T(u_1)"]  & &
          (A\otimes_vA) \otimes_h (V\otimes_vV)\arrow[d,"\nabla\otimes_h(\operatorname{id}_V\otimes_v\operatorname{id}_V)"']\\
          V\otimes_v(A\otimes_h V)\arrow[d,"T(u_2)\otimes_v(\operatorname{id}_A\otimes_h\operatorname{id}_V)"] & & A\otimes_h(V\otimes_vV)\arrow[d,"\operatorname{id}_A\otimes_h\R_{VV}(u_1-u_2)"']\\
          (A\otimes_hV)\otimes_v(A\otimes_hV)\arrow[r,"(*)"]&(A\otimes_vA) \otimes_h (V\otimes_vV)\arrow[r,"\nabla^{(12)}"]  & A\otimes_h(V\otimes_vV)
        \end{tikzcd}
      \]
      with $T^{(1)}(u_1)=T(u_1)\otimes_v(\operatorname{id}_A\otimes_h\operatorname{id}_V)$, $\nabla^{(12)}=\nabla\otimes_h (\operatorname{id}_V\otimes_v\operatorname{id}_V)$ and $(*)$ the morphism (\ref{e-iso}) where we used $\pi^2$-grading of $V$ induced from its $\pi$-grading via $V_i=V_{\operatorname{id}^\bullet_i}$.

    We also have the relations of type $\omega^{(23)}T^{(1)}(u+\lambda)T^{*(2)}(u)=\omega^{(12)}$ described by the diagram
  \[
        \begin{tikzcd}
          V\otimes_vV^*\cong (V\otimes_vV^*)\otimes_h I_v\arrow[r," \omega\otimes_h\eta"] \arrow[d,"T^{*(2)}(u+\lambda)"]
          & A \\
          V\otimes_v (A\otimes_hV^*) \arrow[d,"T^{(1)}(u)"]& A\otimes_h(V\otimes_vV^*)\arrow[u,"\omega^{(23)}"'] \\
           (A\otimes_h V)\otimes_v(A\otimes_h V^*)\arrow[r,"(*)"] & (A\otimes_vA)\otimes_h(V\otimes_vV^*)\arrow[u,"\nabla^{(12)}"'] 
        \end{tikzcd}
      \]  
      and the relations of type $T^{*(1)}(u)T^{(2)}(u+\lambda)\sigma^{(12)}=\sigma^{(23)}$ given by the commutative diagram
       \[
        \begin{tikzcd}
          \mathbf{1}\cong I_v \otimes_h \mathbf{1}\arrow[r,"\eta\otimes_h\sigma"]\arrow[d,"\sigma"]
          & A\otimes_h(V^*\otimes_v V) & (A\otimes_v A)\otimes_h(V^*\otimes_vV)\arrow[l,"\nabla^{(12)}"'] \\
          V^*\otimes_v V\arrow[r,"T^{(2)}(u+\lambda)"] & V^*\otimes_v (A\otimes_h V)\arrow[r,"T^{*(1)}(u)"] & (A\otimes_hV^*)\otimes_v(A\otimes_hV)\arrow[u,"(*)"] 
        \end{tikzcd}
      \]  
      where $(*)$ is again the morphism (\ref{e-iso}).
      \begin{definition}\label{def-gltype} Let $\check R_{VV},\omega,\omega^*$ be as in Section \ref{sec-3.1} and let
        rotations $R_{VV^*}$, $R_{V^*V}$,  $R_{V^*V^*}$ be the rotations of $R$ as in Definition \ref{Rr2}. Assume
        that the crossing symmetry \ref{def-crossing-gl} with parameter $\lambda$ holds. 
        The {\em restricted quantum algebra of general linear type} associated with these data
        is the $\pi^2$-graded unital algebra $A=A_{R,\omega,\omega^*}=\oplus_{\alpha\in\pi^2}A_{\alpha}$ generated by matrix elements $T_i^j(u)$, $(T^*)_{-i}^{-j}(u)$, $i,j\in\tilde\pi$, of $T(u)\in\operatorname{Hom}_{\operatorname{Vect}_{\mathbb{C}}(\pi)}(V,A\otimes_h V)$ and central elements $\xi_2^+(u),\xi_2^{-}(u)\in\oplus_{f\in\pi_1} A_{\mathrm{id}_{f}^\circ}$, subject to relations
  \begin{equation}
  \label{xiA}
      \xi_2^{\pm}(u)\xi_2^{\mp}(u)=\eta(1),
  \end{equation}
\begin{equation}
\label{eqn:1}
\R_{VV}^{(23)}(u_1-u_2)T^{(1)}(u_1)T^{(2)}(u_2)=T^{(1)}(u_2)T^{(2)}(u_1)\R_{VV}^{(12)}(u_1-u_2),
 \end{equation}
 \begin{equation}
\label{eqn:3}
\begin{split}
 \R_{V^*V^*}^{(23)}(u_1-u_2)T^{*(1)}(u_1)T^{*(2)}(u_2)=T^{*(1)}(u_2)T^{*(2)}(u_1)\R_{V^*V^*}^{(12)}(u_1-u_2),
 \end{split}
 \end{equation}
 \begin{equation}
\label{eqn:4}
\begin{split}
\R_{V^*V}^{(23)}(u_1-u_2)T^{*(1)}(u_1)T^{(2)}(u_2)=T^{(1)}(u_2)T^{*(2)}(u_1)\R_{V^*V}^{(12)}(u_1-u_2),
 \end{split}
 \end{equation}
  \begin{equation}
\label{eqn:5}
\omega^{(23)}T^{(1)}(u+\lambda)T^{*(2)}(u)=\omega^{(12)}\xi_2^+(u),
 \end{equation}
   \begin{equation}
\label{eqn:6}
\omega^{*(23)}T^{*(1)}(u+\lambda)T^{(2)}(u)=\omega^{*(12)},
 \end{equation}
    \begin{equation}
\label{eqn:7}
T^{*(1)}(u)T^{(2)}(u+\lambda)\sigma^{(12)}=\xi^+_2(u)\sigma^{(23)},
 \end{equation}
     \begin{equation}
\label{eqn:8}
T^{(1)}(u)T^{*(2)}(u+\lambda)\sigma^{*(12)}=\sigma^{*(23)},
 \end{equation}
where  $\omega$, $\sigma$ and $\sigma^*$ are maps defined in Section \ref{sec-3.1}.
\end{definition}
\begin{remark}
As explained in Section \ref{ss-elliptic algebras}, expanding $T_i^j(u)$, $(T^*)_i^j(u)$ as formal power series (\ref{fps}), (\ref{eqn:1})-(\ref{eqn:8}) become relations involving only finitely many terms $T_i^j[m_1,n_1]T_k^l[m_2,n_2]$. The same holds for central elements $\xi_2^\pm(u)$.
\end{remark}
It is useful to have graphical representations of the relations above, for example relations (\ref{eqn:1}), (\ref{eqn:5}) and (\ref{eqn:7}) become
\begin{center}
$\sum_{k,l}$\begin{tikzcd}
a \arrow[r] \arrow[d, "i"']\arrow[rd,phantom,"T(u_1)"]
& b\arrow[d, "k"]\arrow[rd,"m"] &  \\
f \arrow[r] \arrow[d, "j"']\arrow[rd,phantom,"T(u_2)"] &g \arrow [d, "l"] & c\arrow[ld,"n"] \\
e \arrow[r] & d &
\end{tikzcd}$=\sum_{k,l}$\begin{tikzcd}
& a \arrow[ld,"i"']\arrow[r] \arrow[d, "k"']\arrow[rd,phantom,"T(u_2)"]
& b\arrow[d, "m"] \\
 f\arrow[rd,"j"'] &g\arrow[r] \arrow[d, "l"']\arrow[rd,phantom,"T(u_1)"] & c \arrow [d, "n"] \\
& e \arrow[r] & d
\end{tikzcd}$,$
\end{center}
where the side diagrams represent $(\R_{VV})^{m,n}_{k,l}(u_1-u_2)$ and $(\R_{VV})^{k,l}_{i,j}(u_1-u_2)$, respectively, and
\begin{center}
$\sum_{k,l}$\begin{tikzcd}
a \arrow[r] \arrow[d, "i"']\arrow[rd,phantom,"T(v)"]
& b\arrow[d, "k"]\arrow[dd,bend left=40] \\
 e \arrow[r] \arrow[d, "j"'] \arrow[rd,phantom,"T^*(u)"] & f \arrow [d, "l"] \\
c\arrow[r] &  d
\end{tikzcd}$=\xi_2^+(u)\omega_{i,j},\quad$ 
$\sum_{k,l}$\begin{tikzcd}
a\arrow[dd,bend right=40] \arrow[r] \arrow[d, "k"']\arrow[rd,phantom,"T^*(u)"]
& b\arrow[d, "i"] \\
 e \arrow[r] \arrow[d, "l"'] \arrow[rd,phantom,"T(v)"]& f \arrow [d, "j"] \\
c \arrow[r] & d
\end{tikzcd}$=\xi_2^+(u)\sigma^{i,j},$
\end{center}
with $v=u+\lambda$. 

We define two maps, the co-unit $\epsilon: A_{R,\omega,\omega^*}\rightarrow I_h$ and the coproduct $\Delta :A_{R,\omega,\omega^*}\rightarrow A_{R,\omega,\omega^*}\otimes_h A_{R,\omega,\omega^*}$ by
\begin{align*}
\epsilon (T_{\mathrm{id}^\bullet_i} (u)):=1,\qquad\epsilon ((T^*)_{\mathrm{id}^\bullet_{i^{-1}}}(u)):=1,
\end{align*}
and
\begin{align*}
\copr (T^j_i (u))&:=\sum_k T^k_i(u)\otimes_h T^j_k(u),\\
\copr ((T^*)^{-j}_{-i} (u))&:=\sum_l (T^*)^{-l}_{-i}(u)\otimes_h (T^*)^{-j}_{-l}(u),
\end{align*} 
where $k$ goes over all morphisms such that $t^\circ (\beta)\circ t^\circ(\alpha)=t^\circ(\gamma)$ for $\alpha\in\Gamma^\bullet(i,k)$, $\beta\in\Gamma^\bullet(k,j)$, $\gamma\in\Gamma^\bullet(i,j)$ in the first relation and $l$ goes over all morphisms such that $t^\circ (\beta^\prime)\circ t^\circ(\alpha^\prime)=t^\circ(\gamma^\prime)$ for $\alpha^\prime\in\Gamma^\bullet(i^{-1},l^{-1})$, $\beta^\prime\in\Gamma^\bullet(l^{-1},j^{-1})$, $\gamma^\prime\in\Gamma^\bullet(i^{-1},j^{-1})$ in the second relation. 
We extend the definitions on generators multiplicatively by declaring that $\copr, \epsilon$ are algebra homomorphisms.
Note that for the co-unit one can write
\begin{align*}
\epsilon (T^j_i (u))=\delta_{i,j},\qquad\epsilon ((T^*)^{-j}_{-i}(u))=\delta_{i,j}.
\end{align*}
The graphical representation of the action of coproduct is
\begin{center}
$\epsilon\Bigg($\begin{tikzcd}
a \arrow[r] \arrow[d, "i"']\arrow[rd,phantom,"T(u)"]
& b \arrow[d, "j"] \\
c \arrow[r]
& d 
\end{tikzcd}$\Bigg)=\delta_{i,j},\qquad\epsilon\Bigg($\begin{tikzcd}
a \arrow[r] \arrow[d, "i^{-1}"']\arrow[rd,phantom,"T^*(u)"]
& b \arrow[d, "j^{-1}"] \\
c \arrow[r]
& d 
\end{tikzcd}$\Bigg)=\delta_{i,j}$,
\end{center} 
\begin{center}
$\copr\Bigg($\begin{tikzcd}
a \arrow[r] \arrow[d, "i"']\arrow[rd,phantom,"T(u)"]
& b \arrow[d, "j"] \\
c \arrow[r,"m"']
& d 
\end{tikzcd}$\Bigg)=\sum_{k, o\circ n=m}$\begin{tikzcd}
a \arrow[r] \arrow[d, "i"']\arrow[rd,phantom,"T(u)"]
& e \arrow[d, "k"] \arrow[r]\arrow[rd,phantom,"\quad T(u)"] & b \arrow[d, "j"] \\
c \arrow[r,"n"']
& f \arrow[r,"o"'] & d 
\end{tikzcd},
\end{center}
and similarly for $\copr (T^*(u))$. For the central elements $\xi_2^\pm(u)$ one has
\begin{center}
$\epsilon\Bigg($\begin{tikzcd}
a \arrow[r] \arrow[d, "\mathrm{id}_a"']\arrow[rd,phantom,"\xi_2^\pm(u)"]
& b \arrow[d, "\mathrm{id}_b"] \\
a \arrow[r]
& b 
\end{tikzcd}$\Bigg):=\delta_{a,b},$
\end{center}
\begin{center}
$\copr\Bigg($\begin{tikzcd}
a \arrow[r] \arrow[d, "\mathrm{id}_a"']\arrow[rd,phantom,"\xi_2^\pm(u)"]
& b \arrow[d, "\mathrm{id}_b"] \\
a \arrow[r,"m"']
& b 
\end{tikzcd}$\Bigg):=\sum_{c, o\circ n=m}$\begin{tikzcd}
a \arrow[r] \arrow[d, "\mathrm{id}_a"']\arrow[rd,phantom,"\xi_2^\pm(u)"]
& c \arrow[d, "\mathrm{id}_c"] \arrow[r]\arrow[rd,phantom,"\quad\xi_2^\pm(u)"] & b \arrow[d, "\mathrm{id}_b"] \\
a \arrow[r,"n"']
& c \arrow[r,"o"'] & b 
\end{tikzcd}
\end{center}
and similar for the unit element $\eta(1)$.
\begin{remark}
Strictly speaking, the definition of the co-unit and the coproduct is as follows:
\begin{align*}
    (\operatorname{id}\otimes\epsilon)(T_i^j(u))=\operatorname{id}_{V_i}\delta_{i,j},\qquad (\operatorname{id}\otimes\copr)(T_i^j(u))=m^{(13)}\sum_k T_i^k(u)\otimes T_k^j(u),
\end{align*}
and similarly for $(T^*)_i^j(u)$. However, we will use the abbreviated form to avoid cumbersome expressions in proofs.
\end{remark}
\begin{lemma}
\label{l2An}
The co-unit $\epsilon$ and the coproduct $\copr$ preserve relations (\ref{eqn:1}) - (\ref{eqn:8}) and satisfy the coalgebra axioms.
\end{lemma}
\begin{proof}
We show that the statement holds for the relations (\ref{eqn:1}) and (\ref{eqn:5}). The proof for other relations follows in an analogous way. 

The action of the co-unit on the left-hand side of (\ref{eqn:1}) gives 

$\epsilon\Bigg(\sum_{k,l}$\begin{tikzcd}
a \arrow[r] \arrow[d, "i"']\arrow[rd,phantom,"T(u_1)"]
& b\arrow[d, "k"]\arrow[rd,"m"] &  \\
f \arrow[r] \arrow[d, "j"']\arrow[rd,phantom,"T(u_2)"] &g \arrow [d, "l"] & c\arrow[ld,"n"] \\
e \arrow[r] & d &
\end{tikzcd}\Bigg)$=\sum_{k,l}\delta_{i,k}\delta_{j,l}$\begin{tikzcd}
b \arrow[r,"m"] \arrow[d, "k"']\arrow[rd,phantom,"\R_{VV}"]
& c\arrow[d, "n"] \\
g \arrow[r,"l"'] & d &
\end{tikzcd}$=\delta_{a,b}$\begin{tikzcd}
a\arrow[r,"m"] \arrow[d, "i"']\arrow[rd,phantom,"\R_{VV}"]
& c\arrow[d, "n"] \\
f \arrow[r,"j"'] &e &
\end{tikzcd}
This is equal to the right-hand side,

$\epsilon\Bigg(\sum_{k,l}$\begin{tikzcd}
& a \arrow[ld,"i"']\arrow[r] \arrow[d, "k"']\arrow[rd,phantom,"T(u_2)"]
& b\arrow[d, "m"] \\
 f\arrow[rd,"j"'] &g\arrow[r] \arrow[d, "l"']\arrow[rd,phantom,"T(u_1)"] & c \arrow [d, "n"] \\
& e \arrow[r] & d
\end{tikzcd}$\Bigg)=\sum_{k,l}\delta_{k,m}\delta_{l,n}$\begin{tikzcd}
a \arrow[r,"k"] \arrow[d, "i"']\arrow[rd,phantom,"\R_{VV}"]
& g\arrow[d, "l"] \\
f \arrow[r,"j"'] & e &
\end{tikzcd}$=\delta_{a,b}$\begin{tikzcd}
a \arrow[r,"m"] \arrow[d, "i"']\arrow[rd,phantom,"\R_{VV}"]
& c\arrow[d, "n"] \\
f \arrow[r,"j"'] &e &
\end{tikzcd}

The same holds for the identity (\ref{eqn:5}),
\begin{align*}
\sum_{k,l}\omega_{k,-l}\epsilon\left(T^k_i(u+\lambda)\right)\epsilon\left((T^*)_{-j}^{-l}(u)\right)-\omega_{i,-j}\epsilon\left(\xi_2^+(u)\right)\\
=\sum_{k,l}\omega_{k,-l}\delta_{i,k}\delta_{j,l}-\omega_{i,-j}=0.
\end{align*}
The coproduct is also compatible with the $RTT$-relations:
\begin{align*}
     &\sum_{k,l}(\check R_{VV})^{m,n}_{k,l}(u_1-u_2)\copr(T^{k}_i(u_1))\copr(T^{l}_j(u_2))\\
    &=\sum_{k,l,o,p}T^o_i(u_1) T^p_j(u_2) \otimes_h (\check R_{VV})^{m,n}_{k,l}(u_1-u_2)T^{k}_o(u_1)T^{l}_p(u_2)\\
 &=\sum_{k,l,o,p} (\check R_{VV})^{k,l}_{o,p}(u_1-u_2)T^o_i(u_1) T^p_j(u_2) \otimes_h  T^{m}_k(u_2)T^{n}_l(u_1)\\
    &=\sum_{k,l,o,p} T^{k}_o(u_2)T^{l}_p(u_1)(\check R_{VV})^{o,p}_{i,j}(u_1-u_2)\otimes_h T^{m}_k(u_2)T^{n}_l(u_1) \\
     &=\sum_{o,p}\copr(T^{m}_o(u_2))\copr(T^{n}_p(u_1))(\check R_{VV})^{o,p}_{i,j}(u_1-u_2),
\end{align*}
and with the identity (\ref{eqn:5}):
\begin{align*}
    &\sum_{k,l}\omega_{k,-l}\copr(T^{k}_i(u+\lambda))\copr((T^*)^{-l}_{-j}(u))\\
    &=\sum_{k,l,o,p}\omega_{k,-l}T^o_i(u+\lambda) (T^*)^{-p}_{-j}(u) \otimes_h T^{k}_o(u+\lambda)(T^*)^{-l}_{-p}(u)\\
    &=\xi_2^+(u)\otimes\sum_{o,p}\omega_{o,-p}T^o_i(u+\lambda) (T^*)^{-p}_{-j}(u)=\omega_{i,-j}\xi_2^+(u)\otimes_h\xi_2^+(u)=\omega_{i,-j}\copr(\xi_2^+(u)).
\end{align*}

Finally, we prove that the coalgebra axioms are satisfied:
\begin{enumerate}
\item $(\operatorname{id}\otimes_h\copr)\circ\copr=(\copr\otimes_h\operatorname{id})\circ\copr$:
\begin{align*}
(\operatorname{id}\otimes_h\copr)\circ\copr\left(T^j_i(u)\right)&=(\operatorname{id}\otimes_h\copr)\sum_k T^k_i(u)\otimes_h T^j_k(u)\\
&=\sum_k T^k_i(u)\otimes_h \sum_l T^l_k(u)\otimes_h T^j_l(u)\\
&=(\copr\otimes_h\operatorname{id})\sum_l T^l_i(u)\otimes_h T^j_l(u)\\
&=(\copr\otimes_h\operatorname{id})\circ\copr\left(T^j_i(u)\right),
\end{align*}
\item $(\operatorname{id}\otimes_h\epsilon)\circ\copr=(\epsilon\otimes_h\operatorname{id})\circ\copr=\operatorname{id}$:
\begin{align*}
(\operatorname{id}\otimes_h\epsilon)\circ\copr\left(T^j_i(u)\right)&=(\operatorname{id}\otimes_h\epsilon)\sum_k T^k_i(u)\otimes_h T^j_k(u)\\
&=\sum_k T^k_i(u)\delta_{k,j}=T^j_i(u)\\
(\epsilon\otimes_h\operatorname{id})\circ\copr\left(T^j_i(u)\right)&=(\epsilon\otimes_h\operatorname{id})\sum_k T^k_i(u)\otimes_h T^j_k(u)\\
&=\sum_k \delta_{i,k} T^j_k(u)=T^j_i(u),
\end{align*}
\end{enumerate}
and similarly for the $(T^*)^{-i}_{-j}(u)$ type of generator and $\alpha(u)$.
\end{proof}
Since these maps respect the defining relations, they descend to well-defined maps on $A_{R,\omega,\omega^*}$.

Define further the antipode $S:A_{R,\omega,\omega^*}\rightarrow A_{R,\omega,\omega^*}^{\textrm{op}}$ as
\begin{align}\label{e-antip}
  S(T^j_i(u))&:= \xi_2^-(u-\lambda)\omega_{i,-i}(T^*)_{-j}^{-i}(u-\lambda)\sigma^{-j,j},\notag
  \\
S((T^*)_{-i}^{-j}(u))&:=\omega_{-i,i}^*T^i_j(u-\lambda)\sigma^{*j,-j},\\
  S(\xi_2^{\pm}(u))&:=\xi_2^{\mp}(u),\notag
  \\
S(\eta(1))&:=\eta(1).\notag
\end{align}
Graphically, it is given by
\begin{center}
$S\Bigg($\begin{tikzcd}
a \arrow[r,"l"] \arrow[d, "i"']\arrow[rd,phantom,"T(u)"]
& b \arrow[d, "j"] \\
c \arrow[r,"k"']
& d 
\end{tikzcd}$\Bigg)=\displaystyle\xi_2^-(u-\lambda)\omega_{i,-i}$\begin{tikzcd}
d \arrow[r, "k^{-1}"] \arrow[d, "j^{-1}"']\arrow[rd,phantom,"T^*(v)"]
& c \arrow[d, "i^{-1}"] \\
b \arrow[r,"l^{-1}"']
& a 
\end{tikzcd}$\sigma^{-j,j}$, 
\end{center}
\begin{center}
$S\Bigg($\begin{tikzcd}
a \arrow[r,"l^{-1}"] \arrow[d, "i^{-1}"']\arrow[rd,phantom,"T^*(u)"]
& b \arrow[d, "j^{-1}"] \\
c \arrow[r,"k^{-1}"']
& d 
\end{tikzcd}$\Bigg)=\displaystyle\omega_{-i,i}^*$\begin{tikzcd}
d \arrow[r, "k"] \arrow[d, "j"']\arrow[rd,phantom,"T(v)"]
& c \arrow[d, "i"] \\
b \arrow[r,"l"']
& a 
\end{tikzcd}$\sigma^{*j,-j}$, 
\end{center}
\begin{center}
$S\Bigg($\begin{tikzcd}
a \arrow[r,"k"] \arrow[d, "\mathrm{id}_a"']\arrow[rd,phantom,"\eta(1)"]
& b \arrow[d, "\mathrm{id}_b"] \\
a \arrow[r,"k"']
& b 
\end{tikzcd}$\Bigg)=$\begin{tikzcd}
b \arrow[r, "k^{-1}"] \arrow[d, "\mathrm{id}_b"']\arrow[rd,phantom,"\eta(1)"]
& a \arrow[d, "\mathrm{id}_a"] \\
b \arrow[r,"k^{-1}"']
& a 
\end{tikzcd}, 
\end{center}
with $v=u-\lambda$.
\begin{remark}
    The precise definition of antipode is
    \begin{align*}
        (\operatorname{id}\otimes S)(T_i^j(u))&:= \xi_2^-(u-\lambda)\omega_{i,-i}(T^*)_{-j}^{-i}(u-\lambda)\sigma^{-j,j},\\
(\operatorname{id}\otimes S)                                            ((T^*)_{-i}^{-j}(u))&:=\omega_{-i,i}^*T^i_j(u-\lambda)\sigma^{*j,-j},\\
    \end{align*}
    but we adopt the simplified notation for brevity.
\end{remark}
\begin{lemma}
\label{l3An}
The antipode $S$ satisfies the antipode axioms and preserves relations (\ref{eqn:1})-(\ref{eqn:8}).
\end{lemma}
\begin{proof}
The antipode axioms (H7) are described precisely by the identities (\ref{eqn:5})-(\ref{eqn:8}). To see this, first note that the projection map $A\otimes A\stackrel{p}\twoheadrightarrow A\otimes_vA$ sends $$T^j_i(u)\otimes T^l_k(v)\mapsto\delta_{t^\circ(\alpha_{ij}),s^\circ(\alpha_{kl})} T^j_i(u)\otimes_v T^l_k(v),$$ where $\alpha_{ij}$ denotes a morphism $\alpha\in\pi^2$ with $s^{\bullet}=i$ and $t^{\bullet}=j$, as in the picture: 
\begin{center}
\begin{tikzcd}
a \arrow[r,"n_1"] \arrow[d, "i"']\arrow[rd,phantom,"T(u)"]
& b \arrow[d, "j"] \\
c \arrow[r, "m_1"']
& d 
\end{tikzcd}$\otimes$ \begin{tikzcd}
e \arrow[r,"n_2"] \arrow[d, "k"']\arrow[rd,phantom,"T(v)"]
& f \arrow[d, "l"] \\
g \arrow[r, "m_2"']
& h 
\end{tikzcd}$\mapsto\delta_{m_1,n_2}$\begin{tikzcd}
a \arrow[r,"n_1"] \arrow[d, "i"']\arrow[rd,phantom,"T(u)"]
& b \arrow[d, "j"] \\
c \arrow[r, "m_1"']
& d 
\end{tikzcd}$\otimes_v$ \begin{tikzcd}
c \arrow[r,"m_1"] \arrow[d, "k"']\arrow[rd,phantom,"T(v)"]
& d \arrow[d, "l"] \\
g \arrow[r, "m_2"']
& h 
\end{tikzcd},
\end{center}
which upon the action of the product $\nabla$ becomes
\begin{center}
\begin{tikzcd}
a \arrow[r,"n_1"] \arrow[d, "i"']\arrow[rd,phantom,"T(u)"]
& b \arrow[d, "j"] \\
c \arrow[r, "m_1"']
& d 
\end{tikzcd}$\otimes_v$ \begin{tikzcd}
c \arrow[r,"m_1"] \arrow[d, "k"']\arrow[rd,phantom,"T(v)"]
& d \arrow[d, "l"] \\
g \arrow[r, "m_2"']
& h 
\end{tikzcd}$\mapsto$
\begin{tikzcd}
a \arrow[r, "n_1"] \arrow[d, "i"']\arrow[rd,phantom,"T(u)"]
& b \arrow[d, "j"]\\
c \arrow[d, "k"'] \arrow[r, "m_1"] \arrow[rd,phantom,"T(v)"]
& d \arrow[d, "l"]\\
g \arrow[r, "m_2"'] & h
\end{tikzcd}.
\end{center}
Therefore, we have 
\begin{align*}
&\nabla\circ p\circ (\operatorname{id}\otimes S)\circ i\circ \copr \left(T^j_i(u)\right)\\
&=\nabla\circ p\circ (\operatorname{id}\otimes S)\left(\sum_k T^k_i(u)\otimes T^j_k(u)\right)\\
&=\nabla\circ p\left(\sum_k T^k_i(u)\otimes \xi_2^-(u-\lambda)\omega_{k,-k}(T^*)^{-k}_{-j}(u-\lambda)\sigma^{-j,j}\right)\\
&=\delta_{t^{\circ}(\alpha_{i,k}),s^{\circ}(\alpha_{j^{-1},k^{-1}})}\xi_2^-(u-\lambda)\sum_k \omega_{k,-k}T^k_i(u)(T^*)^{-k}_{-j}(u-\lambda)\sigma^{-j,j}\\
&=\omega_{i,-j}\sigma^{-j,j}\xi_2^-(u-\lambda)\xi_2^+(u-\lambda)
=\delta_{i,j}\eta(1)=\eta\circ\sigma^\prime\circ\epsilon\left(T^{j}_i(u)\right)
\end{align*}
and
\begin{align*}
&\nabla\circ p\circ (\operatorname{id}\otimes S)\circ i\circ \copr \left((T^*)^{-j}_{-i}(u)\right)\\
&=\nabla\circ p\circ (\operatorname{id}\otimes S)\left(\sum_k (T^*)^{-k}_{-i}(u)\otimes(T^*)^{-j}_{-k}(u)\right)\\
&=\nabla\circ p\left(\sum_k (T^*)^{-k}_{-i}(u)\otimes\omega^*_{-k,k}T^{k}_{j}(u-\lambda)\sigma^{*j,-j}\right)\\
&=\delta_{t^{\circ}(\alpha_{i^{-1},k^{-1}}),s^{\circ}(\alpha_{j,k})}\sum_k \omega^*_{-k,k} (T^*)^{-k}_{-i}(u)T^{k}_{j}(u-\lambda)\sigma^{*j,-j}\\
&=\omega^*_{-i,j}\sigma^{*j,-j}\eta(1)
=\delta_{i,j}\eta(1)=\eta\circ\sigma^\prime\circ\epsilon\left((T^*)^{-j}_{-i}(u)\right).
\end{align*}
Moreover, we have $\nabla\circ p\circ (\operatorname{id}\otimes S)\circ i\circ \copr(\xi_2^{\pm}(u))=\xi_2^{\pm}(u)\xi_2^{\mp}(u)=\eta(1)=\eta\circ\sigma^\prime\circ\epsilon(\xi_2^{\pm}(u))$.
The other identity follows in a similar way. From these axioms one can in a standard way prove that the map $S$ is an antihomomorphism.

Furthermore, we show that the antipode preserves the $RTT$-relations. To do so, we first observe that from the rotation relations (\ref{Rrel}) together with $\sigma^{-i,i}\omega_{i,-i}=1$ and  follow the identities
\begin{align*}
    \sigma^{-l,l}\sigma^{-k,k}(\R_{VV})_{k,l}^{m,n}(u)&=(\R_{V^*V^*})_{-n,-m}^{-l,-k}(u)\sigma^{-m,m}\sigma^{-n,n},\\
    (\R_{VV})_{k,l}^{m,n}(u)\omega_{n,-n}\omega_{m,-m}&=\omega_{k,-k}\omega_{l,-l}(\R_{V^*V^*})_{-n,-m}^{-l,-k}(u).
\end{align*}
Now we apply the antipode on the relation (\ref{eqn:1}):
\begin{align*}
&\sum_{k,l}(\R_{VV})^{m,n}_{k,l}(u_1-u_2)S\left(T^{k}_i(u_1)T^{l}_j(u_2)\right)\\
&\quad -\sum_{k,l}S\left(T^{m}_k(u_2)T^{n}_l(u_1)\right)(\R_{VV})^{k,l}_{i,j}(u_1-u_2)\\
&=\sum_{k,l}\omega_{j,-j}(T^*)^{-j}_{-l}(v_2)\sigma^{-l,l} \omega_{i,-i}(T^*)^{-i}_{-k}(v_1)\sigma^{-k,k}(\R_{VV})^{m,n}_{k,l}(v_1-v_2)\\
&\quad-\sum_{k,l}(\R_{VV})^{k,l}_{i,j}(v_1-v_2)\omega_{l,-l}(T^*)^{-l}_{-n}(v_1)\sigma^{-n,n}\omega_{k,-k}(T^*)^{-k}_{-m}(v_2)\sigma^{-m,m}\\
&=\omega_{i,-i}\omega_{j,-j}\Bigg(\sum_{k,l}(T^*)^{-j}_{-l}(v_2) (T^*)^{-i}_{-k}(v_1)(\R_{V^*V^*})^{-l,-k}_{-n,-m}(v_1-v_2)\\
&\quad -\sum_{k,l}(\R_{V^*V^*})^{-j,-i}_{-l,-k}(v_1-v_2)(T^*)^{-l}_{-n}(v_1)(T^*)^{-k}_{-m}(v_2)\Bigg)\sigma^{-m,m}\sigma^{-n,n}=0,
\end{align*}
where we set $v_i:=u_i-\lambda$, $i=1,2$ and we use the relation (\ref{eqn:3}). 

On the other hand, using the definition of $\R_{V^*V^*}'(u)$ and Lemma \ref{rotpi}, one can see that the action of antipode on the relation (\ref{eqn:3}) gives back the relation (\ref{eqn:1}). 

To show that the relation (\ref{eqn:4}) is preserved upon the action of antipode, we first observe that due to Corollary \ref{rotcor}, we have the following identities
\begin{align*}
    \sigma^{*k,-k}\sigma^{-l,l}(\R_{VV^*})_{l,-k}^{-j,i}(u)&=(\R_{V^*V})_{-i,j}^{k,-l}(u)\sigma^{*j,-j}\sigma^{-i,i},\\
    (\R_{VV^*})_{l,-k}^{-j,i}(u)\omega_{i,-i}\omega^*_{-j,j}&=\omega_{l,-l}\omega^*_{-k,k}(\R_{V^*V})_{-i,j}^{k,-l}(u),\\
    \sigma^{-j,j}\sigma^{*i,-i}(\R_{V^*V})_{-i,j}^{k,-l}(u)&=(\R_{VV^*})_{l,-k}^{-j,i}(u)\sigma^{-k,k}\sigma^{*l,-l},\\ 
    (\R_{V^*V})_{-i,j}^{k,-l}(u)\omega^*_{-l,l}\omega_{k,-k}&=\omega^*_{-i,i}\omega_{j,-j}(\R_{VV^*})_{l,-k}^{-j,i}(u).
\end{align*}
Now, since $\R_{VV}(u)$ is quasi-unitary, we have the inversion identity $\R_{VV^*}(u)=\alpha(u)\R_{V^*V}^{-1}(-u)$ due to Lemma \ref{inversion}, and we can write the relation (\ref{eqn:4}) in the following form:
\begin{equation*}
\begin{split}
\R_{VV^*}^{(23)}(u_1-u_2)T^{(12)}(u_1)T^{*(23)}(u_2)
&=T^{*(12)}(u_2)T^{(23)}(u_1)\R_{VV^*}^{(12)}(u_1-u_2),
\end{split}
\end{equation*}
which can be written as
\begin{equation*}
\begin{split}
 \sum_{k,l}(\check R_{VV^*})^{-m,n}_{k,-l}(u_1-u_2)T^{k}_i(u_1)(T^*)^{-l}_{-j}(u_2)\\
 =\sum_{k,l}(T^*)^{-m}_{-k}(u_2)T^{n}_l(u_1)(\check R_{VV^*})^{-k,l}_{i,-j}(u_1-u_2).
 \end{split}
 \end{equation*}
The action of the antipode on this relation with the help of the above identities for $\R_{VV^*}(u)$ and $\R_{V^*V}(u)$ gives the relation (\ref{eqn:4}) and vice versa.

Next, we prove that the identities (\ref{eqn:5})--(\ref{eqn:8}) are preserved. Applying the antipode on the left-hand side of equation (\ref{eqn:5}) gives
\begin{align*}
&\sum_{kl}\omega_{k,-l}S\left(T^k_i(u+\lambda)(T^*)_{-j}^{-l}(u)\right)\\
&=\sum_{kl}\omega^*_{-j,j}T^{j}_{l}(u-\lambda)\sigma^{*l,-l}\xi_2^-(u)\omega_{i,-i}(T^*)_{-k}^{-i}(u)\omega_{k,-l}\sigma^{-k,k}\\
&=\omega_{i,-i}\omega^*_{-j,j}\sum_{kl}T^{j}_{l}(u-\lambda)(T^*)_{-k}^{-i}(u)\sigma^{*l,-l}\delta_{k,l}\xi_2^-(u)\\
&=\omega_{i,-i}\omega^*_{-j,j}\sigma^{*j,-i}\xi_2^-(u)=\omega_{i,-j}S(\xi_2^+(u)).
\end{align*}
It is useful to have a graphical understanding of this proof:
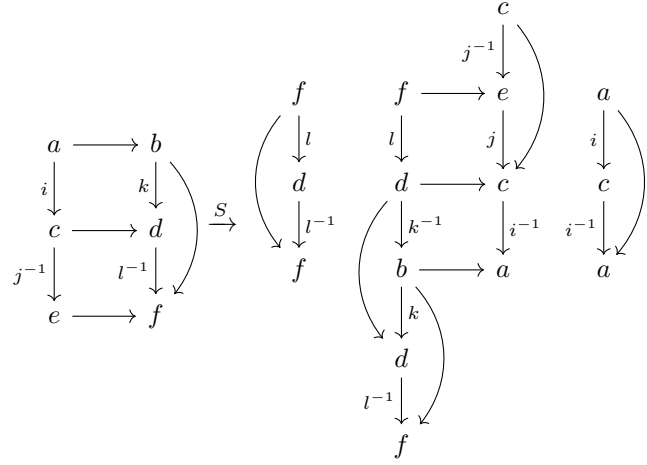
\begin{center}
\begin{tikzcd}
a \arrow[r]\arrow[d,"i"'] & b \arrow[d,"k"']\arrow[dd,bend left=40]\\
c \arrow[d,"j^{-1}"']\arrow[r] & d\arrow[d,"l^{-1}"']\\
e \arrow[r] & f
\end{tikzcd}$\xrightarrow{S}$
\begin{tikzcd}
 &  & c\arrow[d,"j^{-1}"']\arrow[dd,bend left=40] &  \\
f \arrow[dd,bend right=40] \arrow[d,"l"] & f \arrow[r]\arrow[d,"l"'] & e \arrow[d,"j"'] & a\arrow[d,"i"']\arrow[dd,bend left=40]\\
d\arrow[d,"l^{-1}"] & d \arrow[d,"k^{-1}"]\arrow[r]\arrow[dd,bend right=40] & c\arrow[d,"i^{-1}"] & c\arrow[d,"i^{-1}"']\\
f & b \arrow[r]\arrow[d,"k"] \arrow[dd,bend left=40]& a & a \\
 & d\arrow[d,"l^{-1}"'] & & \\
 & f & &
\end{tikzcd}
\end{center}
Now we can contract $\omega_{k,-l}\sigma^{-k,k}=\delta_{k,l}$, which further allows the action of $\sigma^{*l,-l}$:
\begin{center}
    \begin{tikzcd}
 & c\arrow[d,"j^{-1}"']\arrow[dd,bend left=40] &  \\
  f \arrow[r]\arrow[d,"l"] \arrow[dd,bend right=40] & e \arrow[d,"j"'] & a\arrow[d,"i"']\arrow[dd,bend left=40]\\
 d \arrow[d,"l^{-1}"]\arrow[r] & c\arrow[d,"i^{-1}"] & c\arrow[d,"i^{-1}"']\\
 f \arrow[r]& a & a 
\end{tikzcd}$=$
\begin{tikzcd}
 c\arrow[d,"j^{-1}"']\arrow[dd,bend left=40] &  \\
 e\arrow[dd,bend right=40] \arrow[d,"j"'] & a\arrow[d,"i"']\arrow[dd,bend left=40]\\
 c\arrow[d,"i^{-1}"] & c\arrow[d,"i^{-1}"']\\
 a & a 
\end{tikzcd},
\end{center}
where the equality comes from relation (\ref{eqn:8}). The left diagram on the right-hand side corresponds to contraction $\omega^*_{-j,j}\sigma^{*j,-i}=\delta_{i,j}$ which together with the right diagram gives \begin{tikzcd}
 a\arrow[d,"i"']\arrow[dd,bend left=40]\\
 c\arrow[d,"j^{-1}"']\\
 e 
\end{tikzcd}$=\omega_{i,-j}$.

The proof for other relations follows in an analogous way.
\end{proof}
\begin{theorem}
The restricted quantum algebra $A_{R,\omega,\omega^*}$ equipped with $(\copr,\nabla, \epsilon,\eta,S)$ is a double groupoid graded Hopf algebra.
\end{theorem}
\begin{proof}
Bialgebra axioms are satisfied by construction. This, together with Lemma \ref{l2An} and Lemma \ref{l3An} proves the theorem.
\end{proof}
\begin{definition}
We call the $\pi^2$-graded Hopf algebra $(A_{R_{VV},\omega,\omega^*},\copr,\nabla, \epsilon,\eta,S)$ the {\em restricted quantum group of general linear  type} associated with $\R_{VV}(u)$, $\omega, \omega^*$.
\end{definition}
Our main example is given by the $R$-matrix of the restricted $A_{n-1}$-model of Section \ref{sec-4}. The rotations are
introduced in Section \ref{4.3} and the crossing symmetry is the content of Lemma \ref{wttan} with $\lambda=-n/2$.
\begin{definition}
  Let $\ell\in\mathbb Z_{\geq0}$, $n\geq2$, $L=\ell+n$, $\lambda=-n/2$,
  and $\check R_{VV}$ be the $R$-matrix of the restricted $A_{n-1}^{(1)}$-model of
  Section \ref{sec-4} with $\omega,\omega^*$ given by \eqref{e-omega*}. The $\pi^2$-graded Hopf algebra associated
  with this data is called restricted elliptic quantum group $E_\ell(A_{n-1}^{(1)})$ of type $A_{n-1}^{(1)}$ of level $\ell$.
\end{definition}
\begin{remark} For this algebra one can construct, as in \cite{FelderVarchenkoJSP1997},
  a quantum determinant which is central and group-like. This will be discussed elsewhere.
\end{remark}
\subsection{Restricted quantum groups of orthogonal type}
    Let $V\in\operatorname{Vect}^f_{\mathbb{C}}(\pi)$ be a finite dimensional $\pi$-graded vector space and let $\tilde\pi\subseteq\pi $ denote the set that contains all the arrows $k\in\pi_1$ with $V_k\neq\{0\}$ and their inverses $k^{-1}\in\pi_1.$ Consider the generators $ T^i_j(u)\in\operatorname{Hom}_{\mathbb{C}}(V_{j},V_{i})\otimes A_{\alpha}$ described in the Section \ref{sec-5.1}.
\begin{definition}\label{def2}
  Let $\check R(u)\in \operatorname{End}(V \otimes V )$ be a rotationally invariant $R$-matrix on the $\pi$-graded vector space $V$ with inner product $\omega$ and crossing parameter $\lambda$.
  The {\em restricted quantum algebra} $A_{R,\omega}$ {\em of orthogonal type} associated with $R$, $\omega$
  is the $\pi^2$-graded unital algebra $A=\oplus_{\alpha\in\pi^2}A_{\alpha}$ generated by matrix elements $T^i_j(u)$,
$i,j\in\tilde\pi$, of $T(u)\in\operatorname{Hom}_{\operatorname{Vect}_{\mathbb{C}}(\pi)}(V,A\otimes_h V)$ and central elements $\xi^+(u),\xi^-(u)\in\oplus_{f\in\pi_1} A_{\mathrm{id}_{f}^\circ}$,
subject to relations
\begin{equation}
\label{xi}
    \xi^{\pm}(u)\xi^{\mp}(u)=\eta(1),
\end{equation}
\begin{equation}
\label{eqn:RTTc}
\check R^{(23)}(u_1-u_2)T^{(1)}(u_1)T^{(2)}(u_2)=T^{(1)}(u_2)T^{(2)}(u_1)\check R^{(12)}(u_1-u_2),
 \end{equation}
 \begin{equation}
 \label{TTwc}
\omega^{(23)}T^{(1)}(u+\lambda)T^{(2)}(u)=\xi^+(u)\omega^{(12)},
\end{equation}
\begin{equation}
\label{TTsc}
T^{(1)}(u)T^{(2)}(u+\lambda)\sigma^{(12)}=\xi^+(u)\sigma^{(23)},
\end{equation}
\end{definition}

We define two maps, the co-unit $\epsilon: A_{R,\omega}\rightarrow I_h$ and the coproduct $\Delta :A_{R,\omega}\rightarrow A_{R,\omega}\otimes_h A_{R,\omega}$ by
\begin{align*}
\epsilon (T_{\textrm{id}^\bullet_i}(u)):=1,\qquad\epsilon(\xi^\pm(u)):=1,
\end{align*}
and
\begin{align*}
\copr (T^i_j (u)):=\sum_k T^k_j(u)\otimes_h T^i_k(u),\qquad\copr(\xi^\pm(u)):=\xi^\pm(u)\otimes_h\xi^\pm(u),
\end{align*} 
where $k$ goes over all morphisms such that $t^\circ (\beta)\circ t^\circ(\alpha)=t^\circ(\gamma)$ for $\alpha\in\Gamma^\bullet(i,k)$, $\beta\in\Gamma^\bullet(k,j)$, $\gamma\in\Gamma^\bullet(i,j)$.
\begin{lemma}
\label{l2}
The co-unit $\epsilon$ and the coproduct $\copr$ preserve relations (\ref{eqn:RTTc}), (\ref{TTwc}), (\ref{TTsc}) and satisfy the coalgebra axioms.
\end{lemma}
\begin{proof}
The proof is the same as in the case of the restricted quantum algebra $A_{R,\omega,\omega^*}$, which makes co-unit and coproduct well-defined algebra maps on $A_{R,\omega}$.
\end{proof}

Define the antipode $S:A_{R,\omega}\rightarrow A_{R,\omega}^{\textrm{op}}$ to be
\begin{align}\label{e-antip2}
  S(T^j_i(u))&:=\xi^-(u-\lambda)\omega_{i,-i}T_{-j}^{-i}(u-\lambda)\sigma^{-j,j},
  \\
  S(\xi^\pm(u))&:=\xi^\mp(u)\notag
\end{align}
whose components can be depicted as
\begin{center}
$S\Bigg($\begin{tikzcd}
a \arrow[r,"n"] \arrow[d, "i"']\arrow[rd,phantom,"T(u)"]
& b \arrow[d, "j"] \\
c \arrow[r, "m"']
& d 
\end{tikzcd}$\Bigg)=\displaystyle\xi^-(u-\lambda)\omega_{i,-i}$\begin{tikzcd}
d \arrow[r, "m^{-1}"] \arrow[d, "j^{-1}"']\arrow[rd,phantom,"T(v)"]
& c \arrow[d, "i^{-1}"] \\
b \arrow[r,"n^{-1}"']
& a 
\end{tikzcd}$\sigma^{-j,j}$, $v=u-\lambda$.
\end{center}
\begin{lemma}
\label{l3}
The antipode $S$ satisfies the antipode axioms and preserves the relations (\ref{eqn:RTTc}), (\ref{TTwc}), and (\ref{TTsc}).
\end{lemma}
\begin{proof}
The identities (\ref{TTwc}) and (\ref{TTsc}) describe the antipode axioms (H7) $\nabla (\textrm{id}\otimes S)\copr=\eta\circ\sigma^\prime\circ\epsilon$ and $\nabla (S\otimes\textrm{id})\copr=\eta\circ\sigma^\prime\circ\epsilon$ in the same way as in the previous subsection. Thus, the map $S$ is an antihomomorphism. Moreover, using the rotational invariance of $R$-matrix, given by Definition \ref{Rrs}, one can show that the antipode respects the defining relation (\ref{eqn:RTTc}). It also preserves relations (\ref{TTwc}) and (\ref{TTsc}), with the same proof as in the case of the restricted quantum group of general linear type. 
\end{proof}
\begin{theorem}
The restricted quantum algebra $A_{R,\omega}$ equipped with $(\copr,\nabla, \epsilon,\eta,S)$ is a double groupoid graded Hopf algebra.
\end{theorem}
\begin{proof}
Lemma \ref{l2} and Lemma \ref{l3} prove the theorem.
\end{proof}
\begin{definition}\label{def-A}
We call the $\pi^2$-graded Hopf algebra $(A_{R,\omega},\copr,\nabla, \epsilon,\eta,S)$ the \textit{restricted quantum group of orthogonal type} associated with $\R(u)$ and $\omega$.
\end{definition}
The main example is given by the $R$-matrices of the restricted $B_n^{(1)},C_n^{(1)},D_n^{(1)}$-models of Section \ref{sec-4}. The rotations are
introduced in Section \ref{4.3} and the crossing symmetry is the content of Lemma \ref{lWWR} with $\lambda$ given in Table \ref{tab}.
\begin{definition}\label{def-BCD}
  Let $\mathfrak g=B_n,C_n$ or $D_n$, $\ell\in\mathbb Z_{\geq0}$, $L=\frac12(\theta,\theta)(\ell+g)$, $\lambda=-g(\theta,\theta)/4$ as in Table \ref{tab},
  $\check R$ be the $R$-matrix of the restricted models of
 type $\mathfrak g^{(1)}$ of  Section \ref{sec-4} with $\omega$ given in \eqref{e-omega}. The $\pi^2$-graded Hopf algebra associated
  with these data is called the restricted elliptic quantum group $E_\ell(\mathfrak g^{(1)})$ of type $\mathfrak g^{(1)}$ of level $\ell$.
\end{definition}
\subsection{Automorphisms and the square of the antipode}\label{ss-auto}
We describe three classes of automorphisms of our $\pi^2$-graded Hopf algebras. The first two act trivially
on the central elements $\xi^\pm(u),\xi_2^\pm(u)$ and the square of the antipode is a composition of these.

The first class comes from the grading and applies to any $\pi^2$-graded Hopf algebra $A$.
We have an automorphism $D_\gamma$ for each morphism of groupoids $\gamma\colon \pi\to\mathbb C^\times$, where the group $\mathbb C^\times$ is regarded as a groupoid with one object. Thus $\gamma$ is a map $\pi_1\to \mathbb C^\times$ such
that $\gamma(i\circ j)=\gamma(i)\gamma(j)$ for composable arrows $i,j$ and $\gamma(e)=1$ for identity arrows $e$.
For example
for each map $g\colon \pi_0\to \mathbb C^\times $, $\gamma(i)=g(t(i))/g(s(i))$ is such a morphism.

\begin{lemma} Let $A$ be a $\pi$-graded Hopf algebra and $\gamma\colon\pi\to \mathbb C^\times$ a morphism of groupoids.
  Then the map $D_\gamma$ acting as multiplication by $\gamma(s^\bullet(\alpha))/\gamma(t^\bullet(\alpha))$ on $A_\alpha$,
  $\alpha\in \Gamma^\bullet$ (see Section \ref{ss-2.3}) is a Hopf algebra automorphism. 
\end{lemma}
\begin{proof}
  The product is a map $A_\alpha\otimes_v A_\beta\to A_{\beta\circ\alpha}$ and $s^\bullet(\beta\circ \alpha)=
  s^\bullet(\beta)\circ s^\bullet(\alpha)$. Therefore $\gamma(s^\bullet(\beta\circ \alpha)=
  \gamma(s^\bullet(\beta))\gamma (s^\bullet(\alpha))$. The same holds for the target, so that
  $D_\gamma(a\cdot b)=D_\gamma(a)\cdot D_\gamma(b)$ for any $a\in A_\alpha$, $b\in A_\beta$. The
  coproduct is a map $\Delta\colon A_\alpha\to\oplus_{\beta_2\bullet\beta_1=\alpha} A_{\beta_1}\otimes_h A_{\beta_2}$.
  Since for each term in the direct sum
  $s^\bullet(\beta_1))=s^\bullet(\alpha)$, $t^\bullet(\beta_2)=t^\bullet(\alpha)$ and $t^\bullet(\beta_1)=s^\bullet(\beta_2)$,
  the middle term cancels and $\Delta\circ D_\gamma(a)=D_\gamma\otimes D_\gamma\circ\Delta(a)$ for any $a\in A_\alpha$.
  Clearly $D_\gamma$ preserves the unit and the co-unit. The map $D_\gamma$ also commutes with the
  antipode since $S$ maps $A_\alpha$ to $A_{\iota(\alpha)}$
  and $s^\bullet(\iota(\alpha))=(t^\bullet)^{-1}(\alpha)$, $t^\bullet(\iota(\alpha))=(s^{\bullet})^{-1}(\alpha)$,
  see \eqref{e-double-inverse}.
\end{proof}

  In particular, if $\gamma(i)=g(t(i))/g(s(i))$ for some function $g\colon\pi_0\to  \mathbb C^\times$, then
\begin{equation}\label{e-abcd}
    D_\gamma(x)=\frac{g(c)g(b)}{g(a)g(d)}x,\text{ for $x\in A_\alpha$, where }
      \alpha=\begin{tikzcd}
a \arrow[r,"l"] \arrow[d, "i"']\arrow[rd,phantom]
& b \arrow[d, "j"] \\
c \arrow[r,"k"']
& d 
\end{tikzcd},
\end{equation}
is an automorphism of the Hopf algebra.

The second type of automorphisms is defined for
Hopf algebras given by relations of the form \eqref{xiA}--\eqref{eqn:8} and \eqref{xi}--\eqref{TTsc}.
\begin{lemma}
  There is a unique automorphism $\psi$ of the Hopf algebra $A_{R,\omega,\omega^*}$ or $A_{R,\omega}$
  acting as the identity on $\xi^\pm(u)$, $\xi_2^\pm(u)$ and such that
  \[
    \psi(T^j_i(u))=\xi_2^+(u-\lambda)T^j_i(u-2\lambda),\quad
    \psi((T^*)^{-i}_{-j}(u))=\xi_2^-(u-2\lambda)(T^*)^{-i}_{-j}(u-2\lambda),
  \]
  where the second formula applies to the general linear case and, in the orthogonal case,
  \[
    \xi^+_2(u)=\xi^-(u-\lambda)\xi^+(u).
  \]
\end{lemma}
\begin{proof}
  We need to check that $\psi$ respect the relations. The non-trivial checks are for the relations
  \eqref{eqn:5}--\eqref{eqn:8} in the general linear case and \eqref{TTwc}, \eqref{TTsc} in the
  orthogonal case.
  For example in the relation \eqref{eqn:5}
  \[
    \omega^{(23)}T^{(1)}(u+\lambda)T^{*(2)}(u)=\omega^{(12)}\xi_2^+(u),
  \]
  the left-hand side is sent by $\psi$ to
  \begin{align*}
    \xi_2^+(u)\xi_2^-(u-2\lambda)\omega^{(23)}T^{(1)}(u-\lambda)T^{*(2)}(u-2\lambda)
    &= \xi_2^+(u)\xi_2^-(u-2\lambda)\omega^{(12)}\xi_2^+(u-2\lambda)
      \\
    &=\xi_2^+(u)\omega^{(12)}.
  \end{align*}

  The other cases are similar.

\end{proof}
\begin{theorem}\label{t-S2} Let $\mathfrak g$ be of type $A,B,C,D$.
  For the restricted elliptic quantum algebra $E_\ell(\mathfrak g^{(1)})$, the square of the antipode is
  \[
    S^2=D_\gamma\circ \psi,
  \]
  where $\gamma$ is as in \eqref{e-abcd} with $g(a)=G_a$, see \eqref{e-Ga}.
\end{theorem}
To prove this it is sufficient to apply $S^2$ on the generators by using the definitions \eqref{e-antip},
\eqref{e-antip2}. One finds
\[
  S^2(T^j_i(u))=\xi_2^+(u-\lambda)\frac{G_{s(j)}G_{t(i)}}{G_{t(j)}G_{s(i)}}T^j_i(u-2\lambda),
\]
and, in the general linear case,
\[
  S^2((T^*)^{-i}_{-j}(u))=\xi_2^-(u-2\lambda)\frac{G_{s(j)}G_{t(i)}}{G_{t(j)}G_{s(i)}}(T^*)^{-i}_{-j}(u-2\lambda).
\]
Finally, we have (as in the case of Yangians and quantum affine algebras)
the shift by $z\in\mathbb C$ of the spectral parameter for Hopf algebras $A_{R,\omega,\omega^*}$, $A_{R,\omega}$.

Since the structure constants are functions of the difference $u_1-u_2$, the maps on generators
\[
  \tau_z(T^j_i(u))=T^j_i(u-z),\quad \tau_z((T^*)^{-i}_{-j}(u))=(T^*)^{-i}_{-j}(u-z)\quad \tau_z(\xi_2^{\pm}(u))=\xi_2^{\pm}(u-z)
\]
in the general linear case and
\[
  \tau_z(T^j_i(u))=T^j_i(u-z),\quad \tau_z((T^*)^{-i}_{-j}(u))=(T^*)^{-i}_{-j}(u-z)\quad \tau_z(\xi^\pm(u))=\xi^\pm(u-z)
\]
in the orthogonal case extend to automorphisms of the Hopf algebra.

\section{Representations of the restricted quantum groups}
We consider the setting described in the Section \ref{4.2}. The $\pi$-graded vector spaces $V^{\textrm{RSOS}}$ and $(V^{*})^{\textrm{RSOS}}$ together with the restrictions of $R$-matrices and their rotations, given in Sections \ref{4.1} and \ref{4.3}, on these spaces provide the main example of the restricted quantum groups $E_\ell(\mathfrak g^{(1)})$ associated with Lie algebras $\mathfrak{g}=\mathfrak{sl}_n$, $\mathfrak{so}_n$ and $\mathfrak{sp}_{2n}$, see Definitions \ref{def-A}, \ref{def-BCD}.  
\subsection{Vector representations}
\begin{theorem}
  Let $\ell\in\mathbb Z_{\geq0}$, $n\in\mathbb Z_{\geq2}$. The quantum group
  $E_\ell(A_{n-1})$  has a
  representation $V=(V^{\textrm{RSOS}},\sigma_V)$ on the $\pi$-graded vector space $V^{\textrm{RSOS}}$  given by
\begin{align*}
   &\sigma_V (T_i^j(u))=(\R_{VV})_i^j(u),\quad\sigma_V((T^*)_{-i}^{-j}(u))= (\R_{V^*V})_{-i}^{-j}(u),\\
   &\sigma_V(\xi_2^\pm(u))=\rho_2(u)^{\pm1}, 
\end{align*}
with $\R_{VV}(u)$, $\R_{V^*V}(u)$ and $\rho_2(u)$ given by (\ref{RAn}), (\ref{rstar}) and \eqref{e-rho2},
and the operators $(\R_{VV})_i^j(u)$ and $(\R_{V^*V})_{-i}^{-j}(u)$ defined as
\begin{align*}
(\R_{VV})_i^j(u)e_k=&\sum_{\substack{l,\\ \varepsilon_i+\varepsilon_k=\varepsilon_l+\varepsilon_j}}(\R_{VV})_{i,k}^{l,j}(u)e_l,\\
(\R_{V^*V})_{-i}^{-j}(u)e_k=&\sum_{\substack{l,\\ -\varepsilon_{i}+\varepsilon_k=\varepsilon_l-\varepsilon_j}}(\R_{V^*V})_{-i,k}^{l,-j}(u)e_l.
\end{align*}
\end{theorem}
\begin{proof}
This is the result of Lemmas \ref{lYB}, \ref{l513}, and \ref{wttan}, \ref{Rvvstar}. 
\end{proof}
\begin{theorem} Let $\ell\in\mathbb Z_{\geq0}$.
The $R$-matrix operators $\R_i^j(u)$ defined as $\R_i^j(u)e_k=\sum_l\R_{i,k}^{l,j}(u)e_l$, where $ \R(u)$ is given in (\ref{eqn:RM}), together with $\rho(u)$, given in \eqref{e-rho}, define a representation $V=(V^{\textrm{RSOS}},\sigma_V)$ of the restricted quantum group $E_\ell(\mathfrak{g}^{(1)})$ on the $\pi$-graded vector space $V^{\textrm{RSOS}}$ for $\mathfrak{g}=B_n$, $C_n$ and $D_n$:
\begin{align*}
    \sigma_V(T_i^j(u))=\R_i^j(u),\qquad\sigma_V(\xi^{\pm}(u))=\rho(u)^{\pm 1}.
\end{align*}
\end{theorem}
\begin{proof}
The statement follows directly from Lemmas
\ref{lYB} and \ref{lWWR}.
\end{proof}
Let us call $V$ the vector representation of the restricted quantum group $E_\ell(\mathfrak g^{(1)})$ for $\mathfrak g$ of type $A,B,C,D$.
\begin{prop}
  In the case of $B_n^{(1)},D_n^{(1)}$ the vector representation representation splits as a direct
  sum of two submodules
  \[
    V=V^+\oplus V^-
  \]
  where $V^+$ is the direct sum of the homogeneous components of $V$
  of degree $(a,\epsilon_i)$ with $a\in\mathbb Z^n$ and
  $V^-$ is the direct sum of the homogeneous components of $V$
  of degree $(a,\epsilon_i)$ with $a\in(\frac12+\mathbb Z)^n$.    
\end{prop}
\begin{proof}
  Let $P_V\subset P$ be the sublattice spanned by the weights of the
  vector representation. It is generated by $\epsilon_1,\dots,\epsilon_n$ for
  $A_{n-1},B_n,C_n,D_n$. For $A_n,C_n$,
  $P_V=P$. For $B_n$ and $D_n$, $P_V=\mathbb Z^n$ and
  $P=\{a\in\mathbb (\frac 12 Z)^n\,|\, a_i-a_j\in\mathbb Z\}$, see
  Section \ref{4.2}.  In these cases $P_V$ is the kernel of the map 
  $P\mapsto \mathbb Z/2\mathbb Z$, $a\mapsto 2a_1\operatorname{mod} 2$ and
  is thus a subgroup of index two. The set of objects of $\pi$ is the union
  of the two orbits $\pi_0^\pm$ of $P_V$ on $P$, given by weights with integer
  and  half-integer coordinates respectively. Let $\pi^\pm$ be the corresponding full
  subgroupoids of $\pi$. The generators $T^i_j(u)$ with non-trivial image by $\sigma_V$ are labeled by commutative squares in either $\pi^+$ or $\pi^-$. So they
  preserve $V^+$ and $V^-$.
\end{proof}
As in the case of Yangians and quantum affine algebras we can twist a representation by a
shift automorphism, see Section \ref{ss-auto}, to obtain new representations. For $z\in\mathbb C$ let
$\tau_z$ be the automorphism acting on generators as $T^{i}_{j}(u)\mapsto T^{i}_{j}(u-z)$,
$\xi^\pm(u)\mapsto \xi^\pm(u-z)$, $\xi_2^\pm(u)\mapsto \xi_2^\pm(u-z)$.
Then for any representation $W=(W,\sigma_W)$ of $E_\ell(\mathfrak g^{(1)})$,
$\sigma_W\circ\tau_z$ is also a representation on the same underlying vector space $W$.
We denote it by $W(z)$. Since $\tau_{z+w}=\tau_z\circ\tau_w$, $W(z)(w)=W(z+w)$.
\begin{definition}
  The representation $V(z)=(V^{\textrm{RSOS}},\sigma_V\circ\tau_z)$ of
  $E_\ell(\mathfrak g^{(1)})$ is called the
  \emph{vector representation with spectral parameter $z$}.
\end{definition}
\subsection{Scalar representations and duality}
Let $\mathcal O=\mathbb C[[z,p/z]]$, with $z=e^{2\pi iu},p=e^{2\pi i\tau}$ be the ring of formal Fourier series
(see Section \ref{ss-2.4}). For each invertible element $f\in\mathcal O^\times$ of $\mathcal O$ we have a representation $(\mathbb C_f,\sigma_f)$ with underlying vector space with non-trivial weight spaces
\[
(\mathbb C_f)_{\mathrm{id}_a}=\mathbb C, \quad a\in \pi_0.
\]
In the orthogonal case $\mathfrak g=B_n,C_n,D_n$, the action of generators is
\[
\sigma_f(T_j^i(u))=f(u)\delta_{ij}, \quad \sigma_f(\xi^\pm(u))=f(u)^{\pm1}f(u+\lambda)^{\pm1}.
\]
In the general linear case $\mathfrak g=A_n$, it is
\begin{align*}
\sigma_f(T_j^i(u))&=f(u)\delta_{ij}, \quad \sigma_f(T^*)_{_j}^{-i}(u))=f(u-\lambda)^{-1}\delta_{ij},
\\
\sigma_f(\xi_2^\pm(u))&=\frac{f(u+\lambda)^{\pm1}}{f(u-\lambda)^{\pm1}}.    
\end{align*}
\begin{lemma}
    These formulas define a representation of the restricted elliptic quantum groups $E_\ell(\mathfrak g^{(1)})$. For $f=1$ we obtain the trivial representation (given by the co-unit) and for any $f,g\in\mathcal O^\times$,
    \[
    \mathbb C_f\otimes\mathbb C_g\cong\mathbb C_{fg}.
    \]
\end{lemma}
\begin{proof} The proof is a verification of the relations, see Definitions \ref{def-gltype}, \ref{def2}.
\end{proof}

\begin{prop}\label{p-dual}\ 
 \begin{enumerate}
\item[(i)] Let $W$ be a representation of $E_\ell(\mathfrak g^{(1)})$, such that $\xi_2(u)$ acts by a scalar $\rho_2(u)\in\mathcal O$. Then for any $z\in\mathbb C$,
\[
W(z)^{**}\cong W(z+2\lambda)\otimes \mathbb C_{f},\quad f(u)=\rho_2(u-z-\lambda).
\] 
  \item[(ii)] In the $B_n^{(1)},C_n^{(1)},D_n^{(1)}$ cases, the vector representation obeys
    \[
      V(z)\cong V(z-\lambda)^*\otimes \mathbb C_{f},\quad f(u)=\rho(u-z).
    \]
  \end{enumerate}
\end{prop}

\begin{proof}\ 
\begin{enumerate}
    \item[(i)] The representation $W(z)^{**}$ has the same underlying $\pi$-graded vector space $W$
as $W(z)$ but the action
of the algebra is twisted by the automorphism $S^2$. By Theorem \ref{t-S2}, the map
\[
  \phi\colon W_\mu\ni w\mapsto \frac{G_{a}}{G_{b}}w\otimes 1\in (W\otimes\mathbb C_f)_\mu =W_\mu\otimes  (\mathbb C_{f})_{\mathrm{id}_b},\quad \mu\in\pi(a,b),
\]
with $f(u)=\rho_2(u-z-\lambda)$
is an isomorphism $W(z)^{**}\to W(z+2\lambda)\otimes\mathbb C_f$.
    \item[(ii)] The relations (\ref{eqn:invrel}) and (\ref{eqn:rotrel}) prove the identity
    \begin{align*}
        \omega\circ\sigma_{V(z)\otimes V(z-\lambda)}(T_i^j(u))= \omega\circ \sum_k \R_i^k(u-z)\otimes \R_k^j(u-z+\lambda)=\rho(u-z)\omega,
    \end{align*}
    which together with
    \begin{align*}
        \omega\circ\sigma_{V(z)\otimes V(z-\lambda)}(\xi^\pm (u))= \rho(u-z)^{\pm 1}\rho(u-z+\lambda)^{\pm 1}\omega=\sigma_f(\xi^\pm (u))\omega
    \end{align*}
    shows that the map $\omega $ is an isomorphism $\colon V(z)\otimes V(z-\lambda)\to\mathbb C_f$ for $f(u)=\rho(u-z)$.
\end{enumerate}
\end{proof}
\subsection{A category of representations with elliptic coefficients}
Out of the vector representation we can construct a subcategory of the monoidal
category of modules over our elliptic quantum algebras $E_\ell(\mathfrak g^{(1)})$.
Namely we can take subquotients of
\[
  V(z_1)\otimes\cdots \otimes V(z_r)\otimes \mathbb C_f
\]
for any complex numbers $z_1,\dots, z_r$ and formal Fourier series $f$ converging to an elliptic function. In the $A_n$ case we can also replace
some factors $V(z_i)$ by $V^*(z_i)$. This gives a rigid monoidal category
with a forgetful functor to $\pi$-graded vector spaces, and for any pairs
$U,W$ of objects we have a morphism $\check R_{UW}(z)\colon
U(z)\otimes W\to W\otimes U(z)$ for generic $z$; these morphisms
are solutions of Yang--Baxter equations \cite{FelderRen2021}. 

A priori on these representation the generators $T(u)$ act by formal Fourier series \eqref{fps}. But
since these actions are polynomials in the matrix elements of generators $T(u)$ for
the vector representation, which are meromorphic functions, the power series converge to meromorphic functions
of $u\in\mathbb C$ and $\tau$ in the upper half-plane. Moreover they are
doubly periodic:

\begin{prop} Let $q=e^{\frac{\pi i}L}$ and $\tilde\pi$ the set of arrows $(a,\epsilon_i)$
  of $\pi$ appearing as degrees in the vector representation of
  $E_\ell(\mathfrak g^{(1)})$. Let $\chi\colon \tilde\pi\to \frac12\mathbb Z$ be
  the map $(a,\epsilon_i)\mapsto a_i$. Then
  the action of $T^i_j(u)$ and $\xi^{\pm}(u)$, $\xi_2^{\pm}(u)$ for 
  on $V(z_1)\otimes\cdots\otimes V(z_r)$ obeys
  \begin{align*}
    T^i_j(u+L)&=T^i_j(u), \quad T^i_j(u+L\tau)=q^{2(r+\chi(i)-\chi(j))}T^i_j(u),
    \\
    \xi^\pm(u+L)&=\xi^\pm(u),\quad \xi^\pm(u+L\tau)=q^{\pm4r}\xi^\pm(u),
                  \quad \text{for $\mathfrak g =B_n,C_n,D_n$,}
    \\
    \xi_2^\pm(u+L)&=\xi_2^\pm(u),\quad \xi_2^\pm(u+L\tau)=\xi_2^\pm(u),
    \quad \text{for $\mathfrak g=A_n$.}
  \end{align*}
  In particular $T^i_j(u)$, $\xi^\pm(u)$ are elliptic functions
  of $u$ with periods $L$ and $2L^2\tau$.
\end{prop}
\begin{proof}
  This is easily checked for $r=1$ using the transformation properties
  \[
    \frac{[x+L-a]}{[x+L-b]}=\frac{[x-a]}{[x-b]},\qquad
    \frac{[x+L\tau-a]}{[x+L\tau -b]}=e^{\frac{2\pi i}{L}(a-b)}\frac{[x+a]}{[x+b])}.
   \]
   The general case follows by induction in $r$.
 \end{proof}
 Conjecturally, these properties hold for all finite dimensional representations.
 
In certain cases, we can require the central elements to act by 1:
\begin{prop}
    Let $\mathfrak{g}^{(1)}=C_n^{(1)},D_n^{(1)}$. Let $L$ be a period of the function $v\mapsto[v]$ and $\lambda$ the crossing parameter of the algebra. If $\lambda$ and $L$ are rationally related, i.e. $\frac{-\lambda}{L}=\frac{p}{q}$ for $p,q\in\mathbb{N}$ such that $\operatorname{gcd}(p,q)=1$, and $q$ is odd, there exists a solution $\rho'(u)$ of the following equation
    \begin{equation}\label{rho}
\rho' (u)\rho'(u+\lambda)=\rho (u),
       \end{equation}
    where $\rho(u)$ is given in (\ref{eqn:rotrel}), and it is
\begin{equation*}
    \rho'(u)=\prod_{i=0}^{\frac{q-1}{2}}\sqrt{\frac{[u+2i\lambda-1][u+(2i+1)\lambda]}{[u+2i\lambda][u+(2i+1)\lambda+1]}}\prod_{j=0}^{\frac{q-1}{2}-1}\sqrt{\frac{[u+(2j+1)\lambda ][u+2(j+1)\lambda+1]}{[u+(2j+1)\lambda -1][u+2(j+1)\lambda]}}.
  \end{equation*}
\end{prop}
\begin{proof}
It is easy to check that 
\[
\rho'(u)\rho'(u+\lambda)=\sqrt{\frac{[u+q\lambda-1][u+(q+1)\lambda][u-1][u+\lambda]}{[u+q\lambda][u+(q+1)\lambda+1][u][u+\lambda+1]}}.
\]
We have the condition $q\lambda =-pL$ and we use the periodicity $[u-pL]=[u-p'L]=[u]$ to obtain the result.
\end{proof}
In these cases we can consider the restricted quantum group with central elements $\xi^\pm(u)=\eta(1)$. Its representation is given by the renormalization of the $R$-matrix,
$$\tilde \R(u):=\frac{1}{\rho' (u)}\R(u).$$

\appendix
\section{General results in the unrestricted model}\label{apA}
\begin{prop}[Jimbo, Miwa, Okado \cite{JimboMiwaOkado1988}]
The Boltzmann weights satisfy the following properties:
\begin{itemize}
\item Reflection symmetry:
\begin{equation}
\label{eqn:refsym}
\W\Big(\begin{matrix}
  a & b\\
  d & c
\end{matrix}\Big\vert u\Big)=
\W\Big(\begin{matrix}
  a & d\\
  b & c
\end{matrix}\Big\vert u\Big),
\end{equation}
\item Rotational symmetry (for $B_n^{(1)}$, $C_n^{(1)}$ and $D_n^{(1)}$):
\begin{equation}
\label{eqn:rotrel}
\W\Big(\begin{matrix}
  b & a\\
  g & c
\end{matrix}\Big\vert u\Big)=\rho(-u)^{-1}\sqrt{\frac{G_aG_g}{G_bG_c}}
\W\Big(\begin{matrix}
  a & c\\
  b & g
\end{matrix}\Big\vert \lambda-u\Big),
\end{equation}
with
\begin{align*}
    \rho(u)=\frac{[u-1][\lambda +u]}{[u][1+\lambda +u]}, 
\end{align*}
\item Inversion relation:
\begin{equation}
\label{eqn:invrel}
\sum_g 
\W\Big(\begin{matrix}
  b & g\\
  d & c
\end{matrix}\Big\vert u\Big)
\W\Big(\begin{matrix}
  b & a\\
  g & c
\end{matrix}\Big\vert -u\Big)=\delta_{a,d},
\end{equation}
\item Additional inversion relation for $A_{n-1}^{(1)}$:
\begin{equation}
    \label{eqn:invAn}
    \sum_g \sqrt{\frac{G_aG_cG_g^2}{G_b^2G_d^2}}
\W\Big(\begin{matrix}
  a & b\\
  d & g
\end{matrix}\Big\vert\lambda - u\Big)
\W\Big(\begin{matrix}
  c & d\\
  b & g
\end{matrix}\Big\vert\lambda +u\Big)=\delta_{a,c}\rho_2(u),
\end{equation}
with
\begin{align*}
    \rho_2 (u)=\frac{[\lambda +u][\lambda -u]}{[1+\lambda +u][1+\lambda -u]},
\end{align*}
\item Star-triangle relation:
\begin{equation}
\label{eqn:str-appendix}
\begin{split}
\sum_g \W\Big(\begin{matrix}
  f & g\\
  e & d
\end{matrix}\Big\vert u_2\Big)
\W\Big(\begin{matrix}
  b & c\\
  g & d
\end{matrix}\Big\vert u_1-u_2\Big)
\W\Big(\begin{matrix}
  a & b\\
  f & g
\end{matrix}\Big\vert u_1\Big)\\
=\sum_g \W\Big(\begin{matrix}
  a & b\\
  g & c
\end{matrix}\Big\vert u_2\Big)
\W\Big(\begin{matrix}
  a & g\\
  f & e
\end{matrix}\Big\vert u_1-u_2\Big)
\W\Big(\begin{matrix}
  g & c\\
  e & d
\end{matrix}\Big\vert u_1\Big).
\end{split}
\end{equation}
\end{itemize}
\end{prop}
\begin{cor}
The Boltzmann weights associated to the algebra $A_{n-1}^{(1)}$ satisfy the relation 
\begin{equation}
\label{eqn:invAn2}
\begin{split}
    \sum_g \sqrt{\frac{G_aG_cG_g^2}{G_b^2G_d^2}}
\W\Big(\begin{matrix}
  g & b\\
  d & a
\end{matrix}\Big\vert\lambda +u\Big)
\W\Big(\begin{matrix}
  g & d\\
  b & c
\end{matrix}\Big\vert\lambda -u\Big)=\delta_{a,c}\rho_2(u).
\end{split}
\end{equation}
\end{cor}
\begin{proof}
    We need to show
\begin{enumerate}
        \item
\begin{tikzcd}
g \arrow[r, "\varepsilon_j"] \arrow[d, "\varepsilon_i"']\arrow[rd,phantom,"\lambda+u"]
& b \arrow[d, "\varepsilon_i"] \\
d \arrow[r, "\varepsilon_j"']
& a 
\end{tikzcd}\begin{tikzcd}
g \arrow[r, "\varepsilon_i"] \arrow[d, "\varepsilon_j"']\arrow[rd,phantom,"\lambda -u"]
& d \arrow[d, "\varepsilon_j"] \\
b \arrow[r, "\varepsilon_i"']
& a 
\end{tikzcd}$\displaystyle\frac{G_gG_{a}}{G_bG_d}=\rho_2(u),$
\item 
\begin{tikzcd}
g \arrow[r, "\varepsilon_i"] \arrow[d, "\varepsilon_i"']\arrow[rd,phantom,"\lambda+u"]
& b \arrow[d, "\varepsilon_i"] \\
b \arrow[r, "\varepsilon_i"']
& a 
\end{tikzcd}\begin{tikzcd}
g \arrow[r, "\varepsilon_i"] \arrow[d, "\varepsilon_i"']\arrow[rd,phantom,"\lambda -u"]
& b \arrow[d, "\varepsilon_i"] \\
b \arrow[r, "\varepsilon_i"']
& a 
\end{tikzcd}$\displaystyle\frac{G_gG_a}{G_b^2}$\\
$+\sum_{j\neq i}$\begin{tikzcd}
g^\prime \arrow[r, "\varepsilon_j"] \arrow[d, "\varepsilon_j"']\arrow[rd,phantom,"\lambda+u"]
& b \arrow[d, "\varepsilon_i"] \\
b \arrow[r, "\varepsilon_i"']
& a 
\end{tikzcd}\begin{tikzcd}
g^\prime \arrow[r, "\varepsilon_j"] \arrow[d, "\varepsilon_j"']\arrow[rd,phantom,"\lambda -u"]
& b \arrow[d, "\varepsilon_i"] \\
b \arrow[r, "\varepsilon_i"']
& a 
\end{tikzcd}
$\displaystyle\frac{G_{g^\prime}G_{a}}{G_b^2}=\rho_2(u),$
\item 
\begin{tikzcd}
g \arrow[r, "\varepsilon_i"] \arrow[d, "\varepsilon_i"']\arrow[rd,phantom,"\lambda+u"]
& b \arrow[d, "\varepsilon_i"] \\
b \arrow[r, "\varepsilon_i"']
& a 
\end{tikzcd}\begin{tikzcd}
g \arrow[r, "\varepsilon_i"] \arrow[d, "\varepsilon_i"']\arrow[rd,phantom,"\lambda -u"]
& b \arrow[d, "\varepsilon_j"] \\
b \arrow[r, "\varepsilon_j"']
& c
\end{tikzcd}$\displaystyle\frac{G_g\sqrt{G_{a}G_c}}{G_b^2}$\\
$+$\begin{tikzcd}
g^\prime \arrow[r, "\varepsilon_j"] \arrow[d, "\varepsilon_j"']\arrow[rd,phantom,"\lambda+u"]
& b \arrow[d, "\varepsilon_i"] \\
b \arrow[r, "\varepsilon_i"']
& a 
\end{tikzcd}\begin{tikzcd}
g^\prime \arrow[r, "\varepsilon_j"] \arrow[d, "\varepsilon_j"']\arrow[rd,phantom,"\lambda -u"]
& b \arrow[d, "\varepsilon_j"] \\
b \arrow[r, "\varepsilon_j"']
& c 
\end{tikzcd}
$\displaystyle\frac{G_{g^\prime}\sqrt{G_{a}G_c}}{G_b^2}$\\
$+\sum_{k\neq i,j}$\begin{tikzcd}
g^{\prime\prime} \arrow[r, "\varepsilon_k"] \arrow[d, "\varepsilon_k"']\arrow[rd,phantom,"\lambda+u"]
& b \arrow[d, "\varepsilon_i"] \\
b \arrow[r, "\varepsilon_i"']
& a 
\end{tikzcd}\begin{tikzcd}
g^{\prime\prime} \arrow[r, "\varepsilon_k"] \arrow[d, "\varepsilon_k"']\arrow[rd,phantom,"\lambda -u"]
& b \arrow[d, "\varepsilon_j"] \\
b \arrow[r, "\varepsilon_j"']
& c 
\end{tikzcd}
$\displaystyle\frac{G_{g^{\prime\prime}}\sqrt{G_{a}G_c}}{G_b^2}=0.$
 \end{enumerate}
 The first relation follows directly from the definition and the observation that
 \begin{align*}
     \frac{G_gG_{g+\varepsilon_i+\varepsilon_j}}{G_{g+\varepsilon_i}G_{g+\varepsilon_j}}=\frac{[g_i-g_j]^2}{[g_i-g_j-1][g_i-g_j+1]}.
 \end{align*}
 The second and third relation come from the identity (\ref{eqn:invAn}), in particular from expressions
 \begin{align*}
     1+\sum_{i\neq j}\frac{[1]^2}{[\lambda +u+1][\lambda -u+1]}\frac{[\tilde a_i-\tilde a_j-\lambda -u][\tilde a_i-\tilde a_j-\lambda +u]}{[\tilde a_i-\tilde a_j]^2}\frac{G_{\tilde a}G_{\tilde a+\varepsilon_i+\varepsilon_j}}{G_{\tilde a+\varepsilon_i}G_{\tilde a+\varepsilon_j}}\\
     =\rho_2 (u)
 \end{align*}
 and
 \begin{align*}
     &\frac{[1]}{[1+\lambda +u]}\frac{[\tilde b_j-\tilde b_i -1 -\lambda -u]}{[\tilde b_j-\tilde b_i-1]}   G_{\tilde b+\varepsilon_i}+ \frac{[1]}{[1+\lambda -u]}\frac{[\tilde b_i-\tilde b_j -1 -\lambda +u]}{[\tilde b_i-\tilde b_j-1]}G_{\tilde b+\varepsilon_j}\\
     &+\sum_{k\neq i,j}\frac{[1]^2}{[1+\lambda +u][1+\lambda -u]}\frac{[\tilde b_i-\tilde b_k -1 -\lambda +u]}{[\tilde b_i-\tilde b_k-1]}\frac{[\tilde b_j-\tilde b_k -1 -\lambda -u]}{[\tilde b_j-\tilde b_k-1]}G_{\tilde b+\varepsilon_k}=0,
 \end{align*}
 with substitutions $\tilde a=-g-\varepsilon_i -\varepsilon_j$ and $\tilde b=-b$.
\end{proof}

\begin{cor}
The Boltzmann weights associated to the algebras $B_n^{(1)}$, $C_n^{(1)}$ and $D_n^{(1)}$ obey the identities
\begin{equation}
\label{eqn:rel1}
\sum_g \sqrt{\frac{G_{g}}{G_{c}}}
\W\Big(\begin{matrix}
  a & c\\
  b & g
\end{matrix}\Big\vert u+\lambda\Big)
\W\Big(\begin{matrix}
  b & g\\
  d & c
\end{matrix}\Big\vert u\Big)=\sqrt{\frac{G_b}{G_a}}\delta_{a,d}\rho(u),
\end{equation}
and
\begin{equation}
\label{eqn:rel2}
\sum_g \sqrt{\frac{G_{g}}{G_{b}}}
\W\Big(\begin{matrix}
  b & d\\
  g & c
\end{matrix}\Big\vert u\Big)
\W\Big(\begin{matrix}
  g & c\\
  b & a
\end{matrix}\Big\vert u+\lambda\Big)=\sqrt{\frac{G_c}{G_a}}\delta_{a,d}\rho(u).
\end{equation}
\end{cor}
\begin{proof}
The first relation is obtained by inserting the rotational symmetry relation (\ref{eqn:rotrel}) in the inversion relation (\ref{eqn:invrel}). The second relation comes from the first one by applying the reflection symmetry (\ref{eqn:refsym}) and the double rotation
\begin{align*}
\W\Big(\begin{matrix}
  a & b\\
  c & g
\end{matrix}\Big\vert u+\lambda\Big)=
\W\Big(\begin{matrix}
  g & c\\
  b & a
\end{matrix}\Big\vert u+\lambda\Big).
\end{align*}
\end{proof}

\section*{Acknowledgements}
This research was supported in part by the grant 196892 of the Swiss National Science Foundation and by the National Centre of Competence in Research SwissMAP (grant number 205607) of the Swiss National Science Foundation.

We thank Raschid Abedin, Tommaso Botta, Elli Pomoni and Muze Ren for discussions.

\printbibliography

\end{document}